\documentclass[11pt]{article}

\usepackage{amsmath,amssymb,amsfonts,amsthm,bbm}
\usepackage{color,graphicx}
\RequirePackage[OT1]{fontenc}
\RequirePackage[numbers]{natbib}
\RequirePackage[colorlinks,citecolor=blue,urlcolor=blue]{hyperref}
\hypersetup{linktocpage}

\numberwithin{equation}{section}
\graphicspath{ {./figures/} }

\newtheorem{theorem}{Theorem}[section]
\newtheorem{definition}{Definition}
\newtheorem{corollary}[theorem]{Corollary}
\newtheorem{lemma}[theorem]{Lemma}
\newtheorem{proposition}[theorem]{Proposition}
\newtheorem{condition}{Condition}
\newtheorem{remark}[theorem]{Remark}
\renewcommand{\P}{\mathbb{P}}

\newcommand{\E}{\mathbb{E}}
\newcommand{\R}{\mathbb{R}}

\newcommand{\supp}{\text{supp}}
\newcommand{\norm}[1]{\left| \left|#1\right| \right|}
\newcommand{\snorm}[1]{| | #1 | |}
\newenvironment{enumerate*}%

\title{Adaptive Bernstein--von Mises theorems in\\Gaussian white noise}
\author{Kolyan Ray\footnote{Mathematical Institute, Leiden University, P.O. Box 9512, 2300 RA Leiden, The Netherlands. E-mail: \href{mailto:k.m.ray@math.leidenuniv.nl}{k.m.ray@math.leidenuniv.nl}\newline \indent Most of this work was completed during the author's PhD at the University of Cambridge. This work was supported by UK Engineering and Physical Sciences Research Council (EPSRC) grant EP/H023348/1 and the European Research Council under ERC Grant Agreement 320637.}}

\date{}

\begin{document}

\maketitle

\begin{abstract}
We investigate Bernstein--von Mises theorems for adaptive nonparametric Bayesian procedures in the canonical Gaussian white noise model. We consider both a Hilbert space and multiscale setting with applications in $L^2$ and $L^\infty$ respectively. This provides a theoretical justification for plug-in procedures, for example the use of certain credible sets for sufficiently smooth linear functionals. We use this general approach to construct optimal frequentist confidence sets based on the posterior distribution. We also provide simulations to numerically illustrate our approach and obtain a visual representation of the geometries involved.  \\

\end{abstract}

\tableofcontents

\section{Introduction}\label{introduction}

A key aspect of statistical inference is uncertainty quantification and the Bayesian approach to this problem is to use the posterior distribution to generate a \emph{credible set}, that is a region of prescribed posterior probability (often 95\%). This can be considered an advantage of the Bayesian approach since Bayesian credible sets can be computed by simulation. In particular, the Bayesian generates a number of posterior draws and then keeps a prescribed fraction of the draws, discarding the remainder which are considered ``extreme" in some sense. From a frequentist perspective, key questions are whether such a method has a theoretical justification and what is an effective rule for determining which draws to discard. A natural approach is to characterize such draws using a geometric notion, in particular by considering a minimal ball in some metric.

In finite dimensions, the Euclidean distance has a clear interpretation as the natural measure of size. However in infinite dimensions such a notion is less clear-cut: the $L^2$ metric is the natural generalization of the Euclidean norm, but lacks a clear visual interpretation, while $L^\infty$ can be easily visualized but is more difficult to treat mathematically. From the Bayesian perspective of simulating credible sets, the practitioner ultimately seeks a practical and effective rule for sorting through posterior draws and such geometric interpretations can be viewed as somewhat artificial impositions. The aim of this article is therefore to study possible geometric choices of credible sets that behave well from a frequentist asymptotic perspective.

Consider data $Y^{(n)}$ arising from some probability distribution $\P_{f}^{(n)}$, $f \in \mathcal{F}$. We place a prior distribution $\Pi$ on $\mathcal{F}$ and study the behaviour of the posterior distribution $\Pi (\cdot \mid Y^{(n)} )$ under the frequentist assumption $Y^{(n)} \sim \P_{f_0}^{(n)}$ for some non-random true $f_0 \in \mathcal{F}$ as the data size or quality $n \rightarrow \infty$. From such a viewpoint, the theoretical justification for posterior based inference using any (Borel) credible set in finite dimensions is provided by the Bernstein--von Mises (BvM) theorem (see \cite{Le,VV}). This deep result establishes mild conditions on the prior under which the posterior is approximately a normal distribution centered at an efficient estimator of the true parameter. It thus provides a powerful tool to study the asymptotic behaviour of Bayesian procedures and justifies the use of Bayesian simulations for uncertainty quantification.

A BvM in infinite-dimensions fails to hold in even very simple cases. Freedman \cite{Fr} showed that in the basic conjugate $\ell_2$ sequence space setting with both Gaussian priors and data, the BvM does not hold for $\ell_2$-balls centered at the posterior mean -- see also the related contributions \cite{Co,Jo,Lea}. The resulting message is that despite their intuitive interpretation, credible sets based on posterior draws using an $\ell_2$-based selection procedure do not behave as in classical parametric models. Recently, Castillo and Nickl \cite{CaNi,CaNi2} have established fully infinite-dimensional BvMs by considering weaker topologies than the classical $L^p$ spaces. Their focus lies on considering spaces which admit $1/\sqrt{n}$-consistent estimators and where Gaussian limits are possible, unlike $L^p$-type loss. Credible regions selected using these different geometries are shown to behave well, generating asymptotically exact frequentist confidence sets. In this paper, we explore this approach in practice via both theoretical results for \textit{adaptive} priors, as well as by numerical simulations. We consider an empirical Bayes, a hierarchical Bayes and a multiscale Bayes approach.

Before going into more abstract detail, it is useful to consider an example from \cite{CaNi2} to numerically illustrate this approach in practice. Suppose that we observe $Y_1,...,Y_n$ i.i.d. observations from an unknown density $f_0$ on $[0,1]$. We take a simple histogram prior $\Pi$,
\begin{equation}
f = 2^L \sum_{k=0}^{2^L - 1} h_k 1_{I_{Lk}} ,\quad \quad   I_{Lk} = (k2^{-L}, (k+1)2^{-L} ] , \quad k \geq 0 ,
\label{Dirichlet prior}
\end{equation}
where the $h_k$ are drawn from a $\mathcal{D}(1,...,1)$-Dirichlet distribution on the unit simplex in $\R^{2^L}$. Here we ignore adaptation issues and select $L = L_n$ based on the smoothness of the true function. Letting $\{ \psi_{lk} : l \geq 0, \, \, k = 0,...,2^l-1 \}$ denote the standard Haar wavelets and $w_l = l^{1/2 + \epsilon}$ for $\epsilon>0$ small, consider the multiscale credible ball
\begin{equation}
C_n = \left\{ f : \max_{k,l \leq L_n} w_l^{-1}  |\langle f  - \hat{f}_n , \psi_{lk} \rangle | \leq R_n n^{-1/2} \right\}  ,
\label{Dirichlet credible set}
\end{equation}
where $\hat{f}_n$ denotes the posterior mean and $R_n = R(Y_1,...,Y_n)$ is chosen such that $\Pi (C_n \mid Y_1,...,Y_n) = 0.95$. By Proposition 1 of \cite{CaNi2}, $\P_{f_0} (f_0 \in C_n) \rightarrow 0.95$ as $n \rightarrow \infty$, whereas no such result is available for the $L^\infty$-credible ball. Due to the conjugacy of the Dirichlet distribution with multinomial sampling, the posterior distribution can be computed straightforwardly and $R_n$ can be easily obtained by simulation.

For convenience we take $f_0$ to be a Laplace distribution with location parameter $1/2$ and scale parameter 5 that is truncated to $[0,1]$, that is $f_0 (x) \propto e^{-5|x - 1/2|} 1_{[0,1]}(x)$ with $f_0 \in H_2^{s} ([0,1))$ for $s < 3/2$. In Figure \ref{Dirichlet}, we plotted the true density (solid black) and the posterior mean (red) in the cases $n = 1000,2000,5000,10000$. We generated 100,000 posterior draws and plotted the 95\% closest to the posterior mean in the $\mathcal{M}(w)$ sense (grey) to simulate $C_n$. We also used the posterior draws to generate a 95\% credible band in $L^\infty$ by estimating $Q_n$ satisfying $\Pi ( f : \snorm{f - \hat{f}_n}_\infty \leq Q_n \mid Y) = 0.95$ and then plotting $\hat{f_n} \pm Q_n$ (dashed black).

\begin{figure}[h]
\centering
\includegraphics[scale=0.25]{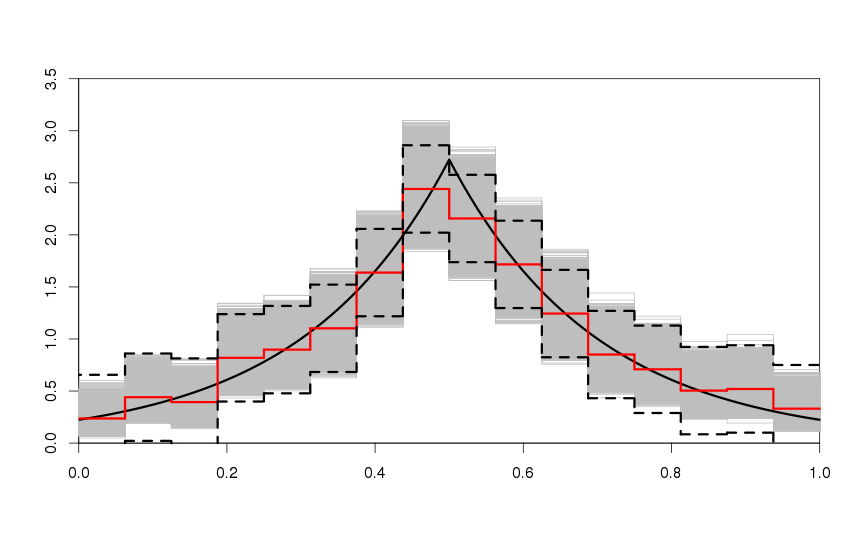}
\includegraphics[scale=0.25]{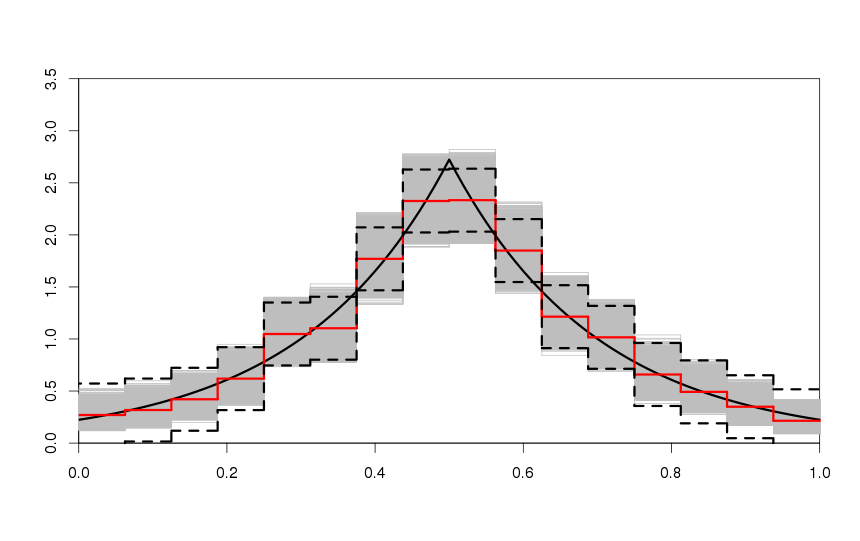}
\includegraphics[scale=0.25]{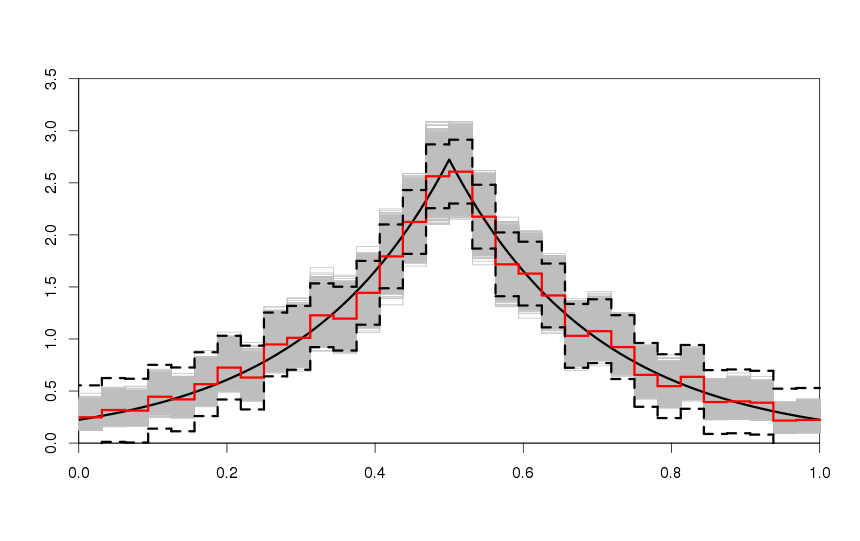}
\includegraphics[scale=0.25]{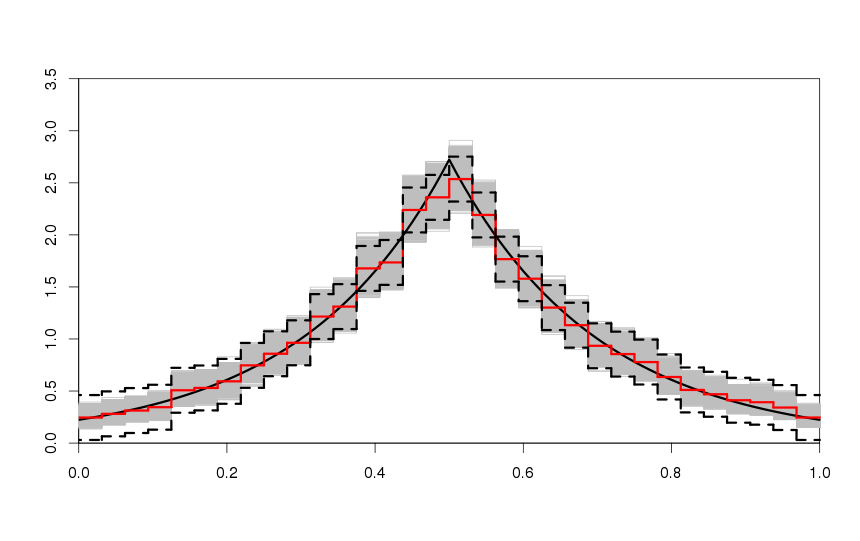}

\caption{\emph{Credible sets based on the Dirichlet prior with the true density function (solid black), the posterior mean (red), a 95\% credible band in $L^\infty$ (dashed black) and the set $C_n$ given in \eqref{Dirichlet credible set} (grey). We have $n = 1000,2000,5000$ and $10000$ respectively.}}
\label{Dirichlet}
\end{figure}

We see that the $L^\infty$ diameter of $C_n$ is strictly greater than that of the $L^\infty$-credible band, with this difference particularly marked at the peak of the density. However, the diameter of $C_n$ is spatially heterogeneous and has greatest width at the peak, whilst having smaller width around points where the true density is more regular. In all cases, $C_n$ contains the true $f_0$, whereas the $L^\infty$ confidence band has more difficulty capturing the peak.

The main message of this numerical example is that simulating the credible set $C_n$, which uses a slightly different geometry, yields a set that does not look particularly strange in practice and in fact resembles an $L^\infty$ credible band. Both approaches are methodologically similar, the only difference being the rule for discarding posterior draws. From a theoretical point of view, the difference between the two sets is far more significant, with $C_n$ yielding exact coverage statements at the expense of unbounded $L^\infty$ diameter. It is however possible to improve upon the naive implementation of such sets to also obtain the optimal $L^\infty$ diameter (see Proposition 1 of \cite{CaNi2} and related results below). Modifying the geometry in such a way to obtain an exact coverage statement therefore comes at little additional cost from a practitioner's perspective.

Nonparametric priors typically contain tuning or hyper parameters, and it is a key challenge to study procedures that select these parameters automatically in a data-driven manner. This avoids the need to make unreasonably strong prior assumptions about the unknown parameter of interest, since incorrect calibration of the prior can lead to suboptimal performance (see e.g. \cite{KnVVVZ}). It therefore makes sense to use an automatic procedure, unless a practitioner is particularly confident that their prior correctly captures the fine details of the unknown parameter, such as its level of smoothness or regularity. Adaptive procedures are widely used in practice, with hyper parameters commonly selected using a hyperprior or an empirical Bayes method. In the case of Gaussian white noise, a number of Bayesian procedures have been shown to be rate adaptive over common smoothness classes (see for example \cite{HoRoSc,KnSzVVVZ,Ra}). Most such frequentist analyses restrict attention to obtaining contraction rates and do not study coverage properties of credible sets. The focus of this paper is therefore to investigate nonparametric BvMs for adaptive priors, with the goal of studying the coverage properties of credible sets.

In the case of Gaussian white noise, there has been recent work \cite{KnVVVZ,Lea} circumventing the need for a BvM by explicitly studying the coverage properties of certain specific credible sets. Of particular relevance is a nice recent paper by Szab\'o et al. \cite{SzVVVZ}, where the authors use an empirical Bayes approach combined with scaling up the radius of $\ell_2$-balls to obtain adaptive confidence sets under a so-called \emph{polished tail condition}. Their approach relies on explicit prior computations and provides an alternative to the more abstract point of view taken here. One of our principal goals is exact coverage statements and this seems more difficult to obtain using such an explicit approach. Since adaptive confidence sets do not exist in full generality, we also require self-similarity conditions on the true parameter to exclude certain ``difficult" functions \cite{GiNi},\cite{HoNi},\cite{Bu}. In particular, we shall consider the procedure of \cite{SzVVVZ} in Section \ref{l2 setting} and obtain exact coverage statements under the self-similarity condition introduced there.

We note other work dealing with BvM results in the nonparametric setting. Leahu \cite{Lea} has studied the impact of prior smoothness on the existence of BvM theorems in the conjugate Gaussian sequence space model. Bickel and Kleijn \cite{BiKl}, Castillo \cite{Ca2}, Rivoirard and Rousseau \cite{RiRo} and Castillo and Rousseau \cite{CaRo}  provide sufficient conditions for BvMs for semiparametric  functionals. For the case of finite-dimensional posteriors with increasing dimension, see Ghosal \cite{Gh} and Bontemps \cite{Bo} for the case of regression or Boucheron and Gassiat \cite{BoGa} for discrete probability distributions.

Much of the approach taken here can equally be applied to other statistical settings such as sparsity and inverse problems \cite{Ra2}, but we restrict to the nonparametric regime for ease of exposition. Since our focus lies on BvM results and coverage statements and this changes little conceptually, we omit such generalizations to maintain mathematical clarity.

\section{Statistical setting}

\subsection{Function spaces and the white noise model}

We use the usual notation $L^p = L^p ([0,1])$ for $p$-times Lebesgue integrable functions and denote by $\ell_p$ the usual sequence spaces. We consider the canonical white noise model, which is equivalent to the fixed design Gaussian regression model with known variance. For $f \in L^2 = L^2([0,1])$, consider observing the trajectory
\begin{equation}
dY_t^{(n)} = f(t) dt + \frac{1}{\sqrt{n}} dB_t, \quad \quad t \in [0,1],
\label{white noise}
\end{equation}
where $dB$ is a standard white noise. By considering the action of an orthonormal basis $\{ e_\lambda \}_{\lambda \in \Lambda}$ on \eqref{white noise}, it is statistically equivalent to consider the Gaussian sequence space model
\begin{equation}
Y^{(n)}_\lambda \equiv  Y_\lambda =  f_\lambda + \frac{1} {\sqrt{n} } Z_\lambda , \quad \quad \quad \lambda \in \Lambda ,
\label{model}
\end{equation}
where the $( Z_\lambda )_{\lambda \in \Lambda}$ are i.i.d. standard normal random variables and the unknown parameter of interest $f = (f_\lambda)_{\lambda \in \Lambda}$ is assumed to be in $\ell_2$. We denote by $\P_{f_0}$ or $\P_0$ the law of $Y$ arising from \eqref{model} under the true function $f_0$. In the following, $\Lambda$ will represent either a Fourier-type basis or a wavelet basis. In the $\ell_2$-setting, \eqref{model} can be interpreted purely in sequence form with $\Lambda = \mathbb{N} $ and we do not need to associate to it a time index $t \in [0,1]$.

In $L^\infty$ we consider a multiscale approach so that $\Lambda = \{ (j,k): j \geq 0 , k=0,...,2^j -1 \}$. In particular, we consider an $S$-regular ($S \geq 0$) wavelet basis of $L^2 ([0,1])$, $\{ \psi_{lk} : l \geq J_0 - 1, \, \, k = 0,...,2^l-1 \}$, with $J_0 \in \mathbb{N}$. For notational simplicity, denote the scaling function $\phi$ by the first wavelet $\psi_{ (J_0-1) 0}$. We consider either periodized wavelets or boundary corrected wavelets (see e.g. \cite{Me} for more details). Moreover, in certain applications we require in addition that the wavelets satisfy a localization property
\begin{equation}
\sup_{x \in [0,1]} \sum_{k=0}^{2^{J_0}-1} | \phi_{{J_0}k}(x)| \leq c(\phi) 2^{{J_0}/2} <  \infty ,  \quad \quad     \sup_{x \in [0,1]} \sum_{k=0}^{2^j-1} | \psi_{jk} (x)| \leq c(\psi) 2^{j/2} < \infty ,
\label{wavelet}
\end{equation}
$j \geq J_0$ (see Section \ref{wavelet section} for more discussion). The sequence model \eqref{model} corresponds to estimating the wavelet coefficients $f_{lk} = \langle f , \psi_{lk} \rangle $, for all $(l,k) \in \Lambda$, since any function $f \in L^2$ generates such a wavelet sequence. Conversely, any such sequence $(f_{lk})$ generates the wavelet series of a function (or distribution if the sequence is not in $\ell_2$) $\sum_{(l,k)} f_{lk} \psi_{lk}$.

For $s, \delta \geq 0$, define the Sobolev spaces at the logarithmic level:
\begin{equation*}
H^{s,\delta} \equiv H_2^{s,\delta} := \left\{ f \in \ell_2 : \norm{f}_{s,2,\delta}^2  : = \sum_{k=1}^\infty k^{2s} (\log k)^{-2\delta} | f_k |^2 < \infty  \right\}  .
\end{equation*}
From this we recover the usual definition of the Sobolev spaces $H^s \equiv H_2^s = H_2^{s,0}$ and by duality we define for $s > 0$, $H_2^{-s} := ( H_2^s)^*$. By standard Hilbert space duality arguments, we can consider $\ell_2$ as a subspace of $H_2^{-s}$ and can similarly define the logarithmic spaces for $s < 0$ and $\delta \geq 0$ using the above series definition. In the $\ell_2$-setting we shall classify smoothness via the Sobolev \emph{hyper rectangles} for $\beta \geq 0$:
\begin{equation*}
\mathcal{Q} (\beta , R) = \left\{  f \in \ell_2 : \sup_{k \geq 1} k^{2\beta + 1} f_k^2 \leq R \right\} .
\end{equation*}

In the $L^\infty ([0,1])$-setting we consider multiscale spaces: for a monotone increasing sequence $w = (w_l)_{l \geq 1}$ with $w_l \geq 1$, define
\begin{equation*}
\mathcal{M} = \mathcal{M} (w) = \left\{ x = (x_{lk} ) : \norm{x}_{\mathcal{M}(w)} := \sup_{l \geq 0} \frac{1}{w_l} \max_k |x_{lk}| < \infty  \right\}  
\end{equation*}
(for further references to multiscale statistics see \cite{CaNi2}). A separable closed subspace is obtained by considering the restriction
\begin{equation*}
\mathcal{M}_0  = \mathcal{M}_0 (w) = \left\{ x \in \mathcal{M}(w) : \lim_{l \rightarrow \infty} \frac{1}{w_l} \max_k | x_{lk} | = 0   \right\} ,
\end{equation*}
that is those (weighted) sequences in $\mathcal{M} (w)$ that converge to 0. Note that $\mathcal{M}$ contains the space $\ell_2$, since $\norm{x}_\mathcal{M} \leq \norm{x}_{\ell_2}$ as $w_l \geq 1$. In this setting, we consider norm-balls in the Besov spaces $B_{\infty\infty}^\beta ([0,1])$,
\begin{equation*}
\mathcal{H} (\beta, R) = \{ f = (f_{lk})_{(l,k) \in \Lambda} : |f_{lk}| \leq R 2^{-l(\beta + 1/2)} , \, \, \forall (l,k) \in \Lambda  \} .
\end{equation*}
We recall that $B_{\infty\infty}^\beta ([0,1]) = C^\beta ([0,1])$, the classical H\"older (-Zygmund in the case $\beta \in \mathbb{N}$) spaces. For more details on these embeddings and identifications see \cite{Me}.

Whether an $\ell_2$-white noise defines a tight random element of $\mathcal{M}_0 (w)$ depends on the weighting sequence $(w_l)$. Recall that we call a sequence $( w_l )_{l \geq 1}$ \emph{admissible} if $w_l / \sqrt{l} \nearrow \infty$ as $l \rightarrow \infty$ \cite{CaNi2}. Let $Z = \{ Z_\lambda = \langle Z , e_\lambda \rangle : \lambda \in \Lambda \}$, where $Z_\lambda \sim N(0,1)$ i.i.d., denote the Gaussian white noise in \eqref{model}. We have from \cite{CaNi,CaNi2} that for $\delta > 1/2$ and $(w_l)$ an admissible sequence, $Z$ defines a tight Gaussian Borel random variable on $H_2^{-1/2,\delta}$ and $\mathcal{M}_0 (w)$ respectively, which we denote $\mathbb{Z}$. In view of this tightness, we can consider \eqref{white noise} as a Gaussian shift model:
\begin{equation*}
\mathbb{Y}^{(n)} = f + \frac{1}{\sqrt{n}} \mathbb{Z} ,
\end{equation*}
where the above inequality is in the $H_2^{-1/2,\delta}$- or $\mathcal{M}_0(w)$-sense. Since $\sqrt{n} ( \mathbb{Y}^{(n)} - f ) = \mathbb{Z}$ in $H_2^{-1/2,\delta}$ or $\mathcal{M}_0(w)$, it immediately follows that $\mathbb{Y}^{(n)}$ is an efficient estimator of $f$ in either norm.

Among the two classes $\{ H_2^{s,\delta} \}_{s \in \R, \delta \geq 0}$ and $\{ \mathcal{M}_0 (w) \}_w$ of spaces considered, one can show that $s = -1/2$, $\delta > 1/2$ and admissibility of $w$ determine the minimal spaces where the law of the $\ell_2$-white noise $Z$ is tight (see \cite{CaNi,CaNi2} for further discussion). We therefore focus attention on these spaces since they provide the threshold for which a weak convergence approach can work. For convenience, we denote $H \equiv H(\delta) \equiv H_2^{-1/2,\delta}$. We further denote the law of $\mathbb{Z}$ in $H$ or $\mathcal{M}_0 (w)$ by $\mathcal{N}$ as appropriate.

\subsection{Weak Bernstein--von Mises phenomena}

Due to the continuous embeddings $\ell_2 \subset H$ and $\ell_2 \subset \mathcal{M}_0 (w)$, any Borel probability measure on $\ell_2$ yields a tight Borel probability measure on $H$ and $\mathcal{M}_0(w)$. Consider a prior $\Pi$ on $\ell_2$ and let $\Pi_n = \Pi ( \cdot \mid Y^{(n)} )$ denote the posterior distribution based on data \eqref{model}. For $S$ a vector space and $z \in S$, consider the map $\tau_z : S \rightarrow S$ given by
\begin{equation*}
\tau_z : f \mapsto \sqrt{n} (f-z) .
\end{equation*}
Let $\Pi_n \circ \tau_{\mathbb{Y}^{(n)}}^{-1}$ denote the image measure of the posterior distribution (considered as a measure on $H$ or $\mathcal{M}_0(w)$) under the map $\tau_{\mathbb{Y}^{(n)}}$. Thus for any Borel set $B$ arising from these topologies,
\begin{equation*}
\Pi_n \circ \tau_{\mathbb{Y}^{(n)}}^{-1} (B) = \Pi ( \sqrt{n} ( f - \mathbb{Y}^{(n)} ) \in B \mid Y^{(n)})  ,
\end{equation*}
so that we can more intuitively write $\Pi_n \circ \tau_{\mathbb{Y}^{(n)}}^{-1} = \mathcal{L} ( \sqrt{n} (f - \mathbb{Y}^{(n)}) \mid Y^{(n)} )$, where $\mathcal{L} (f \mid Y^{(n)})$ denotes the law of $f$ under the posterior. For convenience, we metrize the weak convergence of probability measures via the bounded Lipschitz metric (defined in Section \ref{weak convergence section}). Recalling that we denote by $\mathcal{N}$ the law of the white noise $Z$ in \eqref{model} as an element of $S$, we define the notion of nonparametric BvM.

\begin{definition}\label{weak BvM}
Consider data generated from \eqref{model} under a fixed function $f_0$ and denote by $\P_0$ the distribution of $Y^{(n)}$. Let $\beta_S$ be the bounded Lipschitz metric for weak convergence of probability measures on $S$. We say that a prior $\Pi$ satisfies a weak Bernstein-von Mises phenomenon in $S$  if, as $n \rightarrow \infty$,
\begin{equation*}
\E_0 \beta_S ( \Pi_n \circ \tau_{\mathbb{Y}^{(n)} }^{-1} , \mathcal{N} ) = \E_0 \beta_S (\mathcal{L} ( \sqrt{n} (f - \mathbb{Y}^{(n)} ) \mid Y^{(n)} ) , \mathcal{N}) \rightarrow 0  .
\end{equation*}
Here $S$ is taken to be one of $H(\delta)$ for $\delta > 1/2$, $H^{-s}$ for $s > 1/2$ or $\mathcal{M}_0 (w)$ for $( w_l )_{l \geq 1}$ an admissible sequence.
\end{definition}

The weak BvM says that the (scaled and centered) posterior distribution asymptotically looks like an infinite-dimensional Gaussian distribution in some `weak' sense, quantified via the bounded Lipschitz metric \eqref{bounded Lipschitz metric}. Weak convergence in $S$ implies that these two probability measures are approximately equal on certain classes of sets, whose boundaries behave smoothly with respect to the measure $\mathcal{N}$ (see Sections 1.1 and 4.1 of \cite{CaNi}).

\subsection{Self-similarity}

The study of adaptive BvM results naturally leads to the topic of adaptive frequentist confidence sets. It is known that confidence sets with radius of optimal order over a class of submodels nested by regularity that also possess honest coverage do not exist in full generality (see \cite{HoNi,NiSz} for recent references). We therefore require additional assumptions on the parameters to be estimated and so consider self-similar functions, whose regularity is similar at both small and large scales. Such conditions have been considered in Gin\'e and Nickl \cite{GiNi}, Hoffmann and Nickl \cite{HoNi} and Bull \cite{Bu} and ensure that we remove those functions whose norms (measuring smoothness) are difficult to estimate and which statistically look smoother than they actually are. We firstly consider the $\ell_2$-type self-similarly assumption found in Szab\'o et al. \cite{SzVVVZ}.
\begin{definition}\label{self condition}
Fix an integer $N_0 \geq 2$ and parameters $\rho > 1$, $\varepsilon \in (0,1)$. We say that a function $f \in \mathcal{Q} ( \beta , R)$ is \emph{self-similar} if
\begin{equation*}
\sum_{k=N}^{ \lceil \rho N \rceil } f_k^2 \geq \varepsilon R N^{-2\beta} \quad \quad \text{ for all } \quad N \geq N_0 .
\end{equation*}
We denote the class of self-similar elements of $\mathcal{Q}(\beta,R)$ by $\mathcal{Q}_{SS}(\beta,R,\varepsilon)$.
\end{definition}
This condition says that each block $(f_N,...,f_{\lceil \rho N \rceil })$ of consecutive components contains at least a fixed fraction (in the $\ell_2$-sense) of the size of a ``typical" element of $\mathcal{Q}(\beta,R)$, so that the signal looks similar at all frequency levels (see \cite{SzVVVZ,Ni,NiSz} for further discussion). The parameters $N_0$ and $\rho$ affect the results of this article through the sample size at which the asymptotic results take effect, that is $n \rightarrow \infty$ implicitly implies statements of the form ``for $n\geq n_0$ large enough", where $n_0$ depends on $N_0$ and $\rho$. For this reason, the impact of $N_0$ and $\rho$ is not explicitly mentioned below and one may simply treat these constants as fixed (e.g. $N_0=2$ and $\rho=2$). The lower bound in Definition \ref{self condition} can be slightly weakened to permit for example logarithmic deviations from $N^{-2\beta}$. However since this results in additional technicality whilst adding little extra insight, we do not pursue such a generalization here. It is possible to consider a weaker self-similarity condition using a strictly frequentist approach \cite{NiSz}, though this has not been explored in the Bayesian setting and it is unclear whether our approach extends in such a way. Let $K_j (f) = \sum_k \langle f , \phi_{jk} \rangle \phi_{jk}$ denote the wavelet projection at resolution level $j$. In $L^\infty$ we consider Condition 3 of Gin\'e and Nickl \cite{GiNi}, which can only be slightly relaxed \cite{Bu}.
\begin{definition}\label{self condition 2}
Fix a positive integer $j_0$. We say that a function $f \in \mathcal{H} (\beta,R)$ is \emph{self-similar} if there exists a constant $\varepsilon>0$ such that
\begin{equation*}
\norm{K_j (f) - f}_{\infty} \geq \varepsilon 2^{-j\beta}    \quad \quad \text{ for all } \quad j \geq j_0 .
\end{equation*}
We denote the class of self-similar elements of $\mathcal{H}(\beta,R)$ by $\mathcal{H}_{SS}(\beta,R,\varepsilon)$.
\end{definition}
In particular, since $f \in \mathcal{H} (\beta,R)$, we have that $\norm{K_j (f) - f}_{\infty} \asymp 2^{-j\beta}$ for all $j \geq j_0$. What we really require is that there is at least one significant coefficient at the level $\log_2 ( (n/\log n)^{1/(2\beta + 1)})$ that the posterior distribution can detect. However, this level depends also on unknown constants in practice (see proof of Proposition \ref{credible band}) and so we require a statement for all (sufficiently large) resolution levels as in Definition \ref{self condition 2}. See Gin\'e and Nickl \cite{GiNi} and also Bull \cite{Bu} for further discussion about this condition.

\section{Bernstein--von Mises results}

\subsection{Empirical and hierarchical Bayes in $\ell_2$}\label{l2 setting}

We continue the frequentist analysis of the adaptive priors studied in \cite{KnSzVVVZ,SzVVVZ,SzVVVZ3} in $\ell_2$. For $\alpha >0$ define the product prior on the $\ell_2$-coordinates by the product measure
\begin{equation*}
\Pi_\alpha = \bigotimes_{k=1}^\infty N (0,k^{-2\alpha -1} ), 
\end{equation*}
so that the coordinates are independent. A draw from this distribution will be $\Pi_\alpha$-almost surely in all Sobolev spaces $H_2^{\alpha'}$ for $\alpha ' < \alpha$. The posterior distribution corresponding to $\Pi_\alpha$ is given by
\begin{equation}
\Pi_\alpha (\cdot \mid Y ) = \bigotimes_{k=1}^\infty N \left( \frac{n}{k^{2\alpha + 1} + n } Y_k ,  \frac{1}{k^{2\alpha + 1} + n} \right)  .
\label{eq28}
\end{equation}

If $f_0 \in H^\beta$ and $\alpha = \beta$, it has been shown \cite{BeGo,Ca4,KnVVVZ} that the posterior contracts at the minimax rate of convergence, while if $\alpha \neq \beta$, then strictly suboptimal rates are achieved. Since the true smoothness $\beta$ is generally unknown, two data-driven procedures have been considered in \cite{KnSzVVVZ}. The empirical Bayes procedure consists of selecting the smoothness parameter by using a likelihood-based approach. Namely, we consider the estimate
\begin{equation}
\hat{\alpha}_n =  \underset{ \alpha \in [ 0 , a_n] }{\text{argmax}} \: \:   \ell_n (\alpha) ,
\label{likelihood}
\end{equation}
where $a_n \rightarrow \infty$ is any sequence such that $a_n = o (\log n)$ as $n \rightarrow \infty$ and
\begin{equation*}
\ell_n (\alpha) = - \frac{1}{2} \sum_{k=1}^\infty \left( \log \left( 1 + \frac{n}{k^{2\alpha + 1}   }  \right)  -  \frac{n^2}{k^{2\alpha + 1}  +n } Y_k^2 \right)
\end{equation*}
is the marginal log-likelihood for $\alpha$ in the joint model $(f,Y)$ in the Bayesian setting (relative to the infinite product measure $\otimes_{k=1}^\infty N(0,1)$). The quantity $a_n$ is needed to uniformly control the finite dimensional projections of the empirical Bayes procedure to establish a parametric BvM (Theorem \ref{EB parametric BvM theorem}). The posterior distribution is defined via the plug-in procedure
\begin{equation*}
\Pi_{\hat{\alpha}_n} (\cdot \mid Y )  =  \Pi_\alpha (\cdot \mid Y) \left. \right|_{\alpha = \hat{\alpha}_n }  .
\end{equation*}
If there exist multiple maxima to \eqref{likelihood}, then any of them can be selected.

A fully Bayesian approach is to put a hyperprior on the parameter $\alpha$. This yields the hierarchical prior distribution 
\begin{equation*}
\Pi^H = \int_0^\infty \lambda (\alpha) \Pi_\alpha d\alpha  ,
\end{equation*}
where $\lambda$ is a positive Lebesgue density on $(0,\infty)$ satisfying the following assumption (Assumption 2.4 of \cite{KnSzVVVZ}).

\begin{condition}\label{HB density condition}
Assume that for every $c_1 >0$, there exists $c_2 \geq 0, c_3 \in \R$, with $c_3 > 1$ if $c_2 =0$ and $c_4 >0$ such that for $\alpha \geq c_1$,
\begin{equation*}
c_4^{-1} \alpha^{-c_3}  \exp \left( -c_2 \alpha \right) \leq \lambda (\alpha) \leq c_4 \alpha^{-c_3}  \exp \left( -c_2 \alpha \right)  .
\end{equation*}
\end{condition}

The exponential, gamma and inverse gamma distributions satisfy Condition \ref{HB density condition} for example. Knapik et al. \cite{KnSzVVVZ} showed that both these procedures contract to the true parameter adaptively at the (almost) minimax rate, uniformly over Sobolev balls, and a similar result holds for Sobolev hyper rectangles. Both procedures satisfy weak BvMs in the sense of Definition \ref{weak BvM}.

\begin{theorem}\label{EB weak BvM}
Consider the empirical Bayes procedure described above. For every $\beta,R > 0$ and $s > 1/2$, we have
\begin{equation*}
\sup_{f_0 \in \mathcal{Q}(\beta,R) } \E_0 \beta_{H^{-s}} ( \Pi_{\hat{\alpha}_n} \circ \tau_{\mathbb{Y}}^{-1} , \mathcal{N} ) \rightarrow 0
\end{equation*}
as $n \rightarrow \infty$. Moreover, for $\delta > 2$ we have the (slightly) stronger convergence
\begin{equation*}
\sup_{f_0 \in \mathcal{Q}_{SS} (\beta,R,\varepsilon) } \E_0 \beta_{ H(\delta) } ( \Pi_{\hat{\alpha}_n} \circ \tau_{\mathbb{Y}}^{-1} , \mathcal{N} ) \rightarrow 0
\end{equation*}
as $n \rightarrow \infty$.
\end{theorem}

\begin{theorem}\label{HB weak BvM}
Consider the hierarchical Bayes procedure described above, where the prior density $\lambda$ satisfies Condition \ref{HB density condition}. For every $\beta,R > 0$ and $s > 1/2$, we have
\begin{equation*}
\sup_{f_0 \in \mathcal{Q}(\beta,R) } \E_0 \beta_{H^{-s}} ( \Pi_n^H \circ \tau_{\mathbb{Y}}^{-1} , \mathcal{N} ) \rightarrow 0
\end{equation*}
as $n \rightarrow \infty$. Moreover, for $\delta > 2$ we have the (slightly) stronger convergence
\begin{equation*}
\sup_{f_0 \in \mathcal{Q}_{SS} (\beta,R,\varepsilon) }  \E_0 \beta_{ H(\delta) } ( \Pi_n^H \circ \tau_{\mathbb{Y}}^{-1} , \mathcal{N} ) \rightarrow  0
\end{equation*}
as $n \rightarrow \infty$.
\end{theorem}

The requirement of self-similarity for a weak BvM in $H(\delta)$ could conceivably be relaxed, but such an assumption is natural since it is anyway needed for the construction of adaptive confidence sets in Section \ref{credible set section}. It is not clear whether this is a fundamental limit or a technical artefact of the proof. The condition $\delta > 2$ is also required for technical reasons.

Whilst minimax optimality is clearly desirable from a theoretical frequentist perspective, it may be too stringent a goal in our context. Using a purely Bayesian point of view, we derive an analogous result to Doob's almost sure consistency result. Specifically, a weak BvM holds in $H(\delta)$ for prior draws, almost surely under both the empirical Bayes and hierarchical priors. For this, it is sufficient to show that prior draws are self-similar almost surely.

\begin{proposition}\label{Doob result}
The parameter $f_0$ is self-similar in the sense of Definition \ref{self condition}, $\Pi_\alpha$-almost-surely for any $\alpha > 0$. Consequently, $\Pi_{\hat{\alpha}_n}$ and $\Pi^H$ satisfy a weak BvM in $H(\delta)$ for $\delta > 2$, $\Pi_\alpha$-a.s., $\alpha >0$, and $\Pi^H$-a.s. respectively.
\end{proposition}

In particular, $f$ satisfies Definition \ref{self condition} with smoothness $\alpha$ and parameters $\rho > 1$ and $\varepsilon = \varepsilon (\alpha,\rho,R) >0$ sufficiently small and random $N_0$ sufficiently large, $\Pi_\alpha$-almost surely. As a simple corollary to Theorems \ref{EB weak BvM} and \ref{HB weak BvM}, we have that the rescaled posteriors merge weakly (with respect to weak convergence on $H(\delta)$) in the sense of Diaconis and Freedman \cite{DiFr}. By Proposition 2.1 of \cite{PeRoSc}, we immediately have that the unscaled posteriors merge weakly with respect to the $\ell_2$-topology since they are both consistent \cite{KnSzVVVZ}. However, in the case of bounded Lipschitz functions (rather than the full case of continuous and bounded functions), we can improve this result to obtain a rate of convergence.

\begin{corollary}
For every $\beta,R > 0$, $s > 1/2$ and $\delta > 2$, we have
\begin{equation*}
\sup_{f_0 \in \mathcal{Q} (\beta,R) } \E_0 \beta_{H^{-s}} ( \Pi_n^H \circ \tau_\mathbb{Y}^{-1} , \Pi_{\hat{\alpha}_n} \circ \tau_\mathbb{Y}^{-1} ) \rightarrow 0
\end{equation*}\begin{equation*}
\sup_{f_0 \in \mathcal{Q}_{SS} (\beta,R,\varepsilon ) } \E_0 \beta_{H(\delta)} ( \Pi_n^H \circ \tau_\mathbb{Y}^{-1} , \Pi_{\hat{\alpha}_n} \circ \tau_\mathbb{Y}^{-1} ) \rightarrow 0
\end{equation*}
as $n \rightarrow \infty$. In particular, for $S = H^{-s}$ or $H(\delta)$ as above,
\begin{equation*}
\sup_{u : \norm{u}_{BL} \leq L } \left| \int_S  u \, d ( \Pi_n^H - \Pi_{\hat{\alpha}_n} ) \right| = o_{\P_0} \left( \frac{ L }{\sqrt{n}} \right)  .
\end{equation*}
\end{corollary}

\subsection{Slab and spike prior in $L^\infty$}\label{linfty setting}

Consider the slab and spike prior, whose frequentist contraction rate has been analyzed in Castillo and van der Vaart \cite{CaVV}, Hoffmann et al. \cite{HoRoSc} and Castillo et al. \cite{CaScVV}. The assumptions in \cite{HoRoSc} ensure that prior draws are very sparse and only very few coefficients are fitted. We therefore modify the prior slightly so that the prior automatically fits the first few coefficients of the signal without any thresholding. This ensures that the posterior will have a rough approximation of the signal before fitting wavelet coefficients more sparsely at higher resolution levels. This makes sense from a practical point of view by preventing overly sparse models and is in fact necessary from a theoretical perspective (see Proposition \ref{negative BvM in slab and spike}).

Let $J_n = \lfloor \log n / \log 2 \rfloor$ be such that $n/2 < 2^{J_n} \leq n$ and define some strictly increasing sequence $j_0 = j_0 (n) \rightarrow \infty$ such that $j_0 (n) <  J_n$. For the low resolutions $j \leq j_0(n)$ we fit a simple product prior where we draw the $f_{jk}$'s independent from a bounded density $g$ that is strictly positive on $\R$. For the middle resolution levels $j_0(n) < j \leq J_n$, the $f_{jk}$'s are drawn independently from the mixture
\begin{equation*}
\Pi_j (dx) = (1- w_{j,n} ) \delta_0 (dx) + w_{j,n} g(x) dx , \quad \quad \quad n^{-K} \leq w_{j,n} \leq 2^{-j(1 + \theta)} ,
\end{equation*}
for some $K > 0$ and $\theta > 1/2$. All coefficients at levels $j > J_n$ are set to 0. Since this is a product prior, one can sample from the posterior by sampling from each component separately (using either an MCMC scheme or explicit expressions depending on the choice of density $g$). We have a weak BvM in the multiscale space $\mathcal{M}_0 (w)$, where the rate at which the admissible sequence $(w_l)$ diverges depends on the how many coefficients we automatically fit in the prior via the sequence $j_0 (n)$. Recall that a sequence $( w_l )_{l \geq 1}$ is admissible if $w_l / \sqrt{l} \nearrow \infty$.

\begin{theorem}\label{thm3}
Consider the slab and spike prior defined above with lower threshold given by the strictly increasing sequence $j_0 (n) \rightarrow \infty$. The posterior distribution satisfies a weak BvM in $\mathcal{M}_0 (w)$ in the sense of Definition \ref{weak BvM}, that is for every $\beta,R>0$,
\begin{equation*}
\sup_{f_0 \in \mathcal{H} (\beta,R) } \E_0 \beta_{\mathcal{M}_0 (w) } ( \Pi_n \circ \tau_\mathbb{Y}^{-1} , \mathcal{N}  ) \rightarrow 0 
\end{equation*}
as $n \rightarrow \infty$, for any admissible sequence $(w_l)$ satisfying $w_{j_0(n)} / \sqrt{\log n} \nearrow \infty$.
\end{theorem}

Note that in the limiting case $w_l = \sqrt{l}$, we recover $j_0 (n) \simeq \log n$, so that the prior automatically fits the same fixed fraction of the full $2^{J_n} \simeq n$ coefficients. Since we consider only admissible sequences, the fraction of coefficients that the prior fits automatically is asymptotically vanishing. An alternative way to consider this result is in reverse: based on a desired rate in practice, we prescribe an admissible sequence $w_l = \sqrt{l} u_l$, where $u_l$ is some divergent sequence, and then pick $j_0 (n)$ appropriately. Taking $j_0(n)$ to grow more slowly than any power of $\log n$ means $(w_l)$ must grow faster than any power of $l$, resulting in a greater than logarithmic down-weighting of the wavelet coefficients in $\mathcal{M}(w)$. It may therefore be more appropriate to take $j_0(n)$ a power of $\log n$, which yields the following specific case.

\begin{corollary}\label{multiscale BvM example}
Consider the slab and spike prior defined above with lower threshold $j_0 (n) \simeq (\log n)^{\frac{1}{2\epsilon +1}}$ for some $\epsilon > 0$. Then it satisfies a weak BvM in $\mathcal{M}_0 (w)$ in the sense of Definition \ref{weak BvM}, that is for every $\beta,R>0$,
\begin{equation*}
\sup_{f_0 \in \mathcal{H} (\beta,R) }  \E_0 \beta_{\mathcal{M}_0 (w) } ( \Pi_n \circ \tau_\mathbb{Y}^{-1} , \mathcal{N}  ) \rightarrow 0 
\end{equation*}
as $n \rightarrow \infty$ for the admissible sequence $w_l = l^{1/2 + \epsilon} u_l$, where $u_l$ is any (arbitrarily slowly) diverging sequence.
\end{corollary}

While the requirement to fit the first few coefficients of the prior is mild and of practical use in nonparametrics, it is naturally of interest to study the behaviour of the posterior distribution with full thresholding, that is when $j_0 (n) \equiv 0$, which we denote by $\Pi'$. In general however, the full posterior contracts to the truth at a rate strictly slower than $1/\sqrt{n}$ in $\mathcal{M}(w)$, so that a $\sqrt{n}$-rescaling of the posterior cannot converge weakly to a limit. This holds even for self-similar functions.

\begin{proposition}\label{negative BvM in slab and spike}
Let $(w_l)$ be any admissible sequence. Then for any $\beta,R>0$, there exists $\varepsilon = \varepsilon (\beta,R,\psi)>0$ and $f_0 \in \mathcal{H}_{SS} (\beta,R,\varepsilon)$ such that along some subsequence $(n_m)$,
\begin{equation*}
\E_0 \Pi ' (\norm{f-\mathbb{Y}}_{\mathcal{M}(w) } \geq M_{n_m} {n_m}^{-1/2} \mid Y^{(n_m)} )  \rightarrow 1
\end{equation*}
for all $M_n \rightarrow \infty$ sufficiently slowly. Consequently, for such an $f_0$, a weak BvM in $\mathcal{M}_0 (w)$ in the sense of Definition \ref{weak BvM} cannot hold.
\end{proposition}

It is particularly relevant that Proposition \ref{negative BvM in slab and spike} applies to self-similar parameters since a major application of the weak BvM is the construction of adaptive credible regions with good frequentist properties under self-similarity (see Proposition \ref{credible band}). On the level of a $\sqrt{n}$-rescaling as in Definition \ref{weak BvM}, the rescaled posterior distribution asymptotically puts vanishingly small probability mass on any given $\mathcal{M} (w)$-ball infinitely often. This occurs because the posterior selects non-zero coordinates by thresholding at the level $\sqrt{\log n/ n}$ rather than the required $1/\sqrt{n}$ (Lemma 1 of \cite{HoRoSc}). The weighting sequence $(w_l)$ regularizes the extra $\sqrt{\log n}$ factor at high frequencies, but does not do so at low frequencies. This is the reason that the weighting sequence $(w_l)$ depends explicitly on the thresholding factor $\sqrt{\log n}$ in Theorem \ref{thm3}.

It seems that using such an adaptive scheme on low frequencies of the signal causes the weak BvM to fail. This prior closely resembles the frequentist practice of wavelet thresholding, where such a phenomenon has also been observed. For example, Gin\'e and Nickl \cite{GiNi2} require similar (though stronger) assumptions on the number of coefficients that need to be fitted automatically to obtain a central limit theorem for the distribution function of the hard thesholding wavelet estimator in density estimation (Theorem 8 of \cite{GiNi2}).

\section{Applications}

\subsection{Adaptive credible sets in $\ell_2$}\label{credible set section}

We propose credible sets from the hierarchical or empirical Bayes procedures, which we show are adaptive frequentist confidence sets for self-similar parameters. We consider the natural Bayesian approach of using the quantiles of the posterior distribution to obtain a credible set of prescribed posterior probability. By considering sets whose geometry is amenable to the space $H(\delta)$, the weak BvM implies that such credible sets are asymptotically confidence sets.

Recall that $\snorm{f}_{H(\delta)}^2 = \sum_{k=1}^\infty k^{-1} (\log k)^{-2\delta}  f_k^2  $. For a given significance level $0 < \gamma < 1$, consider the credible set
\begin{equation}
C_n = \left\{ f : \snorm{f - \mathbb{Y}}_{H(\delta)} \leq R_n / \sqrt{n} \right\},
\label{eq29}
\end{equation}
where $R_n = R_n (Y, \gamma )$ is chosen such that $\Pi_{\hat{\alpha}_n} ( C_n | Y) = 1-\gamma$ or $\Pi^H ( C_n | Y ) = 1-\gamma$. Since the empirical and hierarchical Bayes procedures both satisfy a weak BvM in $H(\delta)$, we have from Theorem 1 of \cite{CaNi} that in both cases
\begin{equation*}
\P_{f_0} ( f_0 \in C_n) \rightarrow 1-\gamma \quad \quad and \quad \quad R_n = O_{\P_0} (1)
\end{equation*}
as $n \rightarrow \infty$, so that $C_n$ is asymptotically an exact frequentist confidence set (of unbounded $\ell_2$-diameter). We control the diameter of the set using either the estimator $\hat{\alpha}_n$ or the posterior median as a smoothness estimate, and then use the standard frequentist approach of undersmoothing. In the first case, consider
\begin{equation}
\tilde{C}_n = \left\{  f :   \snorm{f - \mathbb{Y}}_{H(\delta)} \leq R_n / \sqrt{n} , \quad \snorm{f - \hat{f}_n}_{H^{\hat{\alpha}_n - \epsilon_n }} \leq C \sqrt{\log n}    \right\} ,
\label{EB credible set definition}
\end{equation}
where $\hat{f}_n$ is the posterior mean, $R_n$ is chosen as in $C_n$, $\epsilon_n$ (chosen possibly data dependently) satisfies $r_1 / (\log n) \leq \epsilon_n \leq (r_2/\log n)\wedge ( \hat{\alpha}_n/2)$ for some $0 < r_1 \leq r_2 \leq \infty$ and $C > 1/r_1$. The undersmoothing by $\epsilon_n$ is necessary since the posterior assigns probability 1 to $H^{\alpha'}$ for $\alpha ' < \hat{\alpha}_n$, while probability 0 to $H^{\hat{\alpha}_n}$ itself. Geometrically, $\tilde{C}_n$ is the intersection of two $\ell_2$-ellipsoids, $C_n$ and an $H^{\hat{\alpha}_n - \epsilon_n}$-norm ball. For a typical element $f$ in $\tilde{C}_n$, the size of the low frequency coordinates of $f$ are determined by $C_n$, while the smoothness condition in $\tilde{C}_n$ acts to regularize the elements of $C_n$ (which are typically not in $\ell_2$) by shrinking the higher frequencies.

\begin{proposition}\label{EB credible set}
Let $0 < \beta_{1} \leq \beta_{2} < \infty$, $R \geq 1$ and $\varepsilon > 0$. Then the confidence set $\tilde{C}_n$ given in \eqref{EB credible set definition} satisfies
\begin{equation*}
\sup_{\substack{f_0 \in \mathcal{Q}_{SS} (\beta,R,\varepsilon) \\\ \beta \in [\beta_{1},\beta_{2}] }} \left| \P_{f_0} (f_0 \in \tilde{C}_n) - (1-\gamma) \right| \rightarrow 0
\end{equation*}
as $n \rightarrow \infty$. For every $\beta \in [\beta_{1}, \beta_{2}]$, uniformly over $f_0 \in \mathcal{Q}_{SS}(\beta,R,\varepsilon)$,
\begin{equation*}
\Pi_{\hat{\alpha}_n} (\tilde{C}_n \mid Y) = 1 - \gamma  + O_{\P_0} \left(  n^{-C'n^{1/(4\beta+2)}} \right) 
\end{equation*}
for some $C' > 0$ independent of $\beta,R,\varepsilon,N_0,\rho$, while the $\ell_2$-diameter satisfies for $\delta >2$,
\begin{equation*}
|\tilde{C}_n |_2 = O_{\P_0} \left(   n^{ - \beta /(2\beta + 1)} (\log n)^{ (2\delta \beta + 1/2) / ( 2\beta +  1) }   \right) .
\end{equation*}
\end{proposition}

The logarithmic correction in the definition of $H(\delta)$ that is required for a weak BvM causes the $(\log n)^{2\delta \beta  / ( 2\beta +  1)}$  penalty (which is $O((\log n)^{2\delta})$ uniformly over $\beta \geq 0$); this is the price required for using a plug-in approach in $H(\delta)$. The remaining $(\log n)^{1/(4\beta+2)}$ factor arises due to the second constraint in $\tilde{C}_n$, where the $H^{\hat{\alpha}_n}$-radius must be taken sufficiently large to ensure $\tilde{C}_n$ has sufficient posterior probability.

While the second constraint in \eqref{EB credible set definition} reduces the credibility below $1-\gamma$, Proposition \ref{EB credible set} shows that this credibility loss is very small. The Bayesian approach takes care of this automatically since the posterior concentrates on a much more regular set than $\ell_2$. This is corroborated empirically by numerical evidence (see Figure \ref{credibility table}), which shows that the credibility of the set $\tilde{C}_n$ rapidly approaches $1-\gamma$ as $n$ increases.

\begin{remark}\label{credible set remark}
A naive interpretation of $C_n$ yields a credible set that is far too large, having unbounded $\ell_2$-diameter, with the additional constraint in $\tilde{C}_n$ needed to regularize the set. In actual fact the posterior does this regularization automatically with $C_n$ being ``almost optimal". Proposition \ref{EB credible set} could be rewritten for $C_n$ with exact credibility $\Pi_{\hat{\alpha}_n} (C_n \mid Y) = 1-\gamma$ and $\ell_2$-diameter satisfying
\begin{equation*}
\Pi_{\hat{\alpha}_n} ( f\in C_n : \snorm{f - \hat{f}_n}_2 \leq C n^{ - \frac{\beta}{2\beta + 1}} (\log n)^{ \frac{2\delta \beta + 1/2}{ 2\beta +  1} } \mid Y ) = 1 - \gamma +O_{\P_0} \left(  n^{-C'n^{1/(4\beta+2)}} \right) ,
\end{equation*}
for some $C,C'>0$. In view of this, the sets $C_n$ and $\tilde{C}_n$ are essentially the same from the point of view of the posterior, with $C_n$ having exact credibility for finite $n$ and correct $\ell_2$-diameter asymptotically and $\tilde{C}_n$ having the reverse. In particular, the finite time credibility ``gap" for either having too large radius in $C_n$ or smaller than $1-\gamma$ credibility for $\tilde{C}_n$ is of the same size. Moreover, the above statement holds without the need for a self-similarity assumption, which is possible since the confidence set does not strictly have optimal diameter. The same notion also holds for $C_n$ arising from the hierarchical Bayes procedure.
\end{remark}

\begin{remark}\label{efficient estimator remark}
By Lemma \ref{l2 efficient estimator} the empirical Bayes posterior mean $\hat{f}_n$ satisfies $\snorm{\hat{f}_n - \mathbb{Y}}_{H(\delta)} = o_{\P_0}(1/\sqrt{n})$ and so is an efficient estimator of $f_0$ in $H(\delta)$. Consequently one can substitute $\mathbb{Y}$ with $\hat{f}_n$ in the definitions of $C_n$ and $\tilde{C}_n$.
\end{remark}

Replacing the estimate $\hat{\alpha}_n$ with the median $\alpha_n^M$ of the marginal posterior distribution $\lambda_n (\cdot | Y)$ yields a fully Bayesian analogue. To obtain the necessary undersmoothing over a target range $[\beta_{1},\beta_{2}]$, we consider the shifted estimator $\hat{\beta}_n = \alpha_n^M - (C+1) / \log n$, where $C = C(R,\beta_{2},\varepsilon,\rho) = \max_{\beta_1 \leq  \beta \leq \beta_{2} } C(R,\beta,\varepsilon,\rho)$ is the constant appearing in Lemma \ref{posterior median lemma} (which can be explicitly computed). Consider
\begin{equation}
\tilde{C}_n ' = \left\{  f :   \snorm{f - \mathbb{Y}}_H \leq R_n / \sqrt{n} , \quad  \snorm{f - \hat{f}_n}_{H^{\hat{\beta}_n}} \leq M_n \sqrt{\log n}  \right\} ,
\label{credible set definition 2}
\end{equation}
where $\hat{f}_n$ is the posterior mean, $M_n \rightarrow \infty$ grows more slowly than any polynomial and $R_n$ is chosen as in $C_n$. Taking $C_n$ arising from the hierarchical Bayesian procedure $\Pi^H$, $\tilde{C}_n ' $ is a ``fully Bayesian" object. We have an analogue of Proposition \ref{EB credible set}.

\begin{proposition}\label{HB credible set}
Let $0 < \beta_{1} \leq \beta_{2} < \infty$, $R \geq 1$ and $\varepsilon > 0$. Then the confidence set $\tilde{C}_n'$ given in \eqref{credible set definition 2} satisfies
\begin{equation*}
\sup_{\substack{f_0 \in \mathcal{Q}_{SS} (\beta,R,\varepsilon) \\\ \beta \in [\beta_{1},\beta_{2}] }} \left| \P_{f_0} (f_0 \in \tilde{C}_n') - (1-\gamma) \right| \rightarrow 0
\end{equation*}
as $n \rightarrow \infty$. For every $\beta \in [\beta_{1}, \beta_{2}]$, uniformly over $f_0 \in \mathcal{Q}_{SS}(\beta,R,\varepsilon)$,
\begin{equation*}
\Pi^H (\tilde{C}_n' \mid Y) = 1 - \gamma  + o_{\P_0} \left( 1 \right) ,
\end{equation*}
while the $\ell_2$-diameter satisfies for $\delta > 2$,
\begin{equation*}
|\tilde{C}_n' |_2 = O_{\P_0} \left(   n^{ - \beta /(2\beta + 1)} (\log n)^{ (2\delta \beta + 1/2) / ( 2\beta +  1) }   \right) .
\end{equation*}
\end{proposition}

\subsection{Adaptive credible bands in $L^\infty$}\label{credible band section}

We provide a fully Bayesian construction of adaptive credible bands using the slab and spike prior. The posterior median $\tilde{f} = (\tilde{f}_{n,lk})_{(l,k) \in \Lambda}$ (defined coordinate-wise) takes the form of a thresholding estimator (c.f. \cite{AbSaSi}), which we use to identify significant coefficients. This has the advantage of both simplicity and interpretability and also provides a natural Bayesian approach for this coefficient selection. Such an approach was used by Kueh \cite{Ku} to construct an asymptotically honest (i.e. uniform in the parameter space) adaptive frequentist confidence set on the sphere using needlets. In that article, the coefficients are selected based on the empirical wavelet coefficients with the thresholds selected conservatively using Bernstein's inequality. In contrast, we use a Bayesian approach to automatically select the thresholding quantile constants that then yields exact coverage statements.

Let
\begin{equation}
D_n = \{ f : \snorm{f - \mathbb{Y}}_{\mathcal{M}(w)} \leq R_n / \sqrt{n}  \} ,
\label{M credible set}
\end{equation}
where $R_n  = R_n(Y,\gamma)$ is chosen such that $\Pi (D_n \mid Y) = 1-\gamma$. We then define the data driven width of our confidence band
\begin{equation}
\sigma_{n,\gamma} =  \sigma_{n,\gamma} (Y) = \sup_{x \in [0,1]} \sum_{l=0}^{J_n}  v_n \sqrt{\frac{\log n }{n}} \sum_{k=0}^{2^l -1}  1_{ \{ \tilde{f}_{lk} \neq 0 \} } |\psi_{lk} (x) | ,
\label{eq18}
\end{equation}
where $(v_n)$ is any (possibly data-driven) sequence such that $v_n \rightarrow \infty$. Under a local self-similarity type condition as in Kueh \cite{Ku}, one could possibly remove the supremum in \eqref{eq18} to obtain a spatially adaptive procedure. However, we restrict attention to more global self-similarity conditions here for simplicity. Since we consider wavelets satisfying \eqref{wavelet}, we have
\begin{equation*}
\sigma_{n,\gamma} \leq  v_n \sqrt{\frac{\log n }{n}} \sup_{x \in [0,1]} \sum_{l=0}^{J_n}  \sum_{k=0}^{2^l -1} \left| \psi_{lk}(x) \right| \leq  C(\psi)  v_n \sqrt{\frac{\log n }{n}} \sum_{l=0}^{J_n} 2^{l/2}  \leq C'  v_n \sqrt{\log n} < \infty \quad a.s.,
\end{equation*}
for all $n$ and $\gamma \in (0,1)$. Let $\pi_{med}$ denote the projection onto the non-zero coordinates of the posterior median and in a slight abuse of notation set
\begin{equation*}
\pi_{med}(Y)(x) = \sum_{l=0}^{J_n} \sum_{k = 0}^{2^l-1} Y_{lk} 1_{\{ \tilde{f}_{lk} \neq 0 \} } \psi_{lk}(x) ,
\end{equation*}
where we recall $Y_{lk} = \int_0^1 \psi_{lk}(t) dY(t)$. Consider the set
\begin{equation}
\overline{D}_n = \{ f : \snorm{f - \mathbb{Y}}_{\mathcal{M}(w)} \leq R_n / \sqrt{n} , \quad \snorm{f - \pi_{med} (Y)}_\infty \leq  \sigma_{n,\gamma}  (Y) \}   ,
\label{eq14}
\end{equation}
where $R_n$ is as in \eqref{M credible set}. This involves a two-stage procedure: we firstly calculate the required $\mathcal{M}(w)$-radius $R_n$ and then use the posterior median to select the coefficients deemed significant.

\begin{proposition}\label{credible band}
Let $0 < \beta_{1} \leq \beta_{2} < \infty$, $R \geq 1$ and $\varepsilon > 0$. Consider the slab and spike prior defined above with threshold $j_0 (n) \rightarrow \infty$ and let $(w_l)$ be any admissible sequence that satisfies $w_{j_0(n)} / \sqrt{\log n } \nearrow \infty$. Then the confidence set $\overline{D}_n$ given in \eqref{eq14}, using the choice $(w_l)$ and $\sigma_{n,\gamma}(Y)$ defined in \eqref{eq18} for $v_n \rightarrow \infty$, satisfies
\begin{equation*}
\sup_{\substack{f_0 \in \mathcal{H}_{SS} (\beta,R,\varepsilon) \\\ \beta \in [\beta_{1},\beta_{2}] }} \left| \P_{f_0} (f_0 \in \overline{D}_n ) - (1-\gamma) \right| \rightarrow 0
\end{equation*}
as $n \rightarrow \infty$. For every $\beta \in [\beta_{1}, \beta_{2}]$, uniformly over $f_0 \in \mathcal{H}_{SS}(\beta,R,\varepsilon)$,
\begin{equation*}
\Pi (\overline{D}_n \mid Y ) =  1- \gamma + o_{\P_0} (1)  ,
\end{equation*}
while the $L^\infty$-diameter satisfies
\begin{equation*}
| \overline{D}_n |_\infty =  O_{\P_0} \left(   (n / \log n)^{ -\beta/(2\beta + 1) }  v_n  \right) .
\end{equation*}
\end{proposition}

Under self-similarity, $\overline{D}_n$ has radius equal to the minimax rate in $L^\infty$ up to some factor $v_n$ that can be taken to diverge arbitrarily slowly, again mirroring a frequentist undersmoothing penalty. The choice of the posterior median is for simplicity and can be replaced by any other suitable thresholding procedure, for example directly using the posterior mixing probabilities between the atom at zero and the continuous density component.

One could also consider other alternatives to $\sigma_{n,\gamma}$ that simultaneously control the $L^\infty$-norm of the credible set whilst also preserving coverage and credibility. A similar construction to the credible sets in Section \ref{credible set section} could also be pursued by intersecting $D_n$ with a $B_{\infty1}^{\hat{\beta}_n}$-ball, where $\hat{\beta}_n$ is a suitable estimate of the smoothness. Alternatively, in view of Remark \ref{credible set remark}, one can also show that
\begin{equation*}
\Pi \left( \left. f\in D_n : \snorm{f - T_n}_\infty \leq  \left( \frac{w_{j_n(\beta)}}{\sqrt{j_n(\beta)}} n^{-\beta/(2\beta+1)} (\log n)^{(\beta+1)/(2\beta+1)} \right) \right| Y \right) = 1 - \gamma +o_{\P_0} \left(  1 \right) ,
\end{equation*}
where $T_n$ is an efficient estimator of $f_0$ in $\mathcal{M}$ that is also rate-optimal in $L^\infty$ (e.g. \eqref{efficient estimator1} or \eqref{efficient estimator2}) and $2^{j_n} \sim (n \log n)^{1/(2\beta+1)}$. The factor $ w_{j_n(\beta)}/\sqrt{j_n(\beta)}$ can be made to diverge arbitrarily slowly by the prior choice of $j_0(n)$.

\section{Posterior independence of the credible sets}\label{geometry section}

As shown above, the spaces $H(\delta) = H_2^{-1/2,\delta}$ and $\mathcal{M}(w)$ yield credible sets with good frequentist properties. However, given the different geometries proposed, it is of interest to compare them to more classical credible sets. Consider the $\ell_2$-ball studied in \cite{Fr,SzVVVZ} (though without the blow-up factor of the latter)
\begin{equation}
C_n^{\ell_2} = \{ f : \snorm{f-\hat{f}_n}_2 \leq \tilde{Q}_n (\hat{\alpha}_n,\gamma) \}  ,
\label{l2 credible set}
\end{equation}
where $\tilde{Q}_n (\hat{\alpha}_n, \gamma)$ is selected such that $\Pi_{\hat{\alpha}_n} (C_n^{\ell_2} \mid Y) = 1-\gamma$. Since the posterior variance of $\Pi_\alpha (\cdot \mid Y)$ is independent of the data, the radius $\tilde{Q}_n(\hat{\alpha}_n,\gamma)$ depends only on the data through $\hat{\alpha}_n$. By Theorem 1 of \cite{Fr} we have $\tilde{Q}_n (\alpha,\gamma) = Q_n n^{-\alpha/(2\alpha+1)}$, where $Q_n \rightarrow Q >0$.

Numerical examples of $\tilde{C}_n$ and $C_n^{\ell_2}$ are displayed in Section \ref{simulation section}. Given the similarity of $\tilde{C}_n$ and $C_n^{\ell_2}$ in Figures \ref{Fourier} and \ref{Volterra}, a natural question (voiced for example in \cite{Ca3,Ni,SzVVVZ4}) is to what extent these sets actually differ, both in theory and practice. From a purely geometric point of view these sets can be considered as infinite-dimensional ellipsoids with differing orientations. From a Bayesian perspective, an intriguing question is to what degree the decision rules on which these credible sets are based differ with respect to the posterior. For simplicity, we centre $\tilde{C}_n$ at the posterior mean $\hat{f}_n$, which we can do by Remark \ref{efficient estimator remark}.

\begin{theorem}\label{l2 asymptotic independence}
Suppose $a_n$ in \eqref{likelihood} satisfies in addition $a_n \leq \log n/(6\log n \log n)$. Then the $(1-\gamma)$-$H(\delta)$-credible ball $\tilde{C}_n$ defined in \eqref{EB credible set definition} and the $(1-\gamma)$-$\ell_2$-credible ball $C_n^{\ell_2}$ defined in \eqref{l2 credible set} are asymptotically independent under the empirical Bayes posterior, that is as $n \rightarrow \infty$,
\begin{equation*}
\Pi_{\hat{\alpha}_n} (\tilde{C}_n \cap C_n^{\ell_2} \mid Y) = \Pi_{\hat{\alpha}_n} (\tilde{C}_n \mid Y) \Pi_{\hat{\alpha}_n} (C_n^{\ell_2} \mid Y) + o_{\P_0} (1) = (1-\gamma)^2 + o_{\P_0}(1) 
\end{equation*}
uniformly over $f_0 \in \mathcal{Q}(\beta,R)$.
\end{theorem}

The first equality above also holds with $\tilde{C}_n$ replaced by $C_n$ or $C_n^{\ell_2}$ replaced by the blown-up $\ell_2$-credible ball studied in \cite{SzVVVZ}. Moreover, the above statement also holds for the hierarchical Bayes posterior with $\tilde{C}_n$ replaced by the $(1-\gamma)$-$H(\delta)$-credible ball $\tilde{C}_n'$ given in \eqref{credible set definition 2} and $C_n^{\ell_2}$ replaced by the corresponding hierarchical Bayes $\ell_2$-credible set.

Theorem \ref{l2 asymptotic independence} says that the Bayesian decision rules leading to the construction of $\tilde{C}_n$ and $C_n^{\ell_2}$ are fundamentally unrelated - one contains asymptotically no information about the other. Although we can conclude that $\tilde{C}_n$ and (blown-up) $C_n^{\ell_2}$ are frequentist confidence sets with similar properties, they express completely different aspects of the posterior. Note that this is not simply an artefact of the prior choice since the equivalent prior credible sets are not independent under the prior despite its product structure. An alternative interpretation is to consider Bayesian tests based on the credible regions, which have optimal frequentist properties. In this context, the two tests screen different and unrelated features. While both of these approaches are valid, both for the frequentist and the Bayesian, Theorem \ref{l2 asymptotic independence} says that neither of these constructions can be reduced to the other.

The $H(\delta)$-credible sets are principally determined by the low frequencies ($k\leq k_n$, where $k_n \rightarrow \infty$ in the proof), whereas the $\ell_2$-credible sets are driven by the high frequencies ($k > k_n$). The product structure of the posterior asymptotically decouples these two regimes yielding the independence statement. In particular the $H(\delta)$-norm down-weights the higher order frequencies enough that one is dealing with a close to finite-dimensional model. Such a result is unlikely to hold for arbitrary priors, unless there is some degree of posterior independence between the frequency ranges driving the different credible sets (though less independence than a full product posterior is necessary). The numerical simulations in Figure \ref{credibility table} corroborate Theorem \ref{l2 asymptotic independence} very closely, indicating that this result provides a good finite sample approximation to the posterior behaviour.

The posterior draws plotted in Section \ref{simulation section} are approximately drawn from the posterior distribution conditioned to the respective credible sets. Corollary \ref{l2 asymptotic independence corollary} quantifies how close these draws are in terms of the total variation distance $\snorm{\cdot}_{TV}$.

\begin{corollary}\label{l2 asymptotic independence corollary}
Let $\Pi_{\hat{\alpha}_n}^{\tilde{C}_n} (\cdot \mid Y)$, $\Pi_{\hat{\alpha}_n}^{C_n^{\ell_2}} (\cdot \mid Y)$ denote the posterior distribution conditioned to the sets $\tilde{C}_n$, $C_n^{\ell_2}$ respectively. Then as $n \rightarrow \infty$,
\begin{equation*}
\snorm{\Pi_{\hat{\alpha}_n}^{\tilde{C}_n} (\cdot \mid Y) - \Pi_{\hat{\alpha}_n}^{C_n^{\ell_2}} (\cdot \mid Y) }_{TV} = \gamma  + o_{\P_0} (1)  .
\end{equation*}
\end{corollary}

\begin{proof}
Each conditional distribution consists of the posterior distribution restricted to the relevant credible set and normalized by the same factor $(1-\gamma)$. The two distributions are therefore identical on their intersection and so the total variation distance equals
\begin{equation*}
\frac{1}{2} \left(  \frac{\Pi_{\hat{\alpha}_n} (\tilde{C}_n \cap (C_n^{\ell_2})^c \mid Y)}{\Pi_{\hat{\alpha}_n} (\tilde{C}_n \mid Y)} + \frac{\Pi_{\hat{\alpha}_n} (\tilde{C}_n^c \cap C_n^{\ell_2} \mid Y) }{\Pi_{\hat{\alpha}_n} ( C_n^{\ell_2} \mid Y)}  \right) = \frac{1}{2} \left(  \frac{2\gamma(1-\gamma) + o_{\P_0}(1)}{1-\gamma}  \right) = \gamma + o_{\P_0}(1) .
\end{equation*}
\end{proof}

Turning to the $L^\infty$-setting, for mathematical convenience let us consider the slightly stronger Besov norm $\|f \|_{B_{\infty1}^0} = \sum_{l} 2^{l/2} \max_k |\langle f, \psi_{lk}\rangle|$ as in Hoffmann et al. \cite{HoRoSc}. This norm is closely related to the $\|\cdot \|_\infty$-norm via the Besov space embbedings $B_{\infty1}^0 \subset L^\infty \subset B_{\infty\infty}^0$ (Chapter 4.3 of \cite{GiNi3}). Define
\begin{equation}
D_n^{L^\infty} = \{ f : \snorm{f - T_n}_{B_{\infty1}^0} \leq \overline{Q}_n (\gamma)  \}  ,
\label{L_infty credible set}
\end{equation}
where $T_n=T_n^{(2)}$ is given by \eqref{efficient estimator2} and is an efficient estimator of $f_0$ in $\mathcal{M}$ that is also rate-optimal in $L^\infty$ and $\overline{Q}_n(\gamma)$ is selected such that $\Pi (D_n^{L^\infty} \mid Y) = 1-\gamma$. The choice of $T_n$ is not essential, but it is convenient to select an estimator that can simultaneously act as the centering for both $D_n$ and $D_n^{L^\infty}$. We take the density $g$ in $\Pi$ to be Gaussian to simplify certain computations. Analogous results to those in $H(\delta)$ then hold.

\begin{theorem}\label{L_infty asymptotic independence}
Consider the slab and spike prior $\Pi$ with lower threshold $j_0 (n) \rightarrow \infty$ and let $g$ be the density of the Gaussian distribution $N(0,\tau^2)$. Let $(w_l)$ be any admissible sequence satisfying $w_{j_0(n)} / \sqrt{\log n} \nearrow \infty$ as $n\rightarrow \infty$. Then the $(1-\gamma)$-$\mathcal{M}(w)$-credible ball $\overline{D}_n$ defined in \eqref{eq14} and the $(1-\gamma)$-$L^\infty$-credible ball $D_n^{L^\infty}$ defined in \eqref{L_infty credible set} are asymptotically independent under the posterior, that is as $n \rightarrow \infty$,
\begin{equation*}
\Pi (\overline{D}_n \cap D_n^{L^\infty} \mid Y) = \Pi (\overline{D}_n \mid Y) \Pi (D_n^{L^\infty} \mid Y) + o_{\P_0}(1) = (1-\gamma)^2 + o_{\P_0}(1) 
\end{equation*}
 uniformly over $f_0 \in \mathcal{H}(\beta,R)$.
\end{theorem}

In particular the choice of $j_0(n)$ in Corollary \ref{multiscale BvM example} satisfies the conditions of Theorem \ref{L_infty asymptotic independence} since then $w_{j_0(n)} \simeq u_n \sqrt{\log n}$, where $u_n$ can be made to diverge arbitrarily slowly.

\begin{corollary}
Consider the same conditions as in Theorem \ref{L_infty asymptotic independence} and let $\Pi^{\overline{D}_n} (\cdot \mid Y)$, $\Pi^{D_n^{L^\infty}} (\cdot \mid Y)$ denote the posterior distribution conditioned to the sets $\overline{D}_n$, $D_n^{L^\infty}$ respectively. Then as $n \rightarrow \infty$,
\begin{equation*}
\snorm{\Pi^{\overline{D}_n} (\cdot \mid Y) - \Pi^{D_n^{L^\infty}} (\cdot \mid Y) }_{TV} = \gamma  + o_{\P_0} (1)  .
\end{equation*}
\end{corollary}

\subsection*{Heuristics for an extension to density estimation}

The proofs of Theorems \ref{l2 asymptotic independence} and \ref{L_infty asymptotic independence} presented here rely on the independence of the coordinates in the Gaussian white noise model. This model can be viewed as an idealized version of other more concrete statistical models, being mathematically more tractable. In view of the extension of the nonparametric BvM to density estimation in \cite{CaNi2}, let us briefly discuss a heuristic of what we might expect in this setting.

Suppose we observe $Y_1,...,Y_n$ i.i.d. observations from an unknown density $f_0$ on $[0,1]$. Assume that $f_0$ is uniformly bounded away from 0 and that $f_0\in C^\beta([0,1])$, where $1/2<\beta\leq 1$. Consider the simple histogram prior $\Pi$ specified in \eqref{Dirichlet prior}, where we again ignore adaptation issues and select $L = L_n \rightarrow \infty$ based on the smoothness of $f_0$. Such a prior has been shown to contract optimally in $L^\infty$ by Castillo \cite{Ca} and to satisfy a weak BvM in $\mathcal{M}_0$ by Castillo and Nickl \cite{CaNi2}.

Let $(\psi_{lk})$ denote the Haar wavelet basis on $[0,1]$ and $\mathcal{M}$ the related multiscale space for a suitable admissible sequence $(w_l)$. Further let $\pi_A$ denote the projection onto the elements of the Haar wavelet basis with resolution level contained in $A$. One can show that for suitable sequences $j_{n,1}, j_{n,2} \rightarrow \infty$ satisfying $2^{j_{n,1}} \ll 2^{j_{n,2}}$,
\begin{equation*}
\begin{split}
& \Pi  ( f: \snorm{f - T_n}_\mathcal{M} \leq R_n / \sqrt{n} , \quad  \snorm{f - T_n}_\infty \leq  \bar{Q}_n  \mid  Y )  
\\
& = \Pi (f: \| \pi_{\leq j_{n,1}} (f-T_n)\|_\mathcal{M} \leq (R_n + o(1))/\sqrt{n}, \| \pi_{\geq j_{n,2}}(f-T_n)\|_\infty \leq \bar{Q}_n + \delta_n | Y) + o_{\P_0}(1),
\end{split}
\end{equation*}
where $T_n$ is a suitable centering, $R_n/\sqrt{n}$ and $\bar{Q}_n$ are the $(1-\gamma)$-quantiles of the respective credible sets and $\delta_n$ is selected small enough to only change the credibility of the latter set by $o_{\P_0}(1)$.

Using the conjugacy of the Dirichlet distribution with multinomial sampling, the posterior distribution for $(h_k)$ is $\mathcal{D}(N_1+1,...,N_L+1)$, where $N_k = |\{ Y_i : Y_i \in I_{Lk} \}|$. Observe that the law of $f_{lk} = \langle f , \psi_{lk} \rangle$ under the posterior depends principally on the observations falling within $\supp (\psi_{lk}) = [k2^{-l},(k+1)2^{-l}]$. Unlike the Gaussian white noise model, there is dependence across the posterior wavelet coefficients due to the dependence within the Dirichlet distribution and the constraint that the number of observations sums to $n$.

The $\|\cdot \|_\mathcal{M}$-norm in the above display is the weighted maximum of $\{ f_{lk}: l \leq j_{n,1}\}$. If $j_{n,1}$ does not grow too fast, this consists of relatively few ``large sample" averages. Heuristically, this term behaves like a central order statistic, being driven by the average sample behaviour. On the contrary, the $\|\cdot\|_\infty$-term is determined by the largest coefficients at each resolution level $l \geq j_{n,2}$. Since $2^{j_{n,1}} \ll 2^{j_{n,2}}$, these can be seen to behave more like extreme order statistics, being the maximum of many almost independent ``small samples" (at least relative to the frequencies $l\leq j_{n,1}$). Even though order statistics depend, by definition, on all observations, central and extreme order statistics asymptotically depend on the observations in orthogonal ways and become stochastically independent (c.f. Chapter 21 of \cite{VV}). One might therefore hope that the two norms in the previous display are asymptotically independent in the sense of Theorems \ref{l2 asymptotic independence} and \ref{L_infty asymptotic independence}.

An alternative way to understand why the wavelet coefficients at a given resolution level may be considered ``almost independent" under the posterior is via Poissonization. It is well-known that density estimation is asymptotically equivalent to Poisson intensity estimation \cite{LoZh,RaSc}, where one observes a Poisson process with intensity measure $nf_0$. Equivalently, the Poisson experiment corresponds to observing a Poisson random variable $N$ with expectation $n$ and then independently of $N$ observing $Y_1,...,Y_N$ i.i.d. with density $f_0$. In this framework, the dependence induced by the number of observations summing to $n$ is removed, meaning that the variables $(N_1,...,N_L)$ defined above are fully independent. The remaining dependence is due to the Dirichlet distribution and becomes negligible as the number of bins $L_n \rightarrow \infty$. Since this equivalence is asymptotic in nature, one should expect such a heuristic to manifest itself also asymptotically.

\section{Simulation example}\label{simulation section}

We now apply our approach in a numerical example. Following on from the example of the $\mathcal{M}(w)$-based credible set \eqref{Dirichlet credible set}, we now consider the space $H_2^{-1/2,\delta}$. Consider the Fourier sine basis
\begin{equation*}
e_k(x) = \sqrt{2} \sin (k \pi x) , \quad \quad \quad k=1,2,...,
\end{equation*}
and define the true function $f_{0,k} = \langle f_0 , e_k \rangle_2 = k^{-3/2} \sin(k)$ so that the true smoothness is $\beta = 1$. We consider realisations of the data \eqref{model} at levels $n = 500$ and $2000$ and use the empirical Bayes posterior distribution. We plotted the true $f_0$ (black), the posterior mean (red) and an approximation to the credible sets (grey). To simulate the $\ell_2$ credible balls $C_n^{\ell_2}$ given in \eqref{l2 credible set}, we sampled 2000 curves from the posterior distribution and kept the 95\% closest in the $\ell_2$ sense to the posterior mean and plotted them (grey). We performed the same approach to obtain the full $H(\delta)$-credible set $C_n$ given in \eqref{eq29} and then plotted the full adaptive confidence set $\tilde{C}_n$ given in \eqref{EB credible set definition} with $C=1$ and $\epsilon_n = 1/\log n$. We also present the approximate credibility of $\tilde{C}_n$ by considering the fraction of the simulated curves from the posterior that satisfy the extra constraint of $\tilde{C}_n$ that $\snorm{f - \hat{f}_n}_{H^{\hat{\alpha}_n - \epsilon_n}} \leq \sqrt{\log n}$. This is given in Figure \ref{Fourier}.

While the true $\ell_2$ and $H(\delta)$ credible balls are unbounded in $L^\infty$, the posterior draws can be shown to be bounded in $L^\infty$ explaining the boundedness of the plots. Sampling from the posterior (and thereby implicitly intersecting the sets $C_n^{\ell_2}$ and $\tilde{C}_n$ with the posterior support) seems the natural approach for the Bayesian. Indeed those elements that constitute the ``roughest" or least regular elements of the credible sets are not seen by the posterior, that is they have little or no posterior mass (see Lemma \ref{exponential lemma}). The posterior contains significantly more information than merely the $\ell_2$ or $H(\delta)$ norm of the parameter of interest, as can be seen by it assigning mass 1 to a strict subset of $\ell_2$. For further discussion on plotting such credible sets see \cite{Ca3,LoMa,SzVVVZ4}.

\begin{figure}[h]
\centering
\includegraphics[scale=0.245]{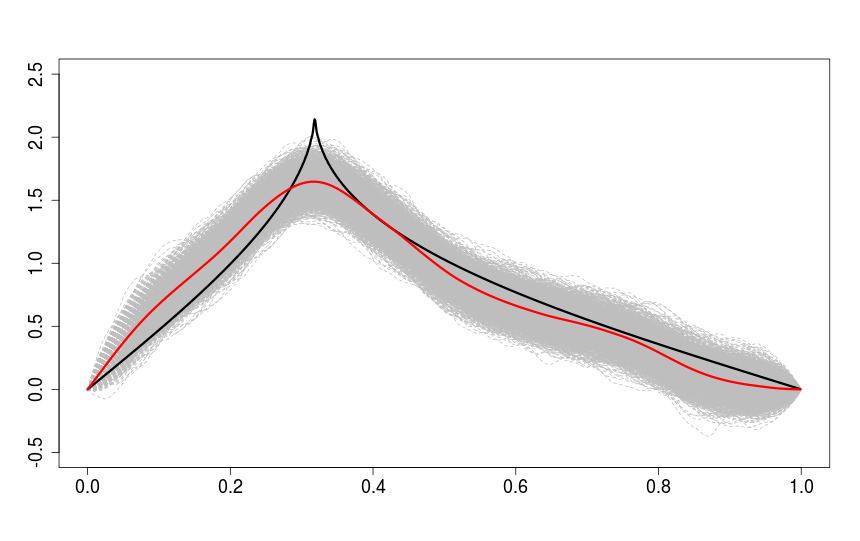}
\includegraphics[scale=0.245]{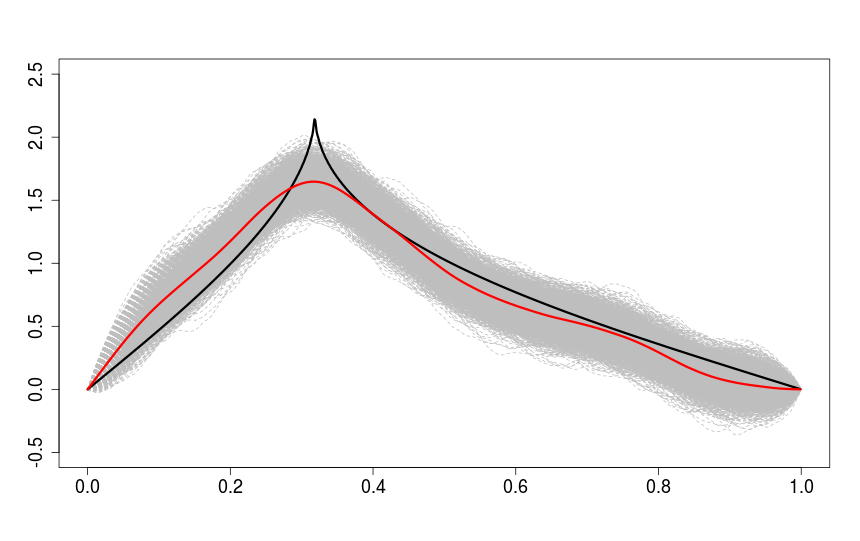}
\includegraphics[scale=0.25]{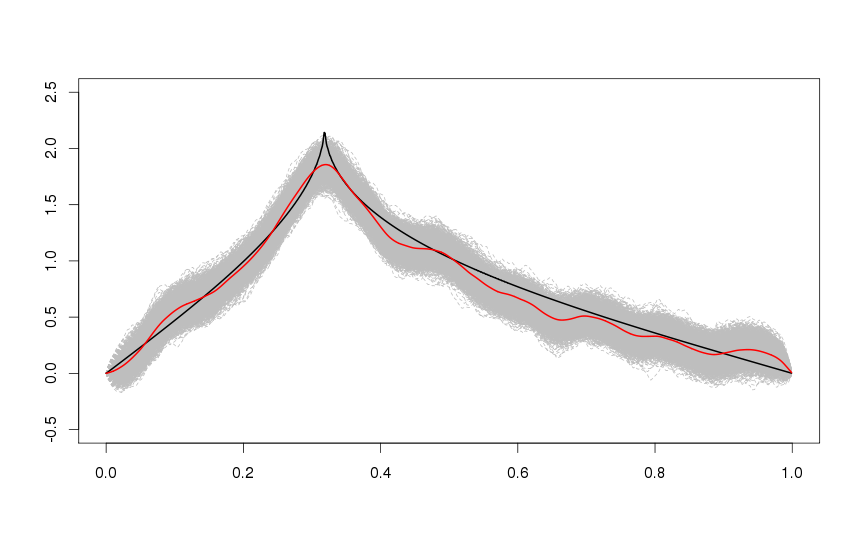}
\includegraphics[scale=0.25]{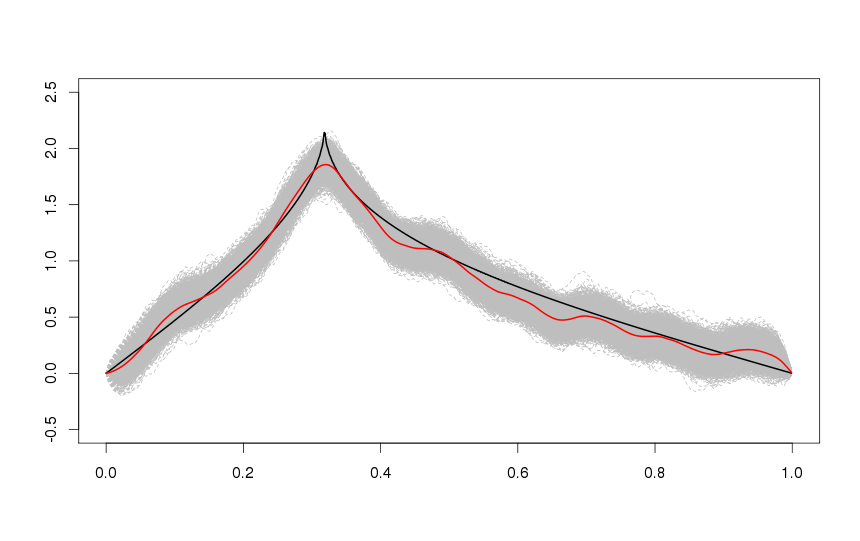}

\caption{\emph{Empirical Bayes credible sets for the Fourier sine basis with the true curve (black) and the empirical Bayes posterior mean (red). The left panels contain the $\ell_2$ credible ball $C_n^{\ell_2}$ given in \eqref{l2 credible set} and the right panels contain the set $\tilde{C}_n$ given in \eqref{EB credible set definition}. From top to bottom, $n = 500, 2000$ and $\hat{\alpha}_n = 1.29,1.01$, with the right-hand side each having credibility 95\%.}}
\label{Fourier}
\end{figure}

For a given set of 2000 posterior draws, we also computed the credibility of $\tilde{C}_n$ at a chosen significance level and the credibility of the posterior draws falling in both $\tilde{C}_n$ and $C_n^{\ell_2}$. This latter quantity has value $(1-\gamma)^2 + o_{\P_0} (1)$ by Theorem \ref{l2 asymptotic independence}. We repeated this 20 times and the average values are presented in Figure \ref{credibility table}.

\begin{figure}[h]
\centering
\begin{tabular}{| l | c | c | c | c | c | }
\hline
& \multicolumn{4}{|c|}{n=500}  \\
\hline
Chosen significance & 0.95 & 0.90 & 0.85 & 0.80  \\
Credibility of $\tilde{C}_n$ & 0.9500 & 0.8999 & 0.8499 & 0.8000   \\
Credibility of $\tilde{C}_n \cap C_n^{\ell_2}$ & 0.9020 & 0.8102 & 0.7220 & 0.6406   \\
Expected credibility of $\tilde{C}_n \cap C_n^{\ell_2}$ & 0.9025 & 0.8100 & 0.7225 & 0.6400 \\
\hline
& \multicolumn{4}{|c|}{n=2000}  \\
\hline
Chosen significance & 0.95 & 0.90 & 0.85 & 0.80  \\
Credibility of $\tilde{C}_n$ & 0.9500 & 0.9000 & 0.8500 & 0.8000   \\
Credibility of $\tilde{C}_n \cap C_n^{\ell_2}$ & 0.9025 & 0.8095 & 0.7226 & 0.6409   \\
Expected credibility of $\tilde{C}_n \cap C_n^{\ell_2}$ & 0.9025 & 0.8100 & 0.7225 & 0.6400 \\
\hline
\end{tabular}

\caption{\emph{Table showing the average credibility of $\tilde{C}_n$, the average credibility of the posterior draws falling in both sets and the expected value of the latter (from Theorem \ref{l2 asymptotic independence}).}}
\label{credibility table}
\end{figure}

The posterior distribution appears to have some difficulty visually capturing the resulting function at its peak. In fact the credible sets do ``cover the true function", but do so in an $\ell_2$ rather than an $L^\infty$-sense. Indeed, any $\ell_2$-type confidence ball will be unresponsive to highly localized pointwise features since they occur on a set of small Lebesgue measure (as in this case). Similar reasoning also explains the performance of the posterior mean at this point. The posterior mean estimates the Fourier coefficients of $f_0$ and hence estimates the true function in an $\ell_2$-sense via its Fourier series.

In Section \ref{geometry section} it was shown that the two approaches behave very differently theoretically and the numerical results in Figure \ref{credibility table} match this theory very closely. It appears that the two methods do indeed use different rejection criteria in practice resulting in different selection outcomes. The visual similarity between the $\ell_2$ and $H (\delta)$-credible balls in Figure \ref{Fourier} is therefore a result of the posterior draws themselves looking similar, rather than the methods performing identically.

We note that already by $n = 500$, $\tilde{C}_n$ has the correct credibility so that the high frequency smoothness constraint is satisfied with posterior probability virtually equal to 1 (c.f. Proposition \ref{EB credible set}). $\tilde{C}_n$ is therefore an actual credible set for reasonable (finite) sample sizes rather than a purely asymptotic credible set. The posterior distribution already strongly regularizes the high frequencies so that the posterior draws are very regular with high probability. This can be quantitatively seen by the rapidly decaying variance term of the posterior distribution \eqref{eq28}. This is indeed the case in the simulation, where the credibility gap is negligible, thereby demonstrating that most of the posterior draws already satisfy the smoothness constraint in $\tilde{C}_n$.

We repeat the same simulation using the same true function $f_{0,k} = k^{-3/2} \sin (k)$, but with basis equal to the singular value decomposition (SVD) of the Volterra operator (c.f. \cite{KnVVVZ}):
\begin{equation*}
e_k (x) = \sqrt{2} \cos ( (k-1/2)\pi x) , \quad \quad \quad k=1,2,...
\end{equation*}
and plot this in Figure \ref{Volterra} for $n = 1000$.and plot this in Figure \ref{Volterra} for $n = 1000$. Unlike Figure \ref{Fourier}, the resulting function has no ``spike" and so both credible sets have no trouble visually capturing the true function (though one should remember that these are $\ell_2$ rather than $L^\infty$ type credible sets).

\begin{figure}[h]
\centering
\includegraphics[scale=0.25]{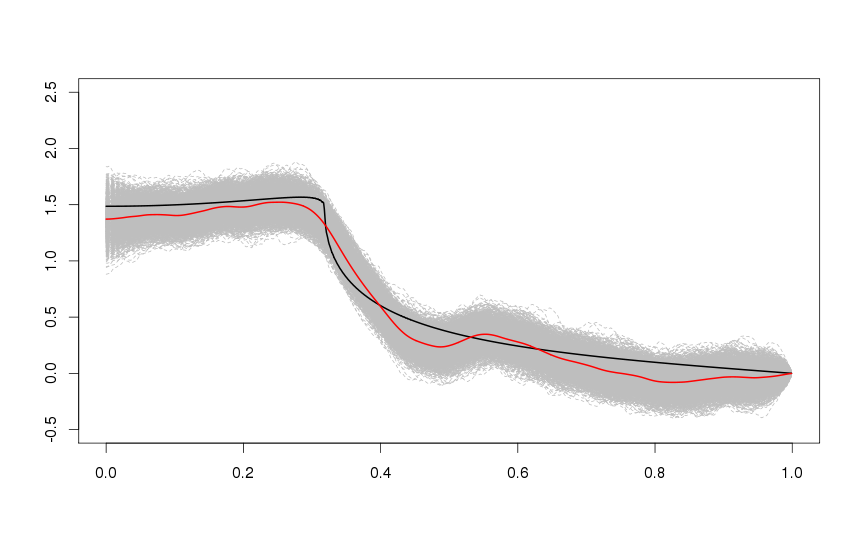}
\includegraphics[scale=0.25]{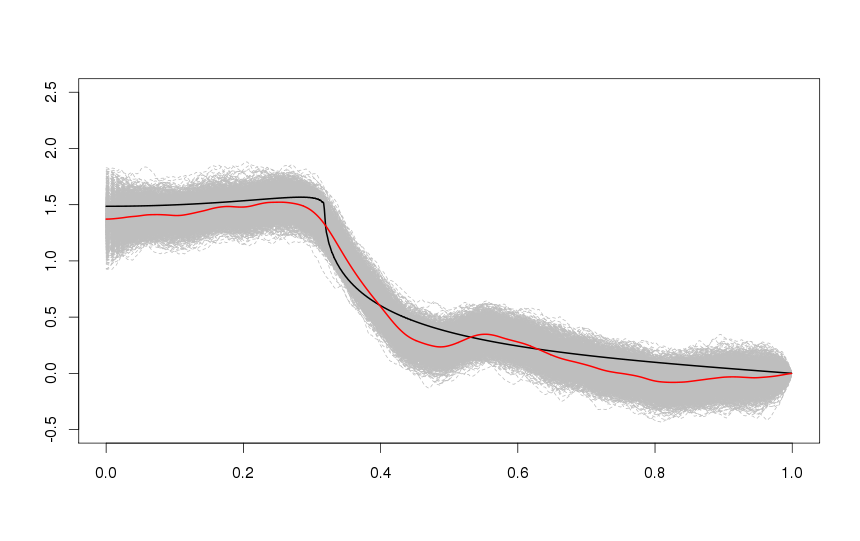}

\caption{\emph{Empirical Bayes credible sets for the Volterra SVD basis with the true curve (black) and the empirical Bayes posterior mean (red) for $n = 1000$ and $\hat{\alpha}_n = 1.07$. The left and right panels contain the $\ell_2$ credible ball $C_n^{\ell_2}$ given in \eqref{l2 credible set} and $\tilde{C}_n$ (credibility 95\%) given in \eqref{EB credible set definition} respectively.}}
\label{Volterra}
\end{figure}

We now illustrate the multiscale approach using the slab and spike prior with lower threshold $j_0(n) = \sqrt{\log n}$, plotting the true function (solid black) and posterior mean (red) at levels $n=200,500$. We have used Haar wavelets, set $g$ to be $N(0,1/2)$ and have taken prior weights $w_{j,n} = \min (n^{-1},2^{-5.5j})$, corresponding to $K=1$ and $\theta = 5$. For $n=200$ and $500$, we have fitted one scaling function plus $2^8-1 = 255$ and $2^9-1=511$ wavelet coefficients respectively (i.e. $2^{J_n+1}-1$). We again sampled 2000 curves from the posterior distribution and plotted the 95\% closest to the posterior mean in the $\mathcal{M}(w)$ sense (grey) to simulate $D_n$ in \eqref{M credible set}. We also used the posterior draws to generate a 95\% credible band in $L^\infty$ by estimating $\bar{Q}_n(0.05)$ and then plotting $D_n^{L^\infty}$ in \eqref{L_infty credible set} (dashed black). Finally we computed local 95\% credible intervals at every point $x \in [0,1]$ and joined these to form a credible band (dashed blue). This is given in Figure \ref{slab and spike picture}.

\begin{figure}[h]
\includegraphics[scale=0.2]{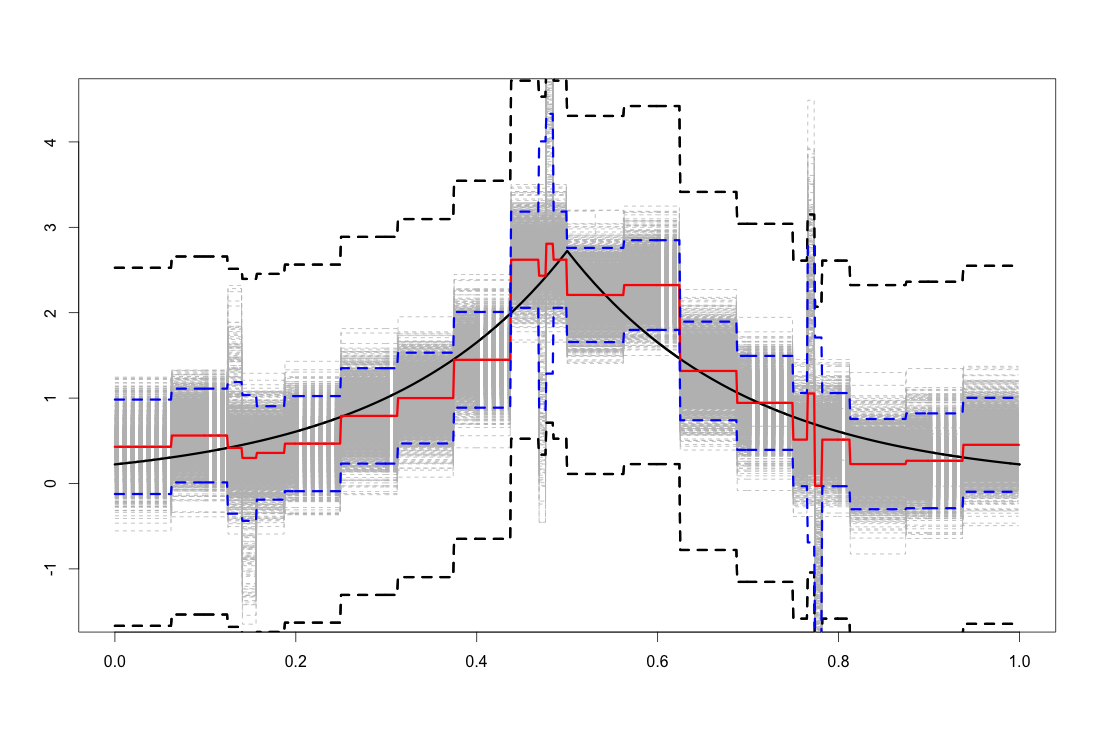}
\includegraphics[scale=0.2]{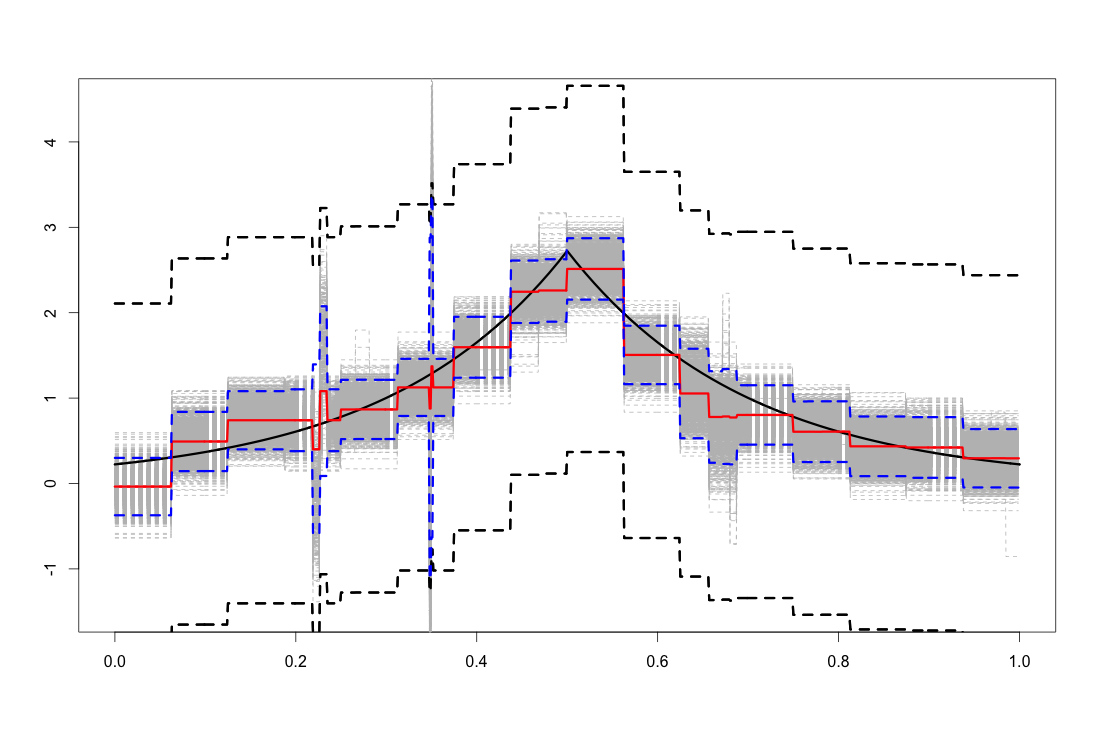}

\caption{\emph{Slab and spike credible sets with the true curve (black), posterior mean (red), a 95\% credible band in $L^\infty$ (dashed black), pointwise 95\% credible intervals (dashed blue) and the set $D_n$ given in \eqref{M credible set} (grey). We have $n=200,500$ respectively.}}
\label{slab and spike picture}
\end{figure}

We see from Figure \ref{slab and spike picture} that each posterior draw consists of a rough approximation of the signal via frequencies $j \leq j_0(n)$ with a few ``spikes" from the high frequencies; the rather unusual shape is a reflection of the prior choice. It is worth noting that the posterior draws are bounded in $L^\infty$ since the posterior contracts rate optimally to the truth in $L^\infty$ \cite{HoRoSc}. We see that the $L^\infty$ diameter of $D_n$ is strictly greater than that of the $L^\infty$-credible bands, though this only manifests itself in a few places. The size of the $L^\infty$-bands is driven by the size of the spikes, which are few in a number but occur in every posterior draw, resulting in seemingly very wide credible bands.

On the contrary, the local credible intervals ignore the spikes since less than 5\% of the draws have a spike at any given point, resulting in much tighter bands. The dashed blue lines in effect correspond to the 95\% $L^\infty$-band from a prior fitting exclusively the low frequencies $j \leq j_0(n)$, which is a non-adaptive prior modelling analytic smoothness. This dramatically oversmoothes the truth resulting in far too narrow credible bands and is highly dangerous since it is known that oversmoothing the truth can yield zero coverage \cite{KnVVVZ,Lea}.

\section{Proofs}

In what follows denote by $\pi_j$  the projection onto either $V_j = \text{span} \{ e_k : 1 \leq k \leq j\}$ or $V_j = \text{span} \{ \psi_{lk} : 0 \leq l \leq j , k = 0,...,2^l - 1 \}$ depending on whether we are considering a Fourier-type basis or a wavelet basis. Similarly define $\pi_{>j}$ to be the projection onto $\text{span} \{ e_k :  k > j\}$ or $V_j = \text{span} \{ \psi_{lk} : l> j , k = 0,...,2^l - 1 \}$.

\subsection{Proofs of weak BvM results in $\ell_2$ (Theorems \ref{EB weak BvM} and \ref{HB weak BvM})}

To prove a weak BvM we need to show that the posterior contracts at rate $1/\sqrt{n}$ to the truth in the relevant space and that the finite-dimensional projections of the rescaled posterior converge weakly to those of the normal law $\mathcal{N}$ (see Theorem 8 of \cite{CaNi} for more discussion). The latter condition is implied by a classical parametric BvM in total variation. Recall that $\hat{\alpha}_n$ is the maximum marginal likelihood estimator defined in \eqref{likelihood}.

\begin{theorem}\label{thm1}
For every $\beta , R>0$ and $M_n \rightarrow \infty$, we have
\begin{equation*}
\sup_{f_0 \in \mathcal{Q} (\beta,R) } \E_0 \Pi_{\hat{\alpha}_n} (  f : \norm{ f - f_0}_S \geq M_n L_n  n^{-1/2}  | Y ) \rightarrow 0,
\end{equation*}
where $S = H(\delta)$ or $H^{-s}$ for $s > 1/2$. If $S = H(\delta)$ then $L_n = (\log n)^{3/2} (\log \log n)^{1/2}$; if in addition $f_0 \in \mathcal{Q}_{SS} (\beta,R,\varepsilon)$, then the rate improves to $L_n = 1$ for $\delta \geq 2$. If $S = H^{-s}$ for $s > 1/2$, then $L_n = 1$.
\end{theorem}

\begin{proof}
This contraction result is proved in the same manner as Theorem 2 in \cite{KnSzVVVZ}, with suitable modifications for the different norms used. In the case $S = H(\delta)$, self-similarity is needed to obtain a sharp upper bound on the behaviour of $\hat{\alpha}_n$, which is required to bound the posterior bias.
\end{proof}

\begin{theorem}\label{EB parametric BvM theorem}
The finite dimensional projections of the empirical Bayes procedure satisfy a parametric BvM, that is for every finite dimensional subspace $V \subset \ell_2$,
\begin{equation*}
\sup_{f_0 \in \mathcal{Q}(\beta,R)} \E_0 \snorm{ \Pi_{\hat{\alpha}_n} (\cdot | Y) \circ T_\mathbb{Y}^{-1} - N_V (0,I) }_{TV} \rightarrow 0 ,
\end{equation*}
where $\pi_V$ denotes the projection onto $V$ and $T_z : f \mapsto \sqrt{n} \pi_V (f-z)$.
\end{theorem}

\begin{proof}
Without loss of generality, let $V = \text{span} \{ e_k : 1 \leq k \leq J \}$. Using Pinsker's inequality and that $\hat{\alpha}_n \in [0,a_n]$ by the choice \eqref{likelihood},
\begin{equation*}
\begin{split}
\snorm{ \Pi_{\hat{\alpha}_n} (\cdot \mid Y) \circ T_\mathbb{Y}^{-1} - N(0,I_J) }_{TV}^2 & \leq \sup_{\alpha \in [0,a_n]}  \snorm{ \Pi_{\alpha} (\cdot \mid Y) \circ T_\mathbb{Y}^{-1} - N(0,I_J) }_{TV}^2  \\
& \leq \sup_{\alpha \in [0,a_n]} KL (\Pi_{\alpha} (\cdot \mid Y) \circ T_\mathbb{Y}^{-1} , N(0,I_J) ) ,
\end{split}
\end{equation*}
where $KL$ denotes the Kullback-Leibler divergence. Using the exact formula for the Kullback-Leibler divergence between two Gaussian measures on $\R^J$,
\begin{equation*}
\begin{split}
KL & (\Pi_{\alpha} (\cdot \mid Y) \circ T_\mathbb{Y}^{-1} , N(0,I_J) ) \\
& = \frac{1}{2} \left[  \sum_{k=1}^J \frac{n}{k^{2\alpha+1} + n} + n\sum_{k=1}^{J} \frac{k^{4\alpha+2}Y_k^2}{(k^{2\alpha+1}+n)^2} + \sum_{k=1}^J \log \left( \frac{n}{k^{2\alpha+1}+n} \right) - J \right]  \\
& \leq \frac{1}{2n}  \sum_{k=1}^J \left[ k^{2\alpha+1} + k^{4\alpha+2} Y_k^2  \right]   \\
& \lesssim \frac{J^{2\alpha+2}}{n} +  \frac{J^{4\alpha+3}}{n} \max_{1 \leq k \leq J} Y_k^2.
\end{split}
\end{equation*}
Since $\E_0 \max_{1 \leq k \leq J} Y_k^2 = O (1)$ for fixed $J$ and $\alpha \leq a_n =  o(\log n)$ by the choice of $a_n$, the result follows.
\end{proof}

\begin{proof}[Proof of Theorem \ref{EB weak BvM}]
Fix $\eta > 0$, let $S$ denote $H^{-s}$ or $H(\delta)$ as appropriate and set $\tilde{\Pi}_{\hat{\alpha}_n} = \Pi_{\hat{\alpha}_n} \circ \tau_\mathbb{Y}^{-1}$. By the triangle inequality,
\begin{equation*}
\beta_S ( \tilde{\Pi}_{\hat{\alpha}_n} , \mathcal{N} ) \leq \beta_S ( \tilde{\Pi}_{\hat{\alpha}_n} , \tilde{\Pi}_{\hat{\alpha}_n} \circ \pi_j^{-1} ) + \beta_S ( \tilde{\Pi}_{\hat{\alpha}_n} \circ \pi_j^{-1} , \mathcal{N} \circ \pi_j^{-1} ) + \beta_S ( \mathcal{N} \circ \pi_j^{-1} , \mathcal{N} ) ,
\end{equation*}
for some $j > 0$. Using the contraction result of Theorem \ref{thm1} and following the argument of Theorem 8 of \cite{CaNi}, we deduce that the $\E_0$-expectation of the first term is smaller that $\eta /3$ for sufficiently large $j$, uniformly over the relevant function class (in the case of $H(\delta)$ the result holds for all $\delta > 2$ - we recall from the proof of that theorem that if the required contraction is established in $H(\delta ')$, then the required tightness argument holds in $H(\delta)$ for any $\delta > \delta '$). A similar result holds for the third term. For the middle term, note that the total variation distance dominates the bounded Lipschitz metric. For fixed $j$, we thus have that for $n$ large enough,
\begin{equation*}
\E_0 \beta_S ( \tilde{\Pi}_{\hat{\alpha}_n} \circ \pi_j^{-1} , \mathcal{N} \circ \pi_j^{-1} ) \leq \E_0 \snorm{ \Pi_{\hat{\alpha}_n} (\cdot | Y) \circ T_\mathbb{Y}^{-1} - N_V (0,I) }_{TV} \leq \eta /3 ,
\end{equation*}
using Theorem \ref{EB parametric BvM theorem} with $V = V_j$.
\end{proof}

A similar situation holds true for the hierarchical Bayesian prior.

\begin{theorem}\label{thm2}
Suppose that the prior density $\lambda$ satisfies Condition \ref{HB density condition}. Then for every $\beta , R>0$ and $M_n \rightarrow \infty$, we have
\begin{equation*}
\sup_{f_0 \in \mathcal{Q} (\beta,R) } \E_0 \Pi^H \left( f :  \norm{f - f_0}_S \geq M_n L_n  n^{-1/2} | Y \right) \rightarrow 0  ,
\end{equation*}
where $S = H(\delta)$ or $H^{-s}$ for $s > 1/2$. If $S = H(\delta)$ then $L_n = (\log n)^{3/2} (\log \log n)^{1/2}$; if in addition $f_0 \in \mathcal{Q}_{SS} (\beta,R,\varepsilon)$, then the rate improves to $L_n = 1$ for $\delta \geq 2$. If $S = H^{-s}$ for $s > 1/2$, then $L_n = 1$.
\end{theorem}

\begin{proof}
This result is proved in the same manner as Theorem 3 in \cite{KnSzVVVZ}, with suitable modifications arising as in the proof of Theorem \ref{thm1}.
\end{proof}

\begin{theorem}\label{parametric BvM theorem}
The finite dimensional projections of the hierarchical Bayesian procedure satisfy a parametric BvM, that is for every finite dimensional subspace $V \subset \ell_2$,
\begin{equation*}
\sup_{f_0 \in \mathcal{Q}(\beta,R)} \E_0 \snorm{ \Pi^H (\cdot | Y) \circ T_\mathbb{Y}^{-1} - N_V (0,I) }_{TV} \rightarrow 0 ,
\end{equation*}
where $\pi_V$ denotes the projection onto $V$ and $T_z : f \mapsto \sqrt{n} \pi_V (f-z)$.
\end{theorem}

\begin{proof}
Again let $V = \text{span} \{ e_k : 1 \leq k \leq J \}$. Using Fubini's theorem and that the total variation distance is bounded by 1,
\begin{equation*}
\begin{split}
\snorm{ \Pi^H & (\cdot \mid Y) \circ T_\mathbb{Y}^{-1} - N(0,I_J) }_{TV} \\
& =  \frac{1}{2} \int_{\R^J} \left| \int_0^\infty \lambda(  \alpha \mid Y ) d\Pi_\alpha (\cdot \mid Y) \circ T_{\mathbb{Y}}^{-1}(x) d\alpha - dN(0,I_J)(x) \right|  \\
& \leq \frac{1}{2} \int_0^\infty \lambda (\alpha \mid Y) \int_{\R^J} \left| d\Pi_\alpha (\cdot \mid Y) \circ T_{\mathbb{Y}}^{-1}(x) d\alpha - dN(0,I_J)(x) \right|  d\alpha  \\
& =  \int_0^\infty \lambda(\alpha \mid Y) \snorm{\Pi_\alpha (\cdot \mid Y) \circ T_\mathbb{Y}^{-1} - N(0,I_J)}_{TV} d\alpha  \\
& \leq  \sup_{0 < \leq \alpha \leq \overline{\alpha}_n} \snorm{\Pi_\alpha (\cdot \mid Y) \circ T_\mathbb{Y}^{-1} - N(0,I_J)}_{TV} \int_0^{\overline{\alpha}_n}  \lambda(\alpha \mid Y)  d\alpha  \\
& \quad \quad + \int_{\overline{\alpha}_n}^\infty \lambda (\alpha \mid Y) d\alpha  ,
\end{split}
\end{equation*}
where $\overline{\alpha}_n$ is defined in Section \ref{other paper section}. The first term is $o_{\P_0} (1)$ by the same argument as in the proof of Theorem \ref{EB parametric BvM theorem} and the second term is $o_{\P_0}(1)$ by the proof of Theorem 3 of \cite{KnSzVVVZ}. Since the total variation distance is bounded, convergence in $\P_0$-probability is equivalent to convergence in $L^1 (\P_0)$.
\end{proof}

\begin{proof}[Proof of Theorem \ref{HB weak BvM}]
The proof is exactly the same as that of Theorem \ref{EB weak BvM}, using Theorems \ref{thm2} and \ref{parametric BvM theorem} instead of Theorems \ref{thm1} and \ref{EB parametric BvM theorem}.
\end{proof}

\subsection{Proof of weak BvM result in $L^\infty$ (Theorem \ref{thm3})}

Following Theorem 3.1 of \cite{HoRoSc}, define the sets
\begin{equation*}
\mathcal{J}_n (\gamma) = \left\{ (j,k) \in \Lambda : |f_{0,jk}| > \gamma \sqrt{\log n / n}  \right\}
\end{equation*}
for $\gamma > 0$. In what follows, we denote by $S$ the support of the prior draw, that is the set of non-zero coefficients of $f = (f_{jk})_{(j,k) \in \Lambda}$ drawn from the prior. We require the following contraction result.

\begin{theorem}\label{slab and spike contraction}
Consider the slab and spike prior defined in Section \ref{linfty setting} with lower threshold given by the strictly increasing sequence $j_0 (n) \rightarrow \infty$. Then for every $0 < \beta_{min} \leq \beta_{max}$, $R>0$ and $M_n \rightarrow \infty$, we have
\begin{equation*}
\sup_{f_0 \in \mathcal{H} (\beta,R)}  \E_0 \Pi ( f : \norm{f-f_0}_{\mathcal{M} (w) } \geq M_n n^{-1/2} \mid Y)  \rightarrow 0  
\end{equation*}
uniformly over $\beta \in [\beta_{min},\beta_{max}]$, where $(w_l)$ is any admissible sequence satisfying $w_{j_0 (n)} \geq c \sqrt{\log n}$ for some $c>0$.
\end{theorem}

\begin{proof}[Proof of Theorem \ref{slab and spike contraction}]
Fix $\eta >0$. Consider the event
\begin{equation}
A_n = \{ S^c \cap \mathcal{J}_n (\overline{\gamma}) =  \emptyset  \} \cap \{ S \cap \mathcal{J}_n^c (\underline{\gamma}) = \emptyset \} \cap \{ \max_{(j,k) \in \mathcal{J}_n (\underline{\gamma} ) } |f_{0,jk} - f_{jk}| \leq \overline{\gamma} \sqrt{ (\log n ) / n} \}.
\label{eq9}
\end{equation}
By Theorem 3.1 of \cite{HoRoSc}, there exist constants $0 < \underline{\gamma} < \overline{\gamma} < \infty$ (independent of $\beta$ and $R$) such that
\begin{equation}
\sup_{f_0 \in \cup_{\beta \in [\beta_{min},\beta_{max}]} \mathcal{H}(\beta,R) } \E_0 \Pi (A_n^c \mid Y) \lesssim n^{-B} ,
\label{eq23}
\end{equation}
for some $B = B(\beta_{min},\beta_{max},R)>0$ (this follows since the probabilities of the complements of each of the events constituting $A_n$ satisfy the above bound individually). We then have the following decomposition for some $D = D(\eta) >0$ large enough to be specified later,
\begin{equation}
\begin{split}
\E_0  \Pi &  \left( \norm{f - f_0}_\mathcal{M} \geq M_n n^{-1/2} \mid Y \right) \\
& \leq  \E_0 \Pi \left( \{ \norm{f - f_0  }_\mathcal{M} \geq M_n n^{-1/2} \} \cap \{ \norm{ \pi_{j_0} (f - f_0) }_\mathcal{M} \leq D n^{-1/2} \} \cap A_n \mid Y \right) \\
& \quad +  \E_0 \Pi \left( \{ \norm{f - f_0}_\mathcal{M} \geq M_n n^{-1/2} \} \cap \{ \norm{ \pi_{j_0} (f - f_0) }_\mathcal{M} >  D n^{-1/2} \} \cap A_n \mid Y \right) \\ 
& \quad + E_0 \Pi (A_n^c \mid Y)  .
\end{split}
\label{eq10}
\end{equation}
Note that the first term on the right-hand side of \eqref{eq10} is bounded by
\begin{equation*}
\E_0 \Pi (  \{ \norm{ \pi_{>j_0} (f - f_0) }_\mathcal{M}  \geq (M_n - D) n^{-1/2} \} \cap A_n \mid Y )  .
\end{equation*}
Combining this with \eqref{eq23}, we can upper bound the right hand side of \eqref{eq10} by
\begin{equation}
\begin{split}
\E_0 \Pi & ( \{ \norm{ \pi_{>j_0} (f - f_0) }_\mathcal{M}  \geq \tilde{M}_n n^{-1/2} \} \cap A_n \mid Y) \\
&  + \E_0 \Pi (  \norm{ \pi_{j_0} (f - f_0) }_\mathcal{M} > D n^{-1/2}  \mid Y) + o(1) ,
\label{eq24}
\end{split}
\end{equation}
where $\tilde{M}_n = M_n - D \rightarrow \infty$ as $n \rightarrow \infty$. We bound the two remaining terms in \eqref{eq24} separately.

For the first term in \eqref{eq24}, we can proceed as in the proof of Theorem 3.1 of \cite{HoRoSc}. By the definition of the H\"older ball $\mathcal{H}(\beta,R)$, there exists $J_n (\beta)$ such that $2^{J_n (\beta)} \leq k (n / \log n)^{1/(2\beta+1)}$ for some constant $k >0$ such that $\mathcal{J}_n (\underline{\gamma}) \subset \{ (j,k): j \leq J_n (\beta), \, \, k=0,...,2^j-1 \}$ and 
\begin{equation*}
\sup_{f_0 \in \mathcal{H} (\beta,R)} \sup_{l > J_n(\beta)} w_l^{-1} \max_k |f_{0,lk}| \leq \frac{ R 2^{-J_n (\beta) (\beta +1/2)} }{ \sqrt{J_n(\beta)} }  \leq C( \beta, R) \frac{1}{\sqrt{n}} .
\end{equation*}
Consider now the frequencies $j_0 < l \leq J_n (\beta)$. On the event $A_n$, we have that
\begin{equation*}
\sup_{j_0 < l \leq J_n (\beta) } \frac{1}{w_l} \max_k | f_{lk} - f_{0,lk} | \leq \frac{1}{ w_{j_0} } \overline{\gamma} \sqrt{\frac{\log n }{n} } \leq  \frac{ \overline{\gamma} }{c} \frac{1}{\sqrt{n}} ,
\end{equation*}
since $w_{j_0 (n)} \geq c \sqrt{\log n}$ by hypothesis. We thus have that on the event $A_n$, $\norm{ \pi_{>j_0}  (f-f_0) }_\mathcal{M} = O(n^{-1/2})$ for any $f_0 \in \mathcal{H} (\beta,R)$, which proves that the first term in \eqref{eq24} is 0 for $n$ sufficiently large.

Consider now the second term in \eqref{eq24}. We shall use the approach of \cite{Ca} using the moment generating function to control the low frequency terms. Recall that on these coordinates we have the simple product prior $\Pi(dx_1,...,dx_{j_0} )= \prod_{k=1}^{j_0} g(x_i) dx_i $. Let $\E^\Pi (\cdot \mid Y)$ denote the expectation with respect to the posterior measure. Following Lemma 1 of \cite{Ca}, we have the subgaussian bound
\begin{equation*}
\E_0 \E^\Pi ( e^{ t \sqrt{n} (f_{lk} - Y_{lk} ) } \mid Y ) \leq C e^{t^2 /2}
\end{equation*}
for some some $C>0$. Using this and proceeding as in the proof of Theorem 4 of \cite{CaNi2} yields
\begin{equation*}
\sqrt{n} \E_0 \E^\Pi \left( \norm{\pi_{j_0} (f - Y) }_{\mathcal{M}} \mid Y \right) =  \E_0 \E^\Pi  \left( \sup_{j \leq j_0} l^{-1/2} \max_k \sqrt{n} | f_{lk} - Y_{lk} | \mid Y \right) \leq C,
\end{equation*}
for some $C >0$. By Markov's inequality and then the triangle inequality, the second term in \eqref{eq24} is then bounded by
\begin{equation}
\begin{split}
\frac{\sqrt{n}}{D} \E_0 \E^\Pi \left( \norm{ \pi_{j_0} (f - f_0) }_\mathcal{M}  \mid Y \right) & \leq   \frac{\sqrt{n}}{D} \E_0 \E^\Pi \left( \norm{\pi_{j_0} (Y - f_0) }_\mathcal{M} \mid Y \right) + \frac{C}{D}  \\
& \leq \frac{\E_0 \snorm{\mathbb{Z}}_\mathcal{M}}{D}  + \frac{C}{D}.
\label{eq12}
\end{split}
\end{equation}
By Proposition 2 of \cite{CaNi2} and the fact that $(w_l)$ is an admissible sequence, the first term in \eqref{eq12} is also bounded by $C'/D$ for some $C' >0$. Taking $D = D(\eta) > 0$ sufficiently large, \eqref{eq12} can be then made smaller than $\eta /2$.
\end{proof}

\begin{proof}[Proof of Theorem \ref{thm3}]
Fix $\eta > 0$ and denote $\tilde{\Pi}_n = \Pi_n \circ \tau_\mathbb{Y}^{-1}$. By the triangle inequality, uniformly over the relevant class of functions,
\begin{equation*}
\beta_{\mathcal{M}_0} ( \tilde{\Pi}_n , \mathcal{N} ) \leq \beta_{\mathcal{M}_0} ( \tilde{\Pi}_n , \tilde{\Pi}_n \circ \pi_j^{-1} ) + \beta_{\mathcal{M}_0} ( \tilde{\Pi}_n \circ \pi_j^{-1} , \mathcal{N} \circ \pi_j^{-1} ) + \beta_{\mathcal{M}_0} ( \mathcal{N} \circ \pi_j^{-1} , \mathcal{N} ) ,
\end{equation*}
for fixed $j > 0$. Since we have a $1/\sqrt{n}$-contraction rate in $\mathcal{M}$ for the posterior from Theorem \ref{slab and spike contraction}, we can make the $\E_0$-expectation of the first term smaller than $\eta /3$ by taking $j$ sufficiently large, again using the arguments of Theorem 8 of \cite{CaNi}. We recall from the proof of that theorem that if the required contraction is established in $\mathcal{M}(\overline{w})$ for an admissible sequence $(\overline{w}_l)$, then the required tightness argument holds in $\mathcal{M}_0 (w)$ for any admissible $(w_l)$ such that $w_l / \overline{w}_l \nearrow \infty$. A similar result holds for the third term.

For the middle term, note that $j_0 (n) \geq j$ for $n$ large enough. For such $n$, the projected prior onto the first $j$ coordinates is a simple product prior which satisfies the usual conditions of the parametric BvM, namely it has a density that is positive and continuous at the true (projected) parameter (see Chapter 10 of \cite{VV} for more details). Since the total variation distance dominates the bounded Lipschitz metric, this completes the proof.
\end{proof}

\subsection{Credible sets}

\subsubsection*{$\ell_2$ confidence sets}

\begin{proof}[Proof of Proposition \ref{EB credible set}]
By Lemma \ref{coverage lemma} and the definition of $\tilde{C}_n$, we have
\begin{equation*}
\sup_{\substack{f_0 \in \mathcal{Q}_{SS} (\beta,R,\varepsilon) \\\ \beta \in [\beta_{1},\beta_{2}] }} |\P_0 (f_0 \in \tilde{C}_n) - \P_0 ( \snorm{f_0 - \mathbb{Y} }_H \leq R_n / \sqrt{n} )| \rightarrow 0
\end{equation*}
as $n \rightarrow \infty$, so that it is sufficient to show that the second probability in the above display tends to $1-\gamma$, uniformly over the relevant self-similar Sobolev balls. This follows directly by Theorem 1 of \cite{CaNi} (an examination of that proof shows that the convergence holds uniformly over the parameter space as long as the weak BvM itself holds uniformly, as is the case here), thereby establishing the required coverage statement.

Since $\epsilon_n \geq r_1/\log n$ and $C > 1/r_1$ in the definition \eqref{EB credible set definition} of $\tilde{C}_n$, applying Lemma \ref{exponential lemma} (with $\eta = 1)$ yields the inequality
\begin{equation*}
\Pi_{\hat{\alpha}_n} \big( \snorm{f - \hat{f}_n}_{H^{\hat{\alpha}_n - \epsilon_n}} \geq C \sqrt{\log n} \big| Y \big) \lesssim \exp \left( -C' (\log n) n^{\frac{1}{4\hat{\alpha}_n + 2}} \right),
\end{equation*}
where $C' >0$ does not depend on $\hat{\alpha}_n$. Since by Lemma \ref{smoothness lemma} we have $\hat{\alpha}_n \leq \beta$ for large enough $n$ with $\P_0$-probability tending to 1, the right-hand side is bounded by a multiple of $\exp(-C'' (\log n) n^{1/(4\beta+2)})$ with the same probability. The completes the credibility statement.

Let $f_1,f_2 \in \tilde{C}_n$ and set $g = f_1 - f_2$. Picking $J_n \sim  [n / (\log n)^{2\delta-1} ]^{1/( 1 + 2\hat{\alpha}_n - 2\epsilon_n )}$ yields
\begin{equation*}
\begin{split}
\norm{g}_2^2  = \sum_{k=1}^\infty | g_k|^2 & = \sum_{k=1}^{J_n} k k^{-1} (\log k)^{2\delta - 2\delta} |g_k|^2   +  \sum_{k=J_n+1}^\infty k^{2(\hat{\alpha}_n - \epsilon_n )  - 2(\hat{\alpha}_n - \epsilon_n ) } |g_k|^2 \\
& \leq J_n (\log J_n)^{2\delta} \norm{g}_{H(\delta)}^2 +  J_n^{-2 ( \hat{\alpha}_n - \epsilon_n ) } \norm{g}_{H^{\hat{\alpha}_n - \epsilon_n } }^2   \\
& = O_{\P_0} \left(  J_n (\log J_n)^{2\delta} n^{-1}   +   J_n^{-2 ( \hat{\alpha}_n - \epsilon_n ) } (\log n)  \right)  \\
& = O_{\P_0} \left(  n^{-\frac{  2(\hat{\alpha}_n - \epsilon_n) }{ 1 + 2\hat{\alpha}_n - 2\epsilon_n }} ( \log n)^{ \frac{ 4\delta  ( \hat{\alpha}_n - \epsilon_n )  +1 } { 1 + 2\hat{\alpha}_n -2\epsilon_n } }  \right)  ,
\end{split}
\end{equation*}
where the constants do not depend on $g$. Since $| \hat{\alpha}_n - \beta| = O_{\P_0} (1 / \log n)$ by Lemma \ref{smoothness lemma} and $\epsilon_n = O(1/\log n)$ by assumption, some straightforward computations yield that $\norm{g}_2^2 = O_{\P_0} (n^{-2\beta / (2\beta + 1)} (\log n)^{ (4\delta \beta +1) / (2\beta  + 1) }  )$ as $n \rightarrow \infty$.
\end{proof}

\begin{proof}[Proof of Proposition \ref{HB credible set}]
The proof follows in the same way as that of Proposition \ref{EB credible set}, using Lemma \ref{posterior median lemma} and an analogue of Lemma \ref{coverage lemma}. The only difference is for the credibility statement, where we no longer have an exponential inequality like Lemma \ref{exponential lemma}. However, arguing as in \cite{KnSzVVVZ} with the $H^{\hat{\beta}_n}$-norm instead of the $\ell_2$-norm, one can show that under self-similarity the posterior contracts about the posterior mean at rate $\tilde{M}_n \sqrt{\log n}$ for any $\tilde{M}_n \rightarrow \infty$ (see also Theorem 1.1 of \cite{SzVVVZ4}). It then follows that the second constraint in \eqref{credible set definition 2} is satisfied with credibility $1-o_{\P_0}(1)$.
\end{proof}

\subsubsection*{$L^\infty$ confidence bands}

\begin{proof}[Proof of Proposition \ref{credible band}]
By Lemma \ref{median lemma}, it suffices to prove all the results on the event $B_n$ defined in \eqref{eq19}. We firstly establish the diameter of the confidence set. Recall that $\pi_{med}$ denotes the projection onto the non-zero coordinates of the posterior median and for a set of coordinates $E$, let $\pi_E$ denote the projection onto $\text{span}(E)$. Taking $f_1,f_2 \in \overline{D}_n$ and setting $2^{J_n (\beta)} \simeq (n / \log n)^{1/(2\beta+1)}$, we have on $B_n$,
\begin{equation*}
\begin{split}
\norm{f_1 - f_2 }_{\infty} & \leq \norm{f_1 - \pi_{med}(Y) }_{\infty} + \norm{f_2 - \pi_{med}(Y) }_{\infty}   \\
& \leq 2 \sup_{x \in [0,1]}  \sum_{l=0}^{J_n (\beta)} \sum_{k=0 }^{2^l-1}  v_n \sqrt{\frac{\log n }{n}}  | \psi_{lk} (x) |  \\
& \leq C(\psi) v_n \sqrt{\frac{\log n }{n}}  \sum_{l=0}^{J_n (\beta)} 2^{l/2}  \leq C' v_n \sqrt{\frac{2^{J_n(\beta)} \log n}{n}}  = O_{\P_0} \left( \left(  \frac{\log n}{n} \right)^{ \frac{\beta}{2\beta+1} }  v_n  \right)  .
\end{split}
\end{equation*}

We now establish asymptotic coverage. Split $f_0 = \pi_{ \mathcal{J}_n (\underline{\gamma} ) } (f_0) + \pi_{ \mathcal{J}_n^c (\underline{\gamma}) }  (f_0)$. Since $\pi_{ \mathcal{J}_n^c (\underline{\gamma}) } \circ \pi_{med} (Y) = 0$ on $B_n$, we can write
\begin{equation}
\snorm{f_0 - \pi_{med}(Y)}_\infty \leq \snorm{ \pi_{med} (f_0 - Y) }_{\infty}   +    \snorm{ (id - \pi_{med} ) \circ \pi_{ \mathcal{J}_n  ( \underline{\gamma} ) } (f_0)    }_{\infty} + \snorm{ \pi_{ \mathcal{J}_n^c (\underline{\gamma}) }  (f_0)  }_{\infty} ,
\label{eq15}
\end{equation}
where $id$ denotes the identity operator. For the third term in \eqref{eq15}, note that since $f_0 \in \mathcal{H} (\beta,R)$,
\begin{equation}
\begin{split}
\snorm{ \pi_{ \mathcal{J}_n^c (\underline{\gamma}) }  (f_0)  }_{\infty} & \leq \sum_{l=0}^\infty 2^{l/2} \max_{ k: (l,k) \in \mathcal{J}_n^c (\underline{\gamma}) } | \langle f_0, \psi_{lk} \rangle |  \\
& \leq \sum_{l=0}^{J_n (\beta) } 2^{l/2} \underline{\gamma} \sqrt{\frac{\log n}{n} }  + \sum_{l > J_n (\beta) } 2^{-l\beta}  \leq C (\beta, R) \left( \frac{\log n}{n}  \right)^{ \frac{\beta}{2\beta + 1} } .
\end{split}
\label{eq16}
\end{equation}
For the second term in \eqref{eq15}, we note that any indices remaining satisfy $(l,k) \in \mathcal{J}_n^c (\overline{\gamma}')$ and so by the same reasoning as above, this term is also $O ( (\log n / n)^{\beta/(2\beta +1)} )$.

By the proof of Proposition 3 of \cite{HoNi}, we have that for $f_0 \in \mathcal{H}_{SS} (\beta,R,\varepsilon)$,
\begin{equation*}
\sup_{(l,k) : l \geq j} | \langle f_0 , \psi_{lk} \rangle | \geq d(b,R,\beta,\psi) 2^{-j (\beta + 1/2) } .
\end{equation*}
Let $\tilde{J}_n (\beta)$ be such that $\frac{\epsilon}{2}  (n / \log n)^{1/(2\beta+1) } \leq 2^{\tilde{J}_n (\beta) } \leq \epsilon  (n / \log n)^{1/(2\beta+1) }$, where $\epsilon = \epsilon (b,R,\beta,\psi) > 0 $ is small enough so that $d/\epsilon^{\beta + 1/2} > \overline{\gamma}'$. Using this yields
\begin{equation*}
\sup_{(l,k) : l \geq \tilde{J}_n (\beta) } | \langle f_0 , \psi_{lk} \rangle | \geq \frac{ d(b,R,\beta,\psi) }{ \epsilon^{\beta + 1/2} } \sqrt{ \frac{\log n}{n} }  > \overline{\gamma }' \sqrt{ \frac{\log n}{n} }  .
\end{equation*}
We therefore have that on the event $B_n$, there exists $(l',k')$ with $l' \geq \tilde{J}_n (\beta)$ such that $\tilde{f}_{l'k'} \neq 0$ and a non-zero coefficient therefore appears in the definition \eqref{eq18} of $\sigma_{n,\gamma}$. We can thus lower bound
\begin{equation}
\sigma_{n,\gamma}   \geq v_n \sqrt{\frac{\log n }{n}} \sup_{x\in [0,1]}   | \psi_{l'k'} (x) |     \geq c(\psi) v_n \sqrt{\frac{2^{\tilde{J}_n(\beta)} \log n }{n}}  = c'  v_n   \left( \frac{\log n}{n} \right)^{\frac{\beta}{2\beta + 1}}  .  
\label{eq26}
\end{equation}
Now, since $v_n \rightarrow \infty$ as $n \rightarrow \infty$, we have from \eqref{eq16} and the remark after it that for sufficiently large $n$ (depending on $\beta$ and $R$), the last two terms in \eqref{eq15} satisfy
\begin{equation*}
 \snorm{ (id - \pi_{med} ) \circ \pi_{ \mathcal{J}_n  ( \underline{\gamma} ) } (f_0)    }_{\infty} + \snorm{ \pi_{ \mathcal{J}_n^c (\underline{\gamma}) }  (f_0)  }_{\infty}   \leq C \left( \frac{\log n}{n} \right)^{\frac{\beta}{2\beta + 1}} \leq  \sigma_{n,\gamma} /2 .
\end{equation*}

For the first term in \eqref{eq15} we recall that on $B_n$, the posterior median only picks up coefficients $(l,k)$ with $l \leq J_n (\beta) \leq J_n$. Therefore on this event,
\begin{equation*}
\begin{split}
\snorm{\pi_{med}(f_0 - Y)}_\infty & \leq \sup_{x \in [0,1]} \sum_{(l,k): \tilde{f}_{lk} \neq 0} |f_{0,lk} - Y_{lk} | | \psi_{lk} (x) |  \\
& \leq C(\psi) \sqrt{ \frac{\log n}{n} } \sum_{(l,k) : l \leq J_n (\beta)} 2^{l/2}  \leq C' \left( \frac{\log n}{n} \right)^{ \frac{1}{2\beta + 1} }  .
\end{split}
\end{equation*}
Using the lower bound \eqref{eq26}, we deduce that on $B_n$, $\norm{ \pi_{med} (f_0 - Y)}_\infty \leq  \sigma_{n,\gamma}(Y) /2$ for $n$ large enough, uniformly over $f_0 \in \mathcal{H}_{SS} (\beta,R)$. Combining all of the above yields that $B_n \subset  \{ \snorm{f_0 - \pi_{med}(Y) }_\infty \leq \sigma_{n,\gamma} \}$. We therefore conclude that
\begin{equation*}
\begin{split}
\P_0 (f_0 \in \overline{D}_n ) & = \P_0 ( \{ \snorm{ f_0 - \mathbb{Y}}_{\mathcal{M}(w)} \leq R_n / \sqrt{n} \} \cap \{ \snorm{f_0 - \pi_{med}(Y)}_\infty \leq  \sigma_{n,\gamma} \} \cap B_n ) + o(1)  \\
& =  \P_0 ( \{ \snorm{ f_0 - \mathbb{Y}}_{\mathcal{M}(w)} \leq R_n / \sqrt{n} \} \cap B_n ) + o(1)  \\
& = 1 - \gamma + o(1)  ,
\end{split}
\end{equation*}
where we have used that $\P_0 (B_n) \rightarrow 1$ and that $\P_0 (\snorm{f_0 - \mathbb{Y}}_{\mathcal{M}(w)} \leq R_n / \sqrt{n}) \rightarrow 1 - \gamma$ by Theorem 5 of \cite{CaNi2}. Noting finally that both of these probabilities converge uniformly over the relevant self-similar Sobolev balls, the coverage statement also holds uniformly as required.

For the credibility statement it suffices to show that the second constraint in \eqref{eq14} is satisfied with posterior probability tending to 1. Again using that $\norm{ \pi_{med} (f_0 - Y)}_\infty \leq  \sigma_{n,\gamma}(Y) /2$ on $B_n$ as well as \eqref{eq26}, we have that uniformly over $f_0 \in \mathcal{H}(\beta,R)$,
\begin{equation*}
\begin{split}
& \E_0 \Pi ( f : \snorm{f - \pi_{med}(Y)}_\infty \geq \sigma_{n,\gamma} \mid Y)   \\
& \quad \leq  \E_0   \Pi ( f : \snorm{f - f_0}_\infty \geq \sigma_{n,\gamma}/2 \mid Y) + \E_0 \Pi ( f : \snorm{f_0 - \pi_{med}(Y)}_\infty \geq \sigma_{n,\gamma}/2 \mid Y)  \\
& \quad \leq \E_0   \Pi ( f : \snorm{f - f_0}_\infty \geq c' v_n ((\log n)/n)^{\beta/(2\beta+1)}/2  \mid Y) +  \P_0 (B_n^c) \rightarrow 0 
\end{split}
\end{equation*}
since the posterior contracts at rate $(\log n/n)^{\beta/(2\beta+1)}$ by Theorem 3.1 of \cite{HoRoSc}.
\end{proof}

\subsection{Posterior independence of the credible sets}

\begin{proof}[Proof of Theorem \ref{l2 asymptotic independence}]
We first consider the fixed-regularity prior $\Pi_\alpha$ with $\alpha \in [0,a_n]$, replacing the sets $\tilde{C}_n$ and $C_n^{\ell_2}$ respectively by the $(1-\gamma)$-$H(\delta)$-credible ball $C_n^{(\alpha)}$ and the $(1-\gamma)$-$\ell_2$-credible ball $C_n^{(\alpha,\ell_2)}$ for $\Pi_\alpha (\cdot \mid Y)$ (i.e. \eqref{eq29} and \eqref{l2 credible set} for $\Pi_\alpha (\cdot \mid Y)$ rather than $\Pi_{\hat{\alpha}_n} (\cdot \mid Y)$). By the definition of the posterior distribution \eqref{eq28}, we can write a posterior draw $f \sim \Pi_\alpha (\cdot \mid Y)$ as
\begin{equation*}
f - \hat{f}_{n,\alpha} = \sum_{k=1}^\infty \frac{1}{\sqrt{k^{2\alpha+1}+n}} \zeta_k e_k ,
\end{equation*}
where $\zeta_k \sim N(0,1)$ are independent and $\hat{f}_{n,\alpha}$ is the posterior mean. Let $k_n \rightarrow \infty$ be some sequence satisfying $k_n = o(n^{1/(4\alpha+2)})$. We shall prove the result by showing that the $H(\delta)$-credible ball is determined by the frequencies $k\leq k_n$, while the $\ell_2$-credible ball is determined by the frequencies $k\geq k_n$. We therefore decompose both credible balls according to the threshold $k_n$.

By Lemma 1 of \cite{LaMa}, which can be adapted to the case $D=\infty$, and some elementary computations, we have the following exponential inequalities for any $x \geq 0$:
\begin{equation}
\P \left( \sum_{k=k_n+1}^\infty  \frac{\zeta_k^2}{k(\log k)^{2\delta}} \geq \frac{C (\delta)}{(\log k_n)^{2\delta}} \left(  \log k_n  + \sqrt{\frac{x}{k_n}} +  \frac{x}{k_n } \right) \right) \leq e^{-x} ,
\label{exp inequality 1}
\end{equation}
\begin{equation}
\P \left( \sum_{k=1}^{k_n}  \zeta_k^2 \geq k_n + 2\sqrt{k_n x} + 2x \right) \leq e^{-x} ,
\label{exp inequality 2}
\end{equation}
\begin{equation}
\P \left( \sum_{k=k_n+1}^\infty \frac{\zeta_k^2}{k^{2\alpha+1}+n} \leq n^{-\frac{2\alpha}{2\alpha+1}} \left( \frac{1}{4\alpha}  - 2^{3/2} n^{-\frac{1}{4\alpha+2}} \sqrt{x}  \right)  \right) \leq e^{-x},
\label{exp inequality 3}
\end{equation}
where $C(\delta)<\infty$ for $\delta > 1/2$. Moreover, note that
\begin{equation}
\P \left( \sum_{k=1}^{k_n} \frac{\zeta_k^2}{(k^{2\alpha+1}+n) k (\log k)^ {2\delta-1}} \leq \frac{x}{n} \right) \leq \P \left( \zeta_1^2/2 \leq x \right) \leq \sqrt{x/\pi},
\label{exp inequality 4}
\end{equation}
using that $\zeta_1$ is standard normal. Define the event 
\begin{equation*}
\begin{split}
\tilde{A}_{n,\alpha} &= \left\{  \sum_{k=k_n+1}^\infty \frac{\zeta_k^2}{k(\log k)^{2\delta}} \leq \frac{3C(\delta)}{(\log k_n)^{2\delta-1}} \right\} \cap \left\{ \sum_{k=1}^{k_n} \zeta_k^2 \leq 5k_n  \right\}  \\
& \quad \quad \cap \left\{ \sum_{k=k_n+1}^\infty \frac{\zeta_k^2}{k^{2\alpha+1}+n} \geq \frac{1}{8a_n} n^{-\frac{2\alpha}{2\alpha+1}} \right\} \\
& \quad \quad \cap \left\{ \sum_{k=1}^{k_n} \frac{\zeta_k^2}{(k^{2\alpha+1}+n) k (\log k)^ {2\delta-1}} \geq \frac{1}{n (\log k_n)^{\delta-1/2} }\right\}.
\end{split}
\end{equation*}
Setting $x = k_n$ in the inequalities \eqref{exp inequality 1} and \eqref{exp inequality 2}, $x = 2^{-9} a_n^{-2} n^{1/(2a_n+1)}$ in \eqref{exp inequality 3} and $x=(\log k_n)^{-(\delta-1/2)}$ in \eqref{exp inequality 4} yields
\begin{equation*}
\sup_{\alpha\in[0,a_n]} \Pi_\alpha (\tilde{A}_{n,\alpha}^c \mid Y) \leq  2e^{-k_n} + e^{-2^{-9} a_n^{-2} n^{1/(2a_n+1)} } + (\log k_n)^{-(2\delta-1)/4}/\sqrt{\pi} \rightarrow 0
\end{equation*}
as $n\rightarrow \infty$, since $a_n \leq \log n/(6\log \log n)$ by assumption. We have that on $\tilde{A}_{n,\alpha}$,
\begin{equation*}
\frac{1}{8a_n} n^{-\frac{2\alpha}{2\alpha+1}} \leq \snorm{f - \hat{f}_{n,\alpha}}_2^2 \leq \frac{5k_n}{n} +  \sum_{k=k_n+1}^\infty \frac{\zeta_k^2}{k^{2\alpha+1}+n}  ,
\end{equation*}
\begin{equation*}
\frac{1}{n (\log k_n)^{\delta-1/2}} \leq \snorm{f - \hat{f}_{n,\alpha}}_{H(\delta)}^2  \leq \sum_{k=1}^{k_n} \frac{\zeta_k^2}{(k^{2\alpha+1}+n)k (\log k)^{2\delta}} + \frac{3C(\delta)}{n(\log k_n)^{2\delta-1}}  .
\end{equation*}

Recall that $C_n^{(\alpha)}$ has radius equal to $R_n/\sqrt{n}$, where $R_n (Y,\gamma) \rightarrow^{\P_0} R(\gamma)>0$ by Theorem 1 of \cite{CaNi}. Similarly, $C_n^{(\alpha,\ell_2)}$ has radius $Q_n n^{-\alpha/(2\alpha+1)}$, where $Q_n \rightarrow Q >0$ by Theorem 1 of \cite{Fr}. Using these facts, the above bounds and the definition of $k_n$, the probability $\Pi_\alpha (C_n^{(\alpha)} \cap C_n^{(\alpha,\ell_2)} \mid Y)$ equals
\begin{equation*}
\begin{split}
&  \Pi_\alpha \Big(  \Big\{ \snorm{f-\hat{f}_n}_{H(\delta)}^2  \leq \frac{R_n^2}{n}, \quad   \snorm{f - \hat{f}_n}_2^2  \leq Q_n^2 n^{-\frac{2\alpha}{2\alpha+1}} \Big\} \cap \tilde{A}_n \Big| Y \Big) + o(1)  \\
& = \Pi_\alpha \bigg( \bigg\{ \sum_{k=1}^{k_n} \frac{\zeta_k^2}{(k^{2\alpha+1}+n)k (\log k)^{2\delta}}  \leq \frac{1}{n}\Big( R_n^2 + O \left( (\log k_n)^{-(2\delta-1)} \right) \Big)   , \\
& \quad \quad \quad \quad  \sum_{k=k_n+1}^\infty \frac{\zeta_k^2}{k^{2\alpha+1}+n}  \leq \left( Q_n^2 + O\left( k_n n^{-\frac{1}{2\alpha+1}}  \right) \right) n^{-\frac{2\alpha}{2\alpha+1}} \bigg\} \cap \tilde{A}_n  \bigg| Y \bigg) + o(1) \\
& =  \Pi_\alpha \bigg(  \sum_{k=1}^{k_n} \frac{\zeta_k^2}{(k^{2\alpha+1}+n)k (\log k)^{2\delta}}  \leq \frac{R_n^2 + o(1)}{n} ,   \sum_{k=k_n+1}^\infty \frac{\zeta_k^2}{k^{2\alpha+1}+n}  \leq  \frac{Q_n^2 + o(n^{-\frac{1}{4\alpha+2}})}{n^{\frac{2\alpha}{2\alpha+1}}} \bigg| Y \bigg) + o(1) \\
& =  \Pi_\alpha \left(  \sum_{k=1}^{k_n} \frac{\zeta_k^2}{(k^{2\alpha+1}+n)k (\log k)^{2\delta}}  \leq \frac{R_n^2 + o(1)}{n} \Bigg| Y \right)   \\
& \quad \quad \times  \Pi_\alpha \left( \sum_{k=k_n+1}^\infty \frac{\zeta_k^2}{k^{2\alpha+1}+n}  \leq \frac{Q_n^2 + o(n^{-\frac{1}{4\alpha+2}})}{n^{\frac{2\alpha}{2\alpha+1}}}  \Bigg| Y \right) + o(1) ,
\end{split}
\end{equation*}
where in the last line we have used the independence of the coordinates under the posterior. Using again the inequalities \eqref{exp inequality 1}-\eqref{exp inequality 4}, the final line equals
\begin{equation*}
\begin{split}
& \Pi_\alpha \left( \left. \snorm{f-\hat{f}_{n,\alpha}}_{H(\delta)}^2  \leq (R_n^2 + o(1))/n \right| Y \right) \\
& \quad \times \Pi_\alpha \left( \left. \snorm{f - \hat{f}_{n,\alpha}}_2^2 \leq \left( Q_n^2 + o(n^{-\frac{1}{4\alpha+2}})  \right) n^{-\frac{2\alpha}{2\alpha+1}} \right| Y \right) + o(1)\\
& =: \Pi_{\alpha,n}^{(1)} \times \Pi_{\alpha,n}^{(2)} + o(1).
\end{split}
 \end{equation*}
Since $\sup_{\alpha\in[0,a_n]} \Pi_\alpha (\tilde{A}_{n,\alpha}^c \mid Y)\rightarrow 0$, the previous display holds uniformly over $\alpha \in [0,a_n]$ so that we have shown
\begin{equation}
\sup_{\alpha \in [0,a_n]} \left| \Pi_\alpha (C_n^{(\alpha)} \cap C_n^{(\alpha,\ell_2)} \mid Y) - \Pi_{\alpha,n}^{(1)} \times \Pi_{\alpha,n}^{(2)}  \right| = o(1) .
\label{eq_prod}
\end{equation}

For the full empirical Bayes posterior, note that the second constraint in \eqref{EB credible set definition} is satisfied with posterior probability $1 - o_{\P_0} (1)$ uniformly over $f_0 \in \mathcal{Q}(\beta,R)$ by the proof of Proposition \ref{EB credible set}, so that it suffices to prove the theorem with $C_n$ in \eqref{eq29} instead of $\tilde{C}_n$. Since $a_n=o(\log n)$, we can take $k_n \rightarrow \infty$ such that $k_n = o(n^{1/(4a_n+2)})$, from which also $k_n = o(n^{1/(4\alpha+2)})$ for all $\alpha \in [0,a_n]$, the interval over which $\hat{\alpha}_n$ ranges. By \eqref{eq_prod}, it therefore remains to check that $\Pi_{\hat{\alpha}_n,n}^{(1)}, \Pi_{\hat{\alpha}_n,n}^{(2)} \rightarrow^{\P_0} 1-\gamma$ as $n\rightarrow \infty$.

For this we require a finer understanding of the posterior behaviour of the norms in the second to last display. Consider firstly $\Pi_{\hat{\alpha}_n,n}^{(1)}$. By Theorem \ref{EB weak BvM} and Lemma \ref{l2 efficient estimator}, $\Pi_{\hat{\alpha}_n}$ satisfies a weak BvM in $H(\delta)$ with centering $\hat{f}_n = \hat{f}_{n,\hat{\alpha}_n}$ instead of $\mathbb{Y}$. By Theorem 1 of \cite{CaNi}, $R_n \rightarrow^{\P_0} \tilde{\Phi}^{-1}(1-\gamma)>0$, where $\tilde{\Phi}$ is defined via $\tilde{\Phi}(t) = \mathcal{N}(\|\mathbb{Z}\|_{H(\delta)} \leq t)$ and we recall $\mathcal{N}$ is the law of the white noise $\mathbb{Z}$ as an element of $H(\delta)$. Note that $\Pi_{\hat{\alpha}_n,n}^{(1)}$ equals
\begin{align*}
\mathcal{N}(\|\mathbb{Z}\|_{H(\delta)} \leq R_n + o(1)) + O \left( \sup_{t\geq 0} | \Pi_{\hat{\alpha}_n} (\sqrt{n}\snorm{f-\hat{f}_{n,\hat{\alpha}_n}}_{H(\delta)} \leq t|Y) - \mathcal{N}(\|\mathbb{Z}\|_{H(\delta)} \leq t) | \right).
\end{align*}
Since $\tilde{\Phi}$ is strictly monotone and continuous, the first term equals $\tilde{\Phi} (R_n + o(1)) = (1-\gamma) + o_{\P_0}(1)$. Using that $H(\delta)$-norm balls form a uniformity class for $\mathcal{N}$ (see the proof of Theorem 1 of \cite{CaNi}), the second term is $o_{\P_0}(1)$ as required.

We now turn our attention to $\Pi_{\hat{\alpha}_n,n}^{(2)}$. By Theorem 1 of \cite{KnSzVVVZ}, $\sup_{f_0 \in \mathcal{Q}(\beta,R)} \P_0 (\hat{\alpha}_n \in [\underline{\alpha}_n,\overline{\alpha}_n]) \rightarrow 1$, where $\underline{\alpha}_n,\overline{\alpha}_n$ are defined in Section \ref{other paper section}. By Lemma 1(i) of \cite{KnSzVVVZ}, $\underline{\alpha}_n \geq \beta - C/\log n \geq \beta/2$ for $n$ large enough. Since also $\hat{\alpha}_n \leq a_n$ by the choice \eqref{likelihood}, $\sup_{f_0 \in \mathcal{Q}(\beta,R)} \P_0 (\hat{\alpha}_n \in [\beta/2,a_n]) \rightarrow 1$ and we may therefore restrict $\hat{\alpha}_n$ to this interval.

By Theorem 1 of Freedman \cite{Fr}, we have that for $f\sim \Pi_\alpha (\cdot | Y)$,
\begin{equation*}
\|f-\hat{f}_{n,\alpha}\|_2^2 = C_{n,\alpha} + \sqrt{D_{n,\alpha}} Z_{n,\alpha},
\end{equation*}
where $C_{n,\alpha}n^{\frac{2\alpha}{2\alpha+1}} \rightarrow C_\alpha = \int_0^\infty (1+u^{2\alpha+1})^{-1} du$, $D_{n,\alpha}n^{\frac{4\alpha+1}{2\alpha+1}} \rightarrow D_\alpha = 2\int_0^\infty (1+u^{2\alpha+1})^{-2} du$, $Z_{n,\alpha}$ has mean 0, variance 1 and $Z_{n,\alpha} \rightarrow^d N(0,1)$ as $n \rightarrow \infty$. We wish to make these three convergence statements uniform over $\alpha \in [\beta/2,a_n]$. The first two statements essentially follow from \cite{Fr} upon keeping careful track of the remainder terms. Let $g_\alpha(u) = 1/(1+u^{2\alpha+1})$. First note that $C_{n,\alpha} = \sum_{k=1}^\infty 1/(k^{2\alpha+1}+n)$ and $n^{\frac{2\alpha}{2\alpha+1}}/(k^{2\alpha+1}+n) = h_n g_\alpha(kh_n)$, where $h_n = n^{-1/(2\alpha+1)}$. Using these yields that for any $L>0$, $|C_{n,\alpha}n^{\frac{2\alpha}{2\alpha+1}}-C_\alpha|$ is bounded above by
\begin{align*}
\left| \sum_{k=1}^{L/h_n} h_n g_\alpha(kh_n) - \int_0^L g_\alpha(u) du \right| + \sum_{k=L/h+1}^\infty h_n g_\alpha(kh_n) +  \int_L^\infty g_\alpha(u) du .
\end{align*}
The second and third terms are easily bounded by $L^{-2\alpha} /(2\alpha)$. The first term is just the error when approximating $\int_0^L g_\alpha$ with its (right) Riemann sum with $L/h$ points, which is bounded by $\|g_\alpha'\|_{L^\infty[0,L]}L^2/(2L/h_n)$. One can show that $\|g_\alpha'\|_{L^\infty[0,\infty)} = \alpha^{1-\frac{1}{2\alpha+1}} (\alpha+1)^{1+\frac{1}{2\alpha+1}} /(2\alpha+1)$. Substituting this into the bound for the Riemann sum and optimizing the three terms over $L$ gives that the previous display is bounded by a multiple of
\begin{align*}
\frac{1}{\alpha} \left( \frac{(2\alpha)^{2-\frac{1}{2\alpha+1}}(2\alpha+2)^{1+\frac{1}{2\alpha+1}}}{2\alpha+1} \right)^\frac{2\alpha}{2\alpha+1} h_n^\frac{2\alpha}{2\alpha+1}.
\end{align*}
It can be checked that over the range  $0 < \beta/2 \leq \alpha \leq a_n \rightarrow \infty$, the above display without the $h_n^\frac{2\alpha}{2\alpha+1}$ term is maximized at $a_n$ for $n$ large enough, whereupon it can be bounded by a constant multiple of $a_n$. The continuous function $\alpha \mapsto h_n^\frac{2\alpha}{2\alpha+1} = n^{-\frac{2\alpha}{(2\alpha+1)^2}}$ has a single minimum on $[0,\infty)$ occurring at $\alpha = 1/2$, is strictly decreasing on $[0,1/2]$, strictly increasing on $[1/2,\infty)$ and attains its maximal value of 1 at $\alpha = 0,\infty$. Since we consider only the region $\alpha \geq \beta/2$, it follows that the maximum will occur at $a_n$ for $n$ large enough, depending only on $\beta$. We have therefore shown that the previous display is bounded above by
\begin{align*}
C(\beta)a_n n^{-\frac{2a_n}{(2a_n+1)^2}} \leq C(\beta) \exp \left( \log a_n - \frac{\log n}{4a_n}\right) \rightarrow 0,
\end{align*}
where the inequality holds for $n$ large enough (depending on $a_n$) and the convergence to zero follows since $a_n \leq \log n/(6\log n \log n)$ by assumption. In conclusion, $\sup_{\alpha\in[\beta/2,a_n]} |C_{n,\alpha}n^{\frac{2\alpha}{2\alpha+1}}-C_\alpha | \rightarrow 0$. Identical computations yield that $\sup_{\alpha\in[\beta/2,a_n]} |D_{n,\alpha}n^{\frac{4\alpha+1}{2\alpha+1}} -D_\alpha | \rightarrow 0$.

It remains only to show the uniformity of the convergence in distribution, which is based on the central limit theorem (Theorem 1 of \cite{Fr}). Under the posterior,
\begin{align*}
Z_{n,\alpha} & = D_{n,\alpha}^{-1/2}\sum_{k=1}^\infty \frac{1}{k^{2\alpha+1}+n} (\zeta_{n,k}^2-1) =: \sum_{k=1}^\infty X_{n,k},
\end{align*}
where $\zeta_{n,k} \sim N(0,1)$ are independent and$$X_{n,k} = D_{n,\alpha}^{-1/2}(\zeta_{n,k}^2-1)/(k^{2\alpha+1}+n)$$are also independent. Recalling the exact definition of $D_{n,\alpha}$ from Theorem 1 of \cite{Fr}, $D_{n,\alpha} = 2\sum_{k=1}^\infty (k^{2\alpha+1}+n)^{-2} \geq \tfrac{1}{2}\sum_{k=1}^{n^{1/(2\alpha+1)}} n^{-2} =\tfrac{1}{2} n^{-2+\frac{1}{2\alpha+1}}$. Letting $\lambda_r = \E|\zeta_{n,k}^2-1|^r$,
\begin{align*}
\sum_{k=1}^\infty \E |X_{n,k}|^3 & \leq D_{n,\alpha}^{-3/2} \lambda_3 \left( \sum_{k=1}^{n^{1/(2\alpha+1)}} \frac{1}{n^3} + \sum_{k>n^{1/(2\alpha+1)}} k^{-6\alpha-3} \right) \\
& \leq  2(2n^{2-\frac{1}{2\alpha+1}})^{3/2} \lambda_3 n^{-3+\frac{1}{2\alpha+1}} = 2^{5/2} \lambda_3 n^{-\frac{1}{4\alpha+2}} .
\end{align*}
Let $F_{n,\alpha}$ and $\Phi$ denote the cdfs of $Z_{n,\alpha}$ and the standard normal distribution respectively. By the Berry-Esseen theorem for infinite arrays (Theorem 3.2 of \cite{Fridy2004} with summability matrix $p_{n,k}\equiv1$), there exists a universal constant $C_0$ such that
\begin{align*}
\|F_{n,\alpha} - \Phi\|_\infty \leq C_0 \sum_{k=1}^\infty \E |X_{n,k}|^3 \leq 2^{5/2} C_0 \lambda_3 n^{-\frac{1}{4\alpha+2}},
\end{align*}
which implies $\sup_{\alpha \in[\beta/2,a_n]} \|F_{n,\alpha} - \Phi\|_\infty \rightarrow 0$ since $a_n = o(\log n)$.

We have shown that
\begin{equation*}
\|f-\hat{f}_{n,\alpha}\|_2^2 = (C_\alpha+o(1))n^{-\frac{2\alpha}{2\alpha +1}} + (\sqrt{D_\alpha}+o(1))n^{-\frac{2\alpha+1/2}{2\alpha+1}} Z_{n,\alpha},
\end{equation*}
where all $o(1)$ terms and the convergence in distribution are uniform over $\alpha\in [\beta/2,a_n]$. We consequently see that the remainder term in the posterior probability $\Pi_{\alpha,n}^{(2)}$ only changes this probability by $o(1)$ for any $\alpha \in [\beta/2,a_n]$ and so $\Pi_{\hat{\alpha}_n,n}^{(2)} = (1-\gamma) + o_{\P_0}(1)$.
\end{proof}

\begin{proof}[Proof of Theorem \ref{L_infty asymptotic independence}]
Throughout we will write $T_n$ instead of $T_n^{(2)}$ for convenience, where
\begin{eqnarray*}
T_{n,jk}^{(2)} = \begin{cases} 
       \hat{f}_{n,jk} & \textrm{ if $j \leq j_0(n)$,} \\
      Y_{jk} 1_{ \{ \tilde{f}_{jk} \neq 0 \} } & \textrm{ if $ j_0(n) < j \leq \lfloor \log n/ \log 2 \rfloor $} .\\
      0 & \textrm{ if $\lfloor \log n/ \log 2 \rfloor < j$},\\
   \end{cases} 
\end{eqnarray*}
and $\hat{f}_n$ denotes the posterior mean of the slab and spike procedure. Since the second constraint in \eqref{eq14} is satisfied with posterior probability $1 - o_{\P_0} (1)$ uniformly over $f_0 \in \mathcal{H}(\beta,R)$ by the proof of Proposition \ref{credible band}, it suffices to prove the theorem with $D_n$ in \eqref{M credible set} instead of $\overline{D}_n$. Let $f_0 \in \mathcal{H} (\beta,R)$ and $j_n \rightarrow \infty$ satisfy $j_n \leq j_0(n)$, $w_{j_n} \leq \sqrt{j_0(n)}$ and $w_{j_n}^2 2^{j_n} = o(2^{j_0(n)})$. Similarly to Theorem \ref{l2 asymptotic independence}, we shall decompose both credible balls according to the threshold $j_n$. For this we must understand the typical sizes of the projections of $f-T_n$ in both norms under the posterior.

Consider firstly $\| \cdot \|_{\mathcal{M}}$. For the frequencies $l > j_0(n)$,
\begin{equation}
\snorm{\pi_{>j_0(n)} (f - T_n)}_\mathcal{M} \leq \snorm{\pi_{>j_0(n)} (f - f_0)}_\mathcal{M} + \snorm{\pi_{>j_0(n)} (f_0 - \mathbb{Y})}_\mathcal{M} + \snorm{\pi_{>j_0(n)} (\mathbb{Y} - T_n)}_\mathcal{M}  .
\label{eq30}
\end{equation}
The third term is $O_{\P_0}(w_{j_0(n)}^{-1} \sqrt{(\log n)/n} )$ by the proof of Lemma \ref{efficient estimator lemma}. For the first term, on the event $A_n$ defined in \eqref{eq9},
\begin{equation*}
\begin{split}
\max_{l > j_0(n)} w_l^{-1} \max_k |f_{lk}-f_{0,lk} |  & \leq \max_{j_0(n)< l \leq J_n(\beta)} w_l^{-1} \max_k |f_{lk}-f_{0,lk} | + \max_{l > J_n(\beta)} w_l^{-1} \max_k |f_{0,lk} | \\
& \leq w_{j_0(n)}^{-1} \sqrt{(\log n)/n} + w_{J_n(\beta)}^{-1} \sqrt{(\log n)/n}\\
&  = O(w_{j_0(n)}^{-1} \sqrt{(\log n)/n} ).
\end{split}
\end{equation*}
For the second term in \eqref{eq30},
\begin{equation*}
\E_0 \snorm{\pi_{>j_0(n)} (f_0 - \mathbb{Y})}_{\mathcal{M}(w)} \leq \frac{\sqrt{j_0(n)}}{w_{j_0(n)}\sqrt{n}}  \E_0 \snorm{\pi_{>j_0(n)} (\mathbb{Z}) }_{\mathcal{M}(\sqrt{l})} = O \left( \frac{\sqrt{j_0(n)}}{w_{j_0(n)}\sqrt{n}} \right)
\end{equation*}
using that $\E_0 \snorm{\mathbb{Z}}_{\mathcal{M}(\sqrt{l})}$ is finite by Proposition 2 of \cite{CaNi2}. Combining these yields
\begin{equation*}
\Pi (  f : \snorm{\pi_{>j_0(n)} (f - T_n)}_\mathcal{M} = O(w_{j_0(n)}^{-1} \sqrt{(\log n)/n}) \mid Y) = 1 - o_{\P_0} (1) 
\end{equation*}
since $j_0(n) \lesssim \log n$.

Recall that by definition $T_{n,lk}^{(2)}$ equals the posterior mean for $l \leq j_0(n)$ (see \eqref{efficient estimator2}), so that
\begin{equation*}
f_{lk} - T_{n,lk}^{(2)} |Y_{lk} \sim N\left( 0, \frac{\tau^2}{1+n\tau^2} \right), \quad \quad 0 \leq l \leq j_0(n),
\end{equation*}
under the posterior. For $(\zeta_{lk})$ i.i.d. standard Gaussians, we have the well-known bound $\E \max_{0\leq k <2^l} |\zeta_{lk}| \leq C \sqrt{l}$ for some universal constant $C$. Applying the Borell-Sudakov-Tsirelson inequality \cite{Led} to the maximum at level $l$ yields that for $M>0$ large enough,
\begin{align*}
\Pi (f : & \| (\pi_{>j_n} - \pi_{>j_0(n)})(f-T_n) \|_\mathcal{M} \geq M \sqrt{j_n} w_{j_n}^{-1} n^{-1/2} | Y) \\
& = \P \left( \max_{j_n < l \leq j_0(n)} \frac{1}{w_l} \max_{0\leq k < 2^l} \left| \frac{\tau}{\sqrt{1+n\tau^2}} \zeta_{lk} \right| \geq \frac{M\sqrt{j_n}}{w_{j_n}\sqrt{n}} \right) \\
& \leq \sum_{l=j_n+1}^{j_0(n)} \P \left( \max_{0\leq k < 2^l} |\zeta_{lk}| - \E  \max_{0\leq k < 2^l} |\zeta_{lk}| >  \frac{M\sqrt{j_n} w_l \sqrt{1+n\tau^2}}{w_{j_n} \sqrt{n}\tau} - \E  \max_{0\leq k < 2^l} |\zeta_{lk}| \right) \\
& \leq 2 \sum_{l=j_n+1}^{j_0(n)} \exp \left( -c \left (M \frac{w_l \sqrt{j_n}}{w_{j_n} \sqrt{l}} - C\right)^2 l  \right) \leq C e^{-c'j_n} \rightarrow 0.
\end{align*}
Combining this with the above inequalities gives
\begin{equation}
\Pi ( f : \snorm{\pi_{>j_n} (f - T_n)}_\mathcal{M} = O(w_{j_0(n)}^{-1} \sqrt{(\log n)/n}) \mid Y) = 1 - o_{\P_0} (1) .
 \label{eq31}
\end{equation}
We now show that the $\|\cdot \|_\mathcal{M}$-norm of the remaining frequencies $j\leq j_n$ is of strictly larger size with high probability. For any $u_n \rightarrow 0$,
\begin{equation}
\Pi ( f : \snorm{\pi_{j_n} (f - T_n)}_\mathcal{M} \leq w_0 u_n n^{-1/2} \mid Y) \leq \Pi ( f : |\zeta_{00}| \leq c u_n \mid Y) \leq 2ceu_n/\sqrt{2\pi} = o(1).
\end{equation}
In particular, taking $u_n \gg\sqrt{\log n}/w_{j_0(n)}$ gives the result.

Turn now to $\|\cdot \|_{B_{\infty1}^0}$. For $f \in D_n$,
\begin{equation}
\begin{split}
\snorm{\pi_{j_n}(f - T_n ) }_{B_{\infty1}^0} & = \sum_{l=0}^{j_n} 2^{l/2} \max_k |f_{lk} - T_{n,lk} | \leq \sum_{l=0}^{j_n} 2^{l/2} w_l \frac{R_n}{\sqrt{n}}  =  O_{\P_0} \left(  \frac{w_{j_n} 2^{j_n/2}}{\sqrt{n}} \right) .
\label{eq32}
\end{split}
\end{equation}
Note that 
\begin{equation}
\Pi ( f: \|\pi_{>j_n}(f-T_n)\|_{B_{\infty1}^0} \leq c 2^{j_0(n)/2} \sqrt{j_0(n)/n} \mid Y) \leq \P \left( \max_{0\leq k < 2^{j_0(n)}} \left| \zeta_{j_0(n) k} \right| \geq c' \sqrt{j_0(n)} \right) .
\label{eq35}
\end{equation}
Using again the Borell-Sudakov-Tsirelson inequality as above gives that the right-hand side is $O(2^{-c''j_0(n)})$ for $c>0$ small enough. Combining \eqref{eq31}-\eqref{eq35}, we have shown that for some $C_1,...,C_4>0$ and $u_n \rightarrow 0$ such that $u_n \gg\sqrt{\log n}/w_{j_0(n)}$,
\begin{equation*}
\begin{split}
\bar{A}_n := \Big\{ & \snorm{\pi_{>j_n} (f - T_n)}_\mathcal{M} \leq C_1 \frac{\sqrt{\log n}}{w_{j_0(n)} \sqrt{n}}, \quad \|\pi_{j_n}(f-T_n)\|_\mathcal{M} \geq C_2 \frac{u_n}{\sqrt{n}},   \\
& \quad \snorm{\pi_{j_n}(f - T_n ) }_{B_{\infty1}^0}  \leq  C_3 \frac{w_{j_n} 2^{j_n/2}}{\sqrt{n}}, \quad  \|\pi_{>j_n}(f-T_n)\|_{B_{\infty1}^0} \geq C_4 \frac{\sqrt{j_0(n)} 2^{j_0(n)/2} }{\sqrt{n}}  \Big\}
\end{split}
\end{equation*}
satisfies $\Pi (\bar{A}_n|Y) = 1-o_{\P_0}(1)$.

Write $\delta_n := w_{j_n}2^{j_n/2}/\sqrt{n}$. Note that by the previous display and the assumptions on the growth of $j_n\rightarrow \infty$, the radius $\bar{Q}_n(\gamma)$ of the credible set $D_n^{L^\infty}$ satisfies $\bar{Q}_n \gtrsim \sqrt{j_0(n)}2^{j_0(n)}/\sqrt{n} \gg \delta_n$ with $\P_0$-probability tending to one. Using the independence of the different coordinates under the posterior, the probability $\Pi (D_n \cap D_n^{L^\infty} \mid Y)$ equals
\begin{equation*}
\begin{split}
& \Pi \left(  \left\{ \snorm{f - T_n}_\mathcal{M} \leq R_n / \sqrt{n} , \quad  \snorm{f - T_n}_\infty \leq  \bar{Q}_n \right\}  \cap \bar{A}_n \mid  Y \right)  + o_{\P_0} (1)  \\
& = \Pi \left(  \left\{ \snorm{ \pi_{j_n} (f - T_n) }_\mathcal{M}   \leq (R_n + O(w_{j_0(n)}^{-1} \sqrt{\log n}))/ \sqrt{n}  \right. \right. ,  \\
& \quad \quad \quad \quad \left. \left. \snorm{\pi_{>j_n} (f - T_n) }_{B_{\infty1}^0}    \leq  \bar{Q}_n   + O(\delta_n)  \right\}  \cap \bar{A}_n \mid  Y \right)  + o_{\P_0} (1)  \\
& = \Pi \left( f:  \snorm{\pi_{j_n} (f - T_n) }_\mathcal{M} \leq (R_n + o(1)) / \sqrt{n} \mid Y \right) \\
& \quad \times \Pi \left( f: \snorm{ \pi_{>j_n} (f - T_n) }_{B_{\infty1}^0}  \leq  \bar{Q}_n + O(\delta_n)  \mid  Y \right)  + o_{\P_0} (1)  .
\end{split}
\end{equation*}
Again using that $\Pi (\bar{A}_n|Y) = 1-o_{\P_0}(1)$, the final line equals
\begin{equation}
\begin{split}
& \Pi \left( f :  \snorm{f - T_n}_\mathcal{M} \leq (R_n+o(1)) / \sqrt{n} \right) \\
& \quad \quad \times  \Pi \left( f:  \snorm{f - T_n}_{B_{\infty1}^0}  \leq  \bar{Q}_n + O(\delta_n) \mid  Y \right)  + o_{\P_0} (1).
\label{eq34}
\end{split}
\end{equation}

We now check that the above product has asymptotically the correct posterior probability. By the same argument as in Theorem \ref{l2 asymptotic independence} (replacing $H(\delta)$ by $\mathcal{M}$), the first probability equals $(1-\gamma)+o_{\P_0}(1)$. Setting $M_l = 2^{l/2} \max_{0\leq k < 2^l} |f_{lk} - T_{n,lk}|$, we have $\|f-T_n\|_{B_{\infty1}^0}  = \sum_l M_l$, where the $(M_l)$ are independent due to the product structure of the posterior. We can therefore write the posterior density of $\|f-T_n\|_{B_{\infty1}^0} $ as $h_n*G_n$, where $M_{j_0(n)}$ has density $h_n$ and $\sum_{l\neq j_0(n)} M_l$ has probability distribution $G_n$, both supported on $[0,\infty)$.

Using standard extreme value theory, we can establish the limiting distribution of $M_{j_0(n)}$. For $Z_i \sim N(0,1)$ independent, we have $a_m^{-1} (\max_{1\leq i \leq m} |Z_i|-b_m)$ converges in distribution to the standard Gumbel distribution (i.e. distribution function $F(x) = \exp(-e^{-x})$), where
\begin{equation*}
a_m = \frac{1}{\sqrt{2\log 2m}}, \quad \quad b_m = \sqrt{2\log 2m} - \frac{\log (4\pi \log 2m)}{2\sqrt{2\log 2m}} + O\left( \frac{1}{\log m} \right)
\end{equation*}
($b_m$ is the solution to $2m^2 = \pi b_m^2 e^{b_m^2}$). This follows from Theorem 10.5.2(c) and Example 10.5.3 of \cite{DaNa} with only minor modifications due to the absolute values within the maximum (intuitively it is the same as the maximum of $2m$ standard Gaussians). Moreover, by P\'olya's Theorem (p. 265 of \cite{chow1988}) the convergence of the distribution functions is uniform:
\begin{equation}
\sup_{x\in\R} | (\Phi (a_mx + b_m) - \Phi(-a_mx-b_m))^m - \exp(-e^{-x})| \rightarrow 0,
\label{eq33}
\end{equation}
where we recall $|Z_i|$ has distribution function $\Phi(x)-\Phi(-x)$.

Recall that $M_{j_0(n)}$ is the sum of i.i.d. (rescaled) Gaussians under the posterior. Using that $\bar{Q}_n$ is the $(1-\gamma)$-posterior quantile for $\|f-T_n\|_{B_{\infty1}^0}$,
\begin{align*}
|\Pi  ( f: & \snorm{f - T_n}_{B_{\infty1}^0} \leq  \bar{Q}_n + O(\delta_n) \mid  Y ) - (1-\gamma)|  \\
& \leq \Pi \Big( \bar{Q}_n - c\delta_n \leq \sum_l M_l \leq \bar{Q}_n + c\delta_n \Big| Y \Big) \\
& = \int_{\bar{Q}_n-c\delta_n}^{\bar{Q}_n+c\delta_n} \int_0^\infty h_n(x-y) dG_n(y) dx \\
& = \int_0^\infty \int_{\bar{Q}_n-y-c\delta_n}^{\bar{Q}_n-y+c\delta_n} h_n(z) dG_n(y) \\
& \leq \sup_{t\geq 0} \P ( M_{j_0(n)} \in [t-c\delta_n ,t + c\delta_n] ).
\end{align*}
Using the limiting distribution of $M_{j_0(n)}$, \eqref{eq33} and that the maximum of the standard Gumbel density function is $e^{-1}$, the last probability equals
\begin{align*}
\P & \left(\frac{2^{j_0(n)/2}\tau}{\sqrt{1+n\tau^2}} \max_{0\leq k <2^{j_0(n)}} |\zeta_{j_0(n) k}| \in [t-c\delta_n ,t+c\delta_n ] \right) \\
& \quad = \P \left( a_{2^{j_0(n)}}^{-1} \left( \max_{0\leq k <2^{j_0(n)}} |\zeta_{j_0(n) k}| - b_{2^{j_0(n)}} \right) \in a_{2^{j_0(n)}}^{-1} \frac{\sqrt{1+n\tau^2}}{2^{j_0(n)/2}\tau} [t-b_{2^{j_0(n)}}-c\delta_n ,t-b_{2^{j_0(n)}}+c\delta_n ]\right) \\
& \quad \leq \P \left( \text{Gumbel}(0,1) \in a_{2^{j_0(n)}}^{-1} \frac{\sqrt{1+n\tau^2}}{2^{j_0(n)/2}\tau} [t-b_{2^{j_0(n)}}-c\delta_n ,t-b_{2^{j_0(n)}}+c\delta_n ]\right) + o(1) \\
& \quad \leq c' e^{-1} \sqrt{\frac{nj_0(n)}{2^{j_0(n)}} }  \delta_n + o(1) = c' e^{-1} w_{j_n} \sqrt{j_0(n)} 2^{\frac{j_n-j_0(n)}{2}} + o(1) \rightarrow 0,
\end{align*}
by the choice of $j_n$. In conclusion, we have shown that the second probability in \eqref{eq34} equals $(1-\gamma) + o_{\P_0} (1)$. This completes the proof.
\end{proof}

\subsection{Remaining proofs}

\begin{proof}[Proof of Proposition \ref{Doob result}]
Fix $\rho > 1$, let $\varepsilon = \varepsilon (\alpha,\rho,R) < (1 - \rho^{-2\alpha}) / (2\alpha R)$ be sufficiently small so that $\varepsilon \in (0,1)$ and consider the events $A_{\alpha,N} = \{ \sum_{k=N}^{ \lceil \rho N \rceil } f_k^2 < \varepsilon R N^{-2\alpha} \} $. By a simple integral comparison we have that $\sum_{k=N}^{ \lceil \rho N \rceil } k^{-2\alpha-1} \geq (2\alpha)^{-1} N^{-2\alpha} (1 - \rho^{-2\alpha} )$, so that under the conditional prior,
\begin{equation*}
\begin{split}
\Pi_\alpha (A_{\alpha,N} ) & = \P \left( \sum_{k=N}^{ \lceil \rho N \rceil  }  k^{-2\alpha-1} g_k^2 < \varepsilon R N^{-2\alpha} \right) \\
& \leq \P \left( \sum_{k=N}^{ \lceil \rho N \rceil } k^{-2\alpha-1} (g_k^2 - 1) < \varepsilon R N^{-2\alpha} - \frac{1}{2\alpha} N^{-2\alpha}  ( 1 - \rho^{-2\alpha} )  \right)  \\
& \leq \P \left( \sum_{k=N}^{ \lceil \rho N \rceil } k^{-2\alpha-1} (g_k^2 - 1) <  - \varepsilon' N^{-2\alpha}   \right)  ,
\end{split}
\end{equation*}
where the $g_k$'s are i.i.d. standard normal random variables and $\varepsilon ' > 0$ (by the choice of $\varepsilon$). By (4.2) of Lemma 1 of \cite{LaMa} we have the exponential inequality
\begin{equation*}
\P \left( \sum_{k=N}^{\lceil \rho N \rceil} k^{-2\alpha-1} (g_k^2 -1) \leq -2 \left( \sum_{k=N}^{ \lceil \rho N \rceil } k^{-4\alpha -2} \right)^{1/2} \sqrt{x}  \right) \leq e^{-x}  .
\end{equation*}
For $N \geq 2$, again by an integral comparison we have that $\sum_{k=N}^{ \lceil \rho N \rceil } k^{-4\alpha -2} \leq C(\alpha) N^{-4\alpha -1} $. Using this and letting $x = MN$, the exponential inequality becomes
\begin{equation*}
\P \left(   \sum_{k=N}^{\lceil \rho N \rceil} k^{-2\alpha-1} (g_k^2 -1) \leq - C'(\alpha) \sqrt{M} N^{-2\alpha}  \right) \leq e^{-MN}  .
\end{equation*}
Taking $M$ sufficiently small so that $C' (\alpha) \sqrt{M} < \varepsilon '$, we obtain that $\Pi_\alpha ( A_{\alpha , N} ) \leq e^{-MN}$. Since this sequence is summable in $N$, the result follows from the first Borel-Cantelli Lemma.
\end{proof}

\begin{proof}[Proof of Proposition \ref{negative BvM in slab and spike}]
Under the law $\P_0$, $\sqrt{n} \E_0 \snorm{\mathbb{Y} - f_0}_{\mathcal{M}(w)} = \E_0 \snorm{\mathbb{Z}}_{\mathcal{M}(w)} < \infty$ by Proposition 2 of \cite{CaNi2}. By the triangle inequality it therefore suffices to show the conclusion of Proposition \ref{negative BvM in slab and spike} with $\mathbb{Y}$ replaced by $f_0$. Rewrite the multiscale indices $\Lambda = \{ (l,k) : l \geq 0 , k = 0,...,2^l - 1 \}$ in increasing lexicographic order, so that $\Lambda = \{ (l_m , k_m) : m \in \mathbb{N} \}$, where
\begin{equation*}
\begin{split}
& l_m = i , \quad \quad \quad \quad \quad if \quad 2^i \leq m < 2^{i+1} , \quad  i = 0,1,2,... , \\
& k_m = m - 2^i , \quad \quad \;  if \quad  2^i \leq m < 2^{i+1} , \quad i=0,1,2,...
\end{split}
\end{equation*}
Consider a strictly increasing subsequence $(n_m)_{m \geq 1}$ of $\mathbb{N}$ such that $(\log n_m) / w_{l_m}^2 \rightarrow \infty$ as $m \rightarrow \infty$ (such a subsequence can be constructed for any admissible $(w_l)$ since $w_l \nearrow \infty$). Define a function $f_0 \in \ell_2$ via its wavelet coefficients
\begin{equation*}
\langle f_0 , \psi_{l_m k_m} \rangle = r \sqrt{\log n_m / n_m}  ,
\end{equation*}
where $r \leq \underline{\gamma}$ for $\underline{\gamma}$ the value given in the proof of Theorem \ref{slab and spike contraction}. Since
\begin{equation*}
2^{l_m (\beta + 1/2)} | \langle f_0 , \psi_{l_m k_m} \rangle | \leq  r m^{\beta + 1/2} \sqrt{\frac{\log n_m}{n_m}} ,
\end{equation*}
we can ensure $f_0$ is in any given H\"older ball $\mathcal{H}(\beta,R)$ by letting $r$ be sufficiently small and taking the subsequence $n_m$ to grow fast enough. Consider now a further subsequence, removing terms corresponding to one index per resolution level, say $(l,k_l)$ (i.e. removing terms with indices $m = 2^l + k_l$, $l=0,1,2,...$, from the above subsequence), and set $| \langle f_0, \psi_{lk_l} \rangle | = R 2^{-l(\beta+1/2)}$. Using the Besov space embedding $L^\infty \subset B_{\infty \infty}^0$,
\begin{equation*}
\begin{split}
\norm{K_j(f) - f}_\infty & \geq  C(\psi) \max_{l > j} 2^{l/2} \max_k |\langle f_0, \psi_{lk} \rangle |  \\
&  \geq C(\psi) 2^{(j+1)/2} | \langle f_0 ,  \psi_{(j+1) k_{j+1} } \rangle |  = C(\psi) R 2^{-\beta} 2^{-j\beta} = \varepsilon (\beta,R,\psi) 2^{-j\beta} ,
\end{split}
\end{equation*}
thereby establishing that $f_0 \in \mathcal{H}_{SS}(\beta,R,\varepsilon)$.

Let $A_n$ denote the event defined in \eqref{eq9}. We have that on $A_{n_m}$, the posterior distribution $\Pi ' (\cdot \mid Y^{(n_m)})$ assigns the $(l_m,k_m)$ coordinate to the Dirac mass component of the distribution. Consequently, by the choice of $(n_m)$,
\begin{equation*}
\begin{split}
\E_0 \Pi ' ( & \norm{f - f_0}_\mathcal{M} \leq M_{n_m} {n_m}^{-1/2}  \mid Y^{(n_m)} )  \\
&  = \E_0 \Pi ' ( \{ \norm{f - f_0}_\mathcal{M} \leq M_{n_m} {n_m}^{-1/2} \} \cap A_{n_m} \mid Y^{(n_m)} ) + o(1)  \\
& \leq \E_0 \Pi ' ( \{ |f_{l_m k_m} - f_{0,l_m k_m}| \leq M_{n_m} w_{l_m} {n_m}^{-1/2}  \} \cap A_{n_m} \mid Y^{(n_m)}   )  + o(1)  \\
& = \E_0 \Pi ' ( \{ r \sqrt{\log n_m / n_m} \leq M_{n_m} w_{l_m}  n_m^{-1/2} \} \cap A_{n_m} \mid Y^{(n_m)} )  + o(1) \\
& \leq \E_0 \Pi ' (  r \sqrt{ \log n_m} /w_{l_m} \leq M_{n_m}  \mid Y^{(n_m)} ) + o(1) = o(1)
\end{split}
\end{equation*}
for any sequence $M_n$ such that $M_{n_m} = o (w_{l_m}^{-1} \sqrt{\log n_m})$ as $m \rightarrow \infty$.
\end{proof}

\section{Technical facts and results}

\subsection{Results for $\ell_2$-setting}

The following two lemmas describe the behaviour of the posterior mean of the empirical Bayes procedure. The first says that the posterior mean is a consistent estimator of $f_0$ in a sequence of Sobolev norms with data driven exponent. In particular, we are interested in the case $\epsilon_n \rightarrow 0$ when the Sobolev exponent tends to the true smoothness $\beta$. Note that we require $\epsilon_n$ strictly positive since the posterior mean is itself not an element of $H^{\hat{\alpha}_n}$. The second says that $\hat{f}_n$ is an efficient estimator of $f_0$ in $H_2^{-1/2,\delta}$. Both proofs are similar to that of Theorem 2 of \citep{KnSzVVVZ} and are thus omitted.

\begin{lemma}\label{coverage lemma}
Let $\hat{f}_n$ denote the posterior mean of the empirical Bayes procedure and let $\epsilon_n > 0$. Then for every $\beta,R>0$ and $M_n \rightarrow \infty$, we have
\begin{equation*}
\sup_{f_0 \in \mathcal{Q}_{SS} (\beta,R,\varepsilon)} \P_0 \left( \snorm{\hat{f}_n - f_0}_{H^{\hat{\alpha}_n - \epsilon_n}} \geq M_n \right) \rightarrow 0 
\end{equation*}
as $n \rightarrow \infty$.
\end{lemma}

\begin{lemma}\label{l2 efficient estimator}
Let $\hat{f}_n$ denote the posterior mean of the empirical Bayes procedure. Then for $f_0 \in \mathcal{Q}_{SS} (\beta,R,\varepsilon)$, $\delta > 1$ and as $n \rightarrow \infty$,
\begin{equation*}
\snorm{\hat{f}_n - \mathbb{Y}}_{H(\delta)} = o_{\P_0} (1/\sqrt{n}) .
\end{equation*}
\end{lemma}

We have an exponential inequality which measures posterior spread in a variety of Sobolev norms. Since for fixed $\alpha$, the posterior only depends on the data through the posterior mean $\hat{f}_{n,\alpha}$, the following probabilities are independent of the observed data $Y$.

\begin{lemma}\label{exponential lemma}
Let $\hat{f}_{n,\alpha}$ denote the posterior mean of $\Pi_\alpha ( \cdot \mid Y)$. Then for any $0 \leq s < \alpha$ and any $\eta >0$,
\begin{equation*}
\begin{split}
\Pi_\alpha &  \left( f : \left. \snorm{f - \hat{f}_{n,\alpha}}_{H^s}^2  \geq (1 + \eta) \left[ 1 + \frac{1}{2(\alpha-s)} \right] n^{-\frac{2(\alpha-s)}{2\alpha+1}}  \right| Y \right) \\
& \quad \quad \quad \quad \leq  e^{1/4} \exp \left( -\frac{\eta}{\sqrt{24}}  \left[ 1 + \frac{1}{2(\alpha-s)} \right] n^{1/(4\alpha+2)}  \right) .
\end{split}
\end{equation*}
\end{lemma}

\begin{proof}
For $f \sim \Pi_\alpha (\cdot \mid Y)$ we can use the explicit form of the posterior mean in \eqref{eq28} to write
\begin{equation*}
\snorm{f - \hat{f}_{n,\alpha}}_{H^s}^2 = \sum_{k=1}^\infty \frac{k^{2s}}{ k^{2\alpha+1} + n } \zeta_k^2,
\end{equation*}
where the $\zeta_k \sim N(0,1)$ are independent. Letting $t_n = n^{1/(2\alpha + 1)}$ and using standard tail bounds,
\begin{equation*}
\begin{split}
\E^\Pi \left[ \snorm{f - \hat{f}_{n,\alpha} }_{H^s}^2 \mid Y \right] = \sum_{k=1}^\infty \frac{k^{2s}}{ k^{2\alpha+1} + n } & \leq \frac{1}{n} \sum_{k \leq t_n} k^{2s} + \sum_{k > t_n} k^{-2(\alpha -s) -1} \\
& \leq \left[ 1 + \frac{1}{2(\alpha-s)} \right] n^{-\frac{2(\alpha-s)}{2\alpha+1}} .
\end{split}
\end{equation*}
The posterior variance of $\snorm{f-\hat{f}_{n,\alpha} }_{H^s}^2$ is given by
\begin{equation*}
\nu^2 = 2 \sum_{k=1}^\infty \frac{k^{4s}}{(k^{2\alpha+1} +n)^2}  \leq \frac{2}{n^2} \sum_{k \leq t_n} k^{4s} + \sum_{k > t_n} k^{-4(\alpha-s) -2} \leq 3 n^{-\frac{4(\alpha-s) + 1}{2\alpha+1}} .
\end{equation*}
Combining the above with the exponential inequality for $\chi^2$-squared random variables found in Proposition 6 of \cite{RoDu}, we have
\begin{equation*}
\begin{split}
e^{1/4} e^{-x/\sqrt{8} } & \geq \P \left( \sum_{k=1}^\infty \frac{k^{2s}}{k^{2\alpha+1} + n} (\zeta_k^2 - 1) \geq \nu x \right)  \\
& \geq \Pi_\alpha \left( \left. f : \snorm{f - \hat{f}_{n,\alpha} }_{H^s}^2  \geq\left[ 1 + \frac{1}{2(\alpha-s)} \right] n^{-\frac{2(\alpha-s)}{2\alpha+1}} + \sqrt{3} n^{-\frac{4(\alpha-s) + 1}{4\alpha+2}} x  \right| Y  \right) .
\end{split}
\end{equation*}
Taking $x = (\eta/\sqrt{3}) [ 1 + 1/(2\alpha-2s)] n^{1/(4\alpha+2)}$ gives the desired result.
\end{proof}

\subsection{Results for $L^\infty$-setting}

To prove Proposition \ref{credible band} we need to understand the behaviour of the posterior median under the law $\P_0$.

\begin{lemma}\label{median lemma}
Let $\tilde{f} = \tilde{f}_n$ denote the posterior median (defined coordinate-wise) of the slab and spike prior. Then the event
\begin{equation}
\begin{split}
B_n  = \{ \tilde{f}_{lk} & = 0 \quad \forall (l,k) \in \mathcal{J}_n^c (\underline{\gamma})  \} \cap \{ \tilde{f}_{lk} \neq 0  \quad \forall  (l,k) \in  \mathcal{J}_n (\overline{\gamma} ' )  \}  \\
& \cap \{ \sqrt{n} | Y_{lk} - f_{0,lk}| \leq  (8 l \log 2  +  a \log n)^{1/2} \quad \forall l \leq J_n , \forall k= 0,...,2^l-1 \}   
\label{eq19}
\end{split}
\end{equation}
satisfies $\inf_{f_0 \in \mathcal{H} (\beta,R) } \P_0 (B_n ) \rightarrow 1$ as $n \rightarrow \infty$, for some constants $0 < \underline{\gamma} < \overline{\gamma}' < \infty$ and $a > 0$.
\end{lemma}

\begin{proof}
We show that the $\P_0$-probability of each of these events individually tends to 1. For the first event
\begin{equation*} 
\begin{split}
\{ \tilde{f}_{lk} = 0 \quad \forall (l,k) \in \mathcal{J}_n^c (\underline{\gamma}) \} &  \supseteq \{ \Pi ( f_{lk} = 0 \mid Y) \geq 1/2 \quad \forall (l,k) \in \mathcal{J}_n^c (\underline{\gamma})  \}  \\
& \supseteq \{ \Pi ( f_{lk} = 0 \quad \forall (l,k) \in \mathcal{J}_n^c ( \underline{\gamma} ) ) \geq 1/2 \} \\
& = \{ \Pi ( S \cap \mathcal{J}_n^c (\underline{\gamma}) = \emptyset ) \geq 1/2 \} .
\end{split}
\end{equation*}
By Lemma 1 of \cite{HoRoSc} the $\P_0$-probability of this last event tends to 1 for some $\underline{\gamma} > 0$ as $n \rightarrow \infty$.

Consider the third event,
\begin{equation*}
\Omega_n = \{ \sqrt{n} | Y_{lk} - f_{0,lk}| \leq  (8 l \log 2  +  a \log n)^{1/2} \quad \forall l \leq J_n , \forall k= 0,...,2^l-1 \}  ,
\end{equation*}
which by (41) of \cite{HoRoSc} (or the Borell-Sudakov-Tsireslon inequality \cite{Led}) satisfies $\P_0 (\Omega_n^c) \rightarrow 0$. We shall lastly show that 
\begin{equation}
\Omega_n \subset  \{ \tilde{f}_{lk} \neq 0  \quad \forall  (l,k) \in  \mathcal{J}_n (\overline{\gamma} ' ) \}  ,
\label{eq20}
\end{equation}
which then completes the proof.

Consider firstly the case $f_{0,lk} \in \mathcal{J}_n (\overline{\gamma}')$ with $f_{0,lk} > 0$. Write
\begin{equation}
\Pi ( f_{lk} \leq 0 \mid Y ) = \Pi ( f_{lk} = 0 \mid Y ) + \Pi ( f_{lk} < 0 \mid Y)  .
\label{eq21}
\end{equation}
By the proof of Lemma 1 of \cite{HoRoSc}, we have that on the event $\Omega_n$ and for sufficiently large $\overline{\gamma}'$, the first posterior probability in \eqref{eq21} is bounded above by a multiple of $n^{K + 1/2 - (\overline{\gamma}')^2 /8}$. Again on the event $\Omega_n$, we use (42) of \cite{HoRoSc} to bound the second term via
\begin{equation}
\begin{split}
\Pi ( f_{lk} < 0 \mid Y ) & = \frac{ w_{jn} \int_{-\infty}^0 e^{ -\frac{n}{2} (x-Y_{lk})^2 } g(x) dx }{ w_{jn} \int_{-\infty}^\infty  e^{ -\frac{n}{2} (x-Y_{lk})^2 } g(x) dx + (1- w_{j,n}) }   \\
& \leq   \frac{  \norm{g}_\infty  \int_{-\infty}^{-\sqrt{n} Y_{lk} } e^{-\frac{1}{2}v^2}  dv }{ a(\pi /n)^{1/2} }  = C \sqrt{n} \bar{\Phi} ( \sqrt{n} Y_{lk} ) ,
\end{split}
\label{eq22}
\end{equation}
where $\bar{\Phi} = 1 - \Phi$ with $\Phi$ the distribution function of a standard normal variable. On $\Omega_n$, we have for $l \leq J_n$,
\begin{equation*}
Y_{lk} = (Y_{lk} - f_{0,lk}) + f_{0,lk} \geq   -  \sqrt{ \frac{ 2 J_n \log 2 + \frac{1}{2} \log n }{n} } + \overline{\gamma}' \sqrt{ \frac{\log n}{n} }  \geq \delta \sqrt{ \frac{\log n}{n} } 
\end{equation*}
for some $\delta = \delta(\overline{\gamma}') >0$ that can be made arbitrarily large by taking $\overline{\gamma}'$ large enough. Thus applying the standard tail bounds for $\bar{\Phi}$ we have that the right-hand side of \eqref{eq22} is bounded above by a multiple of
\begin{equation*}
\sqrt{n} \bar{\Phi} (\delta \sqrt{\log n} ) \leq  \frac{\sqrt{n}}{ \delta \sqrt{  2\pi \log n} }  e^{-\frac{1}{2} \delta^2 \log n} = C(\delta) \frac{ n^{ \frac{1}{2} - \frac{1}{2} \delta^2 } }{ \sqrt{\log n} }  .
\end{equation*}
Combining the above results, we have that for sufficiently large $\overline{\gamma}'$ (and hence $\delta$), \eqref{eq21} is bounded above by a constant times $n^{-B}$ for some $B > 0$, uniformly over the positive coefficients in $\mathcal{J}_n (\overline{\gamma}')$. In particular, the posterior median satisfies $\tilde{f}_{lk} > 0$ for all $(l,k) \in \mathcal{J}_n (\overline{\gamma}')$ with $f_{lk} > 0$ and $n$ large enough. The case $f_{0,lk} < 0$ is dealt with similarly, thereby proving \eqref{eq20}.
\end{proof}

\subsubsection*{A simultaneous estimator in $\mathcal{M}(w)$ and $L^\infty$}

It may be of interest to obtain an efficient estimator of $f_0$ in $\mathcal{M}(w)$ that is also an element of $L^\infty$, unlike $\mathbb{Y}$. Letting $\tilde{f} = \tilde{f}_n$ and $\hat{f}_n$ denote the posterior median and mean of the slab and spike procedure respectively, define the estimators
\begin{eqnarray}\label{efficient estimator1}
T_{n,jk}^{(1)} = \begin{cases} 
      Y_{jk} & \textrm{ if $j \leq j_0(n)$,} \\
      Y_{jk} 1_{ \{ \tilde{f}_{jk} \neq 0 \} } & \textrm{ if $ j_0(n) < j \leq \lfloor \log n/ \log 2 \rfloor $} .\\
      0 & \textrm{ if $\lfloor \log n/ \log 2 \rfloor < j$}, \\
   \end{cases} 
\end{eqnarray}
\begin{eqnarray}\label{efficient estimator2}
T_{n,jk}^{(2)} = \begin{cases} 
       \hat{f}_{n,jk} & \textrm{ if $j \leq j_0(n)$,} \\
      Y_{jk} 1_{ \{ \tilde{f}_{jk} \neq 0 \} } & \textrm{ if $ j_0(n) < j \leq \lfloor \log n/ \log 2 \rfloor $} .\\
      0 & \textrm{ if $\lfloor \log n/ \log 2 \rfloor < j$} .\\
   \end{cases} 
\end{eqnarray}

\begin{lemma}\label{efficient estimator lemma}
Consider the slab and spike prior $\Pi$ with lower threshold $j_0 (n) \rightarrow \infty$  satisfying $j_0 (n) = o(\log n)$ and let $(w_l)$ be any admissible sequence satisfying $w_{j_0(n)} = o(n^v)$ for any $v > 0$. Then the estimators $T_n^{(i)}$, $i=0,1$, defined in \eqref{efficient estimator1} and \eqref{efficient estimator2} satisfy for some $M'>0$ and any $M_n \rightarrow \infty$,
\begin{equation*}
\sup_{f_0 \in \mathcal{H}(\beta,R)} \P_0 ( \snorm{T_n^{(i)} - f_0}_{\mathcal{M}(w)} \geq M_n /\sqrt{n} ) \rightarrow 0  ,
\end{equation*}
\begin{equation*}
\sup_{f_0 \in \mathcal{H}(\beta,R)} \P_0 ( \snorm{T_n^{(i)} - f_0}_\infty \geq M' (\log n/n)^{\beta/(2\beta+1)} ) \rightarrow 0
\end{equation*}
as $n \rightarrow \infty$. Moreover, $\snorm{T_n^{(i)} - \mathbb{Y}}_{\mathcal{M}(w)} = o_{\P_0} (n^{-1/2})$, uniformly over $f_0 \in \mathcal{H}(\beta,R)$.
\end{lemma}

\begin{proof}
Since $\mathbb{Y}$ is an efficient estimator of $f_0$ in $\mathcal{M}$, if suffices to establish $\snorm{T_n^{(i)} - \mathbb{Y}}_\mathcal{M} = o_{\P_0}(n^{-1/2})$ to show that $T_n^{(i)}$ is also an efficient estimator of $f_0$ in $\mathcal{M}$. Consider firstly $j > j_0(n)$, where the estimators coincide, and let $J_n(\beta)$ be as in the proof of Theorem \ref{slab and spike contraction}. On the event $B_n$ defined in \eqref{eq19} and following the proof of Theorem \ref{slab and spike contraction}, we have
\begin{equation*}
\begin{split}
\snorm{\pi_{>j_0(n)} (T_n^{(i)} - \mathbb{Y} ) }_\mathcal{M} &  \leq \max_{j_0(n) < l \leq J_n(\beta)} w_l^{-1} \max_k |Y_{lk}1_{ \{ \tilde{f}_{lk} = 0 \} }| + \max_{l \geq J_n(\beta)} w_l^{-1} \max_k |f_{0,lk}|  \\
& \leq \max_{j_0(n) < l \leq J_n(\beta)} w_l^{-1} \max_{k:(l,k)\in \mathcal{J}_n^c (\overline{\gamma})} \left( |Y_{lk} - f_{0,lk}| + |f_{0,lk}| \right)   + o(n^{-1/2})  \\
& \leq C w_{j_0(n)}^{-1} \sqrt{(\log n)/n} + o(n^{-1/2}) = o(n^{-1/2})
\end{split}
\end{equation*}
by the choice of $j_0(n)$.

For $j \leq j_0(n)$ and $i=1$, we trivially have $\snorm{\pi_{j_0(n)} (T_n^{(1)} - \mathbb{Y})}_\mathcal{M} = 0$. Consider now $i=2$. Arguing as in Theorem 2 of \cite{CaNi2} and using the conditions on $j_0(n)$, one obtains the uniform bound $\E_0 \E^{\Pi} [ \snorm{\sqrt{n} \pi_{j_0(n)} (f-\mathbb{Y}) }_{\mathcal{M}(w)}^{1+\epsilon} \mid Y] \leq C(\epsilon)$ for $\epsilon >0$ small enough. From the weak convergence of $\Pi (\cdot \mid Y) \circ \tau_{\mathbb{Y}}^{-1}$ towards $\mathcal{N}$ and a uniform integrability argument (via the moment bound), it follows as in Theorem 10 of \cite{CaNi} that $\sqrt{n} \pi_{j_0(n)} (\E^{\Pi} (f | Y) - \mathbb{Y}) \rightarrow \E \mathcal{N} =0$ in $\mathcal{M}_0(w)$ in probability, which implies the result.

For $L^\infty$ we have on $B_n$,
\begin{equation*}
\begin{split}
\snorm{\pi_{>j_0(n)} (T_n^{(i)} - f_0 ) }_\infty &  \lesssim  \sum_{l=j_0(n)+1}^{J_n(\beta)} 2^{l/2} \max_k  \left( |Y_{lk}-f_{0,lk}|1_{ \{ (l,k) \in \mathcal{J}_n (\underline{\gamma}) \} } + |f_{0,lk}| 1_{ \{ (l,k) \in \mathcal{J}_n^c (\overline{\gamma}) \} } \right) \\
& \quad \quad + \sum_{l=J_n(\beta)+1}^\infty  2^{l/2} \max_k |f_{0,lk}|  \\
& \lesssim 2^{J_n(\beta)/2} \sqrt{\frac{\log n}{n}}  + R 2^{-J_n(\beta)\beta} \lesssim \left( \frac{\log n}{n} \right)^\frac{\beta}{2\beta+1}  .
\end{split}
\end{equation*}
Now
\begin{equation*}
\snorm{\pi_{j_0(n)}(T_n^{(1)} - f_0) }_\infty \lesssim \sum_{l=0}^{j_0(n)} 2^{l/2} \max_k |Y_{lk} - f_{0,lk}| \lesssim 2^{j_0(n)/2} \sqrt{\frac{\log n}{n}} \lesssim \left( \frac{\log n}{n} \right)^\frac{\beta}{2\beta+1} ,
\end{equation*}
thereby proving the second statement for $T_n^{(1)}$. For $T_n^{(2)}$, using the convergence of the posterior mean to $\mathbb{Y}$ in $\mathcal{M}_0$ for $j \leq j_0(n)$ shown above,
\begin{equation*}
\begin{split}
\snorm{\pi_{j_0(n)} (T_n^{(2)} - T_n^{(1)}) }_\infty & \lesssim \sum_{l=0}^{j_0(n)} 2^{l/2} \max_k |\hat{f}_{n,lk} - Y_{lk}| \\
& =  o_{\P_0} \left( \sum_{l=0}^{j_0(n)} 2^{l/2} \frac{w_l}{\sqrt{n}}\right) = o_{\P_0} \left( w_{j_0(n)} \frac{2^{j_0(n)/2}}{\sqrt{n}} \right) = o_{\P_0}  \left(  \left( \frac{\log n}{n} \right)^\frac{\beta}{2\beta+1} \right).
\end{split}
\end{equation*}
\end{proof}

\subsection{Wavelets}\label{wavelet section}

Let us briefly recall the notion of periodized and boundary corrected wavelets and discuss condition \eqref{wavelet}. Let $\phi$, $\psi$ denote a scaling and corresponding wavelet function on $\R$ satisfying
\begin{equation}
\sup_{x \in \mathbb{R}} \sum_{k \in \mathbb{Z}} |\phi(x-k)| < \infty, \quad \quad \sup_{x \in \mathbb{R}} \sum_{k \in \mathbb{Z}} |\psi(x-k)| < \infty.
\label{wavelet bound on R}
\end{equation}
Examples include Meyer wavelets (see Section 2 in \cite{Me} for other choices).

Consider firstly the periodic case. As usual define the dilated and translated wavelet at resolution level $j$ and scale position $k/2^j$ by $\phi_{jk}(x) = 2^{j/2} \phi (2^j x-k)$, $\psi_{jk}(x) = 2^{j/2} \psi (2^j x-k)$ for $j,k \in \mathbb{Z}$. Periodize the wavelet functions via
\begin{equation*}
\phi_{jk}^{per} (x) = \sum_{m \in \mathbb{Z}} \phi_{jk} (x+m), \quad  \quad \psi_{jk} ^{per} (x) = \sum_{m \in \mathbb{Z}} \psi_{jk} (x+m) , \quad x \in [0,1]
\end{equation*}
for $j = 0,1,...$ and $k = 0,...,2^j -1$. Then the wavelet system $\{ \phi_{J_0 k}^{per}, \psi_{jm}^{per} : k=0,...,2^{J_0}-1, m = 0,...,2^j-1, j= J_0,J_0+1,... \}$ forms an orthonormal wavelet basis of $L^2 ((0,1])$ and satisfies \eqref{wavelet} due to \eqref{wavelet bound on R}.

In the case of $\R$, an orthonormal basis of $V_j = \text{span} \{ (\phi_{jk})_k \}$, $j \geq J_0$, can be obtained by taking $2^{j-J_0}$ dilations of the orthonormal basis $( \phi_{J_0k} )_k$ of a basic resolution space $V_{J_0}$. In the case of boundary corrected wavelets, the analogous orthonormal basis of the basic resolution space $V_{J_0}$, $J_0 \in \mathbb{N}$, contains $2^{J_0}$ elements and consists of 3 components. At resolution level $J_0$, a basis consists of 3 components. Firstly, $N$ left edge functions $\phi_{J_0k}^{left}(x) = 2^{J_0/2} \phi_{k}^{left}(2^{J_0} x)$, $k=0,...,N-1$, where $\phi_k^{left}$ is a modification of $\phi$ that remains bounded and has compact support. Secondly, $N$ right edge functions $\phi_{J_0k}^{right}(x) = 2^{J_0/2} \phi_{k}^{right}(2^{J_0} x)$, $k=0,...,N-1$, with the same properties. Thirdly, $2^{J_0}-2N$ interior functions, that are the usual translates of dilations of $\phi$ defined on $\R$, that is $\phi_{J_0k}$ for $k = N,...,2^{J_0}-N-1$, which we note are all supported in the interior of $[0,1]$. Writing for convenience $\{ \phi_{J_0k}^{bc}: k = 0,...,2^{J_0}-1 \}$ instead of $\{ \phi_{J_0k}^{left} , \phi_{J_0k'}^{right}, \phi_{J_0m}: k=0,...,N-1, k'=0,...,N-1, m = N,...,2^{J_0} - N-1 \}$, we have the first part of \eqref{wavelet}
\begin{equation}
\begin{split}
\sum_{k=0}^{2^{J_0}-1} |\phi_{J_0k}^{bc}(x)| & \leq 2^{J_0/2} N \max_{0 \leq k < N} \snorm{\phi_k^{left}}_\infty  + 2^{J_0/2} N \max_{0 \leq k < N} \snorm{\phi_k^{right}}_\infty  + \sum_{k=N}^{2^{J_0}-N-1} 2^{J_0/2}|\phi(2^{J_0} x-k)|  \\
& \leq 2^{J_0/2} C(N,\phi) + 2^{J_0/2} C'(\phi),
\label{wavelet function bound}
\end{split}
\end{equation}
where we have used that $N$ is fixed and that the original wavelet function on $\R$ satisfies \eqref{wavelet bound on R}.

Starting at resolution level $J_0$ with the usual dilated wavelets on $\R$, $\psi_{J_0k} (x)= 2^{J_0/2} \psi (2^{J_0}x-k)$, $2^{J_0} \geq N$, it is possible to construct corresponding boundary wavelet functions
\begin{equation*}
\{ \psi_{J_0k}^{left}, \psi_{J_0k'}^{right}, \psi_{J_0m} : k=0,...,N-1, k'=0,...,N-1,m=N,...,2^{J_0}-N-1  \}  .
\end{equation*}
For $j \geq J_0$, we can then define the dilates of the boundary wavelets in the usual way:
\begin{equation*}
\psi_{jk}^{left}(x) = 2^{(j-J_0)/2} \psi_{J_0 k}^{left} (2^{j-J_0}x) , \quad \quad \psi_{jk}^{right}(x) = 2^{(j-J_0)/2} \psi_{J_0 k}^{right} (2^{j-J_0}x) .
\end{equation*}
This yields the required wavelets at resolution level $j$, namely $\{ \psi_{jk}^{left} , \psi_{jk'}^{right}, \psi_{jm}: k=0,...,N-1, k'=0,...,N-1, m = N,...,2^{j} - N-1 \}$, which for convenience we write as $\{ \psi_{jk}^{bc}: k = 0,...,2^{j}-1 \}$. Arguing as in \eqref{wavelet function bound} and again using \eqref{wavelet bound on R} gives the second part of \eqref{wavelet}.

\subsection{Weak convergence}\label{weak convergence section}

For $\mu$ and $\nu$ probability measures on a metric space $(S,d)$, define the bounded Lipschitz metric by
\begin{equation}
\beta_S (\mu , \nu) = \sup_{u: \norm{u}_{BL} \leq 1} \left| \int_S u(s) (d\mu (s) - d\nu(s) ) \right| ,
\label{bounded Lipschitz metric}
\end{equation}
\begin{equation*}
\norm{u}_{BL} = \sup_{s \in S} | u(s) | + \sup_{s,t \in S: s \neq t} \frac{| u(s) - u(t)| }{ d(s,t) }.
\end{equation*}
$\beta_S$ metrizes the weak convergence of probability distributions, that is random variables $X_n \rightarrow^d X$ converge in distribution in $(S,d)$ if and only if $\beta_S ( \mathcal{L}(X_n) , \mathcal{L}(X)) \rightarrow 0$, where $\mathcal{L}(X)$ denotes the law of $X$. In particular, we shall consider the choices $S = H(\delta) = H_2^{-1/2,\delta}$ or $S = H^{-s}$ for $s > 1/2$ in $\ell_2$ and $S = \mathcal{M}_0 (w)$ for $\{ w_l \}_{l \geq 1}$ an admissible sequence in $L^\infty$.

\subsection{Results on empirical and hierarchical Bayes procedures}\label{other paper section}

Let us recall some definitions and results from \cite{KnSzVVVZ,SzVVVZ} that appear in proofs elsewhere. Define $h_n : (0,\infty) \rightarrow [0,\infty)$ to be
\begin{equation}
h_n (\alpha) = \frac{1 + 2\alpha}{n^{1/(2\alpha + 1)}  \log n} \sum_{k=1}^\infty \frac{n^2 k^{2\alpha+1} f_{0,k}^2 \log k}{(k^{2\alpha+1} + n)^2}
\label{hn function}
\end{equation}
and for $0 < l < L$ define the bounds
\begin{equation*}
\underline{\alpha}_n = \inf \{ \alpha >0 : h_n (\alpha) > l \} \wedge \sqrt{\log n}  ,
\end{equation*}
\begin{equation*}
\overline{\alpha}_n = \inf \{ \alpha > 0 : h_n (\alpha) > L (\log n)^2 \} .
\end{equation*}

The behaviour of the empirical Bayes estimator $\hat{\alpha}_n$ defined in \eqref{likelihood} is contained in Lemma 3.11 of \cite{SzVVVZ}, which is summarized below for convenience.

\begin{lemma}[Szab\'o et al.]\label{smoothness lemma}
Fix $\beta_{max} > 0$. For any $0 < \beta \leq \beta_{max}$ and $R \geq 1$, there exist constants $K_1$ and $K_2$ such that $\P_0 ( \beta - K_1 / \log n \leq \hat{\alpha}_n \leq \beta + K_2 / \log n ) \rightarrow 1$ uniformly over $f_0 \in \mathcal{Q}_{SS} (\beta,R,\varepsilon)$.
\end{lemma}

As mentioned in the discussion following the lemma in \cite{SzVVVZ}, the constant $K_2$ is negative for large enough $R$ so that the estimate $\hat{\alpha}_n$ undersmooths the true $\beta$. We have an analogous result in the hierarchical case.

\begin{lemma}\label{posterior median lemma}
The posterior median $\alpha_n^M$ of the marginal posterior distribution $\lambda_n ( \cdot | Y)$ satisfies
\begin{equation*}
\inf_{f_0 \in \mathcal{Q} (\beta , R) } \P_0 \left( \alpha_n^M \in [\underline{\alpha}_n , \overline{\alpha}_n ]  \right) \rightarrow 1
\end{equation*}
as $n \rightarrow \infty$. Moreover, for $C = C(\beta,R,\varepsilon,\rho)$,
\begin{equation*}
\inf_{f_0 \in \mathcal{Q}_{SS}(\beta , R,\varepsilon) } \P_0 \left( \left| \alpha_n^M - \beta \right| \leq C / \log n  \right) \rightarrow 1 .
\end{equation*}
\end{lemma}

\begin{proof}
This follows directly from the proof of Theorem 3 of \cite{KnSzVVVZ}.
\end{proof}

\section*{Acknowledgements}
The author would like to thank Richard Nickl, Aad van der Vaart, Johannes Schmidt-Hieber, the Associate Editor and two referees for their valuable comments. The author would like to express particular thanks to one referee for a very detailed report, including suggesting a simplified argument for Theorems \ref{EB parametric BvM theorem} and \ref{parametric BvM theorem}.

\bibliography{Reference}{}
\bibliographystyle{acm}

\end{document}